\documentclass[11pt,twoside,a4paper,english,french,reqno]{amsart}
\usepackage{amsmath, amsthm, amsfonts}
\usepackage{hyperref} 
\hypersetup{hidelinks,backref=true,pagebackref=true,hyperindex=true,colorlinks=true,citecolor=red,linkcolor=blue,breaklinks=true,urlcolor=ocre,bookmarks=true,bookmarksopen=false,pdftitle={Title},pdfauthor={Author}}

\usepackage{amsmath,amsthm,amssymb}
\usepackage{amscd,indentfirst,epsfig}
\usepackage{latexsym}
\usepackage{times}
\usepackage{enumerate}
\usepackage{mathrsfs}
\usepackage{stmaryrd}
\usepackage{amsopn}
\usepackage{amsmath}
\usepackage{amssymb}
\usepackage{amsfonts,bm,dsfont}

\usepackage{amsbsy}
\usepackage{amscd}

at 11pt \font\sevenrsf=rsfs7 at 8pt \font\fiversf=rsfs5 at 6pt
\newfam\rsffam
\textfont\rsffam=\tenrsf \scriptfont\rsffam=\sevenrsf
\scriptscriptfont\rsffam=\fiversf

\newcommand{\verti}[1]{{\left\vert\kern-0.25ex\left\vert\kern-0.25ex\left\vert #1 
    \right\vert\kern-0.25ex\right\vert\kern-0.25ex\right\vert}}

\newtheorem{theo}{Theorem}[section]
\newtheorem{lemme}[theo]{Lemma}
\newtheorem{propo}[theo]{Proposition}
\newtheorem{cor}[theo]{Corollary}
\newtheorem{hyp}[theo]{Assumption}
\newtheorem{defi}[theo]{Definition}
\newtheorem{exe}[theo]{Example}

\newtheorem{nb}[theo]{Remark}

\def \bq {\begin{equation}}
\def \eq {\end{equation}}

\def \leq {\leqslant}
\def \geq {\geqslant}

\def \N {\mathbb{N}}
\def \ind {\mathbf{1}}

\def \S {\mathbb{S}}
\numberwithin{equation}{section}

\def\lp {L^1_+}
\def\lm {L^1_-}

\def \d {\mathrm{d}}
\def \D {\mathscr{D}}

\def \C {\mathbb{C}}
\def \Rs {\mathcal{R}}

\def \ml {M_{\lambda}}

\def \R {\mathbb{R}}

\def \M {\mathcal{M}}
\def \G {\bm{G}}

\def \X {\mathbb{X}}
\def \Y {\mathbb{Y}}
\def \l+ {L^1_+}
\def \l- {L^1_-}

\renewcommand{\epsilon}{\varepsilon}

\def \l {\lambda}
\def \T {\mathsf{T}}
\def \B {\mathsf{B}}
\def \e {\varepsilon}
\def \H {\mathsf{H}}

\begin{document}
\title[Diffuse boundary conditions]{Quantitative tauberian approach to collisionless transport equations with diffuse boundary operators}

 \author{B. Lods}

 \address{Universit\`{a} degli
 Studi di Torino \& Collegio Carlo Alberto, Department of Economics and Statistics, Corso Unione Sovietica, 218/bis, 10134 Torino, Italy.}\email{bertrand.lods@unito.it}

 \author{M. Mokhtar-Kharroubi}

 \address{Universit\'e de Bourgogne-Franche-Comt\'e, Equipe de Math\'ematiques, CNRS UMR 6623, 16, route de Gray, 25030 Besan\c con Cedex, France
}
\email{mustapha.mokhtar-kharroubi@univ-fcomte.fr}

\maketitle
\begin{abstract}
This paper gives a spectral approach to time asymptotics of collisionless transport semigroups with general diffuse boundary operators. The strong stability of the invariant density is derived from the classical Ingham theorem. A recent quantitative version of this theorem provides algebraic rates of convergence to equilibrium.



\end{abstract}

\tableofcontents
 \section{Introduction }

We consider here the time asymptotics for collisionless kinetic  equations of the form
\begin{subequations}\label{1}
\begin{equation}\label{1a}
\partial_{t}f(x,v,t) + v \cdot \nabla_{x}f(x,v,t)=0, \qquad (x,v) \in \Omega \times V, \qquad t \geq 0
\end{equation} 
with initial data
\begin{equation}\label{1c}f(x,v,0)=f_0(x,v), \qquad \qquad (x,v) \in \Omega \times V,\end{equation}
under \emph{diffuse  boundary}
\begin{equation}\label{1b}
f_{|\Gamma_-}=\mathsf{H}(f_{|\Gamma_+}),
\end{equation}\end{subequations}
where 
$$\Gamma _{\pm }=\left\{ (x,v)\in \partial \Omega \times V;\ \pm
v \cdot n(x)>0\right\}$$
($n(x)$ {being} the outward unit normal at $x\in \partial \Omega$)  and $\mathsf{H}$  {is
a linear 
boundary operator relating the
outgoing and incoming fluxes $f_{\mid \Gamma _{+}}$ and $f_{\mid
\Gamma _{-}}$ and is  bounded   on
the trace spaces
$$L^{1}_{\pm}=L^{1}(\Gamma_{\pm}\,;\,|v\cdot n(x)| \pi(\d  x)\otimes \bm{m}(\d v))=L^{1}(\Gamma_{\pm},\d\mu_{\pm}(x,v))$$
where $\pi$ denotes the Lebesgue surface measure on $\partial \Omega$ and $\bm{m}$ is a Borel measure on the set of velocities (see Assumptions \ref{hypO} hereafter). The
boundary operator }%
$$\H\::\:\lp \rightarrow \lm$$
is nonnegative and stochastic, i.e.
$$\int_{\Gamma_{-}}\H\psi\,\d\mu_{-}=\int_{\Gamma_{+}}\psi\,\d\mu_{+}, \qquad \forall \psi \in L^{1}(\Gamma_{+},\d\mu_{+})$$
 so that \eqref{1} is 
governed by a stochastic  $C_{0}$-semigroup $\left(
U_{\H}(t)\right) _{t\geq 0\text{ }}$ on $L^{1}(\Omega \times V\,,\,\d x\otimes %
\bm{m}(\d v))$ with generator $\T_{\H}.$ 

A general theory on the existence of an invariant density and its asymptotic
stability (i.e. convergence to equilibrium) has been published recently \cite{LMR} (see also earlier one-dimensional
results \cite{MKR}). The paper \cite{LMR} deals with general \emph{partly diffuse boundary
operators} $\H$ of the typical  
form
\begin{equation}\label{eq:Hintrod}\begin{split}
\mathsf{H}\varphi(x,v)&=\alpha (x)\varphi (x,v-(v\cdot n(x))n(x))\\
\phantom{+++} &+(1-\alpha
(x))\int_{v'\cdot n(x) >0}\bm{k}(x,v,v^{\prime })\varphi (x,v^{\prime })|v'\cdot n(x)|
\bm{m}(\d v^{\prime })\end{split}\end{equation}
where $\alpha\::\:x\in \partial\Omega \longmapsto \alpha (x)\in \left[ 0,1%
\right] $ is a measurable function and $\bm{k}$ is a nonnegative kernel.\medskip

\subsection{Our contribution in a nutshell}\label{sec:nut}

We consider here \textit{diffuse} boundary operators only for which, typically,
\begin{equation}\label{eq:Hhkernel}
\H\psi(x,v)=\int_{v'\cdot n(x) > 0}\bm{h}(x,v,v')\psi(x,v')\,|v'\cdot n(x)|\bm{m}(\d v'), \qquad (x,v) \in \Gamma_{-}\end{equation}
where,
$$\int_{v \cdot n(x) <0}\bm{h}(x,v,v^{\prime })|v\cdot n(x)|\bm{m}(\d v)=1,\ \quad (x,v')\in \Gamma
_{+}$$
We do not consider the case where the velocities are bounded away from
zero which deserves a separate analysis, mainly because in this case $\left( U_{\H}(t)\right) _{t\geq 0\text{ }}$ exhibits a
spectral gap and the convergence to equilibrium is exponential \cite{LM-iso}. When we allow arbitrarily small velocities
then the boundary of the spectrum of $\T_{\H}$ is the imaginary axis:
\begin{equation}\label{eq:iR}
i\mathbb{R} 
\subset \mathfrak{S}(\T_{\H})\end{equation}
(see Theorem \ref{theo:spectTH}) and the rate of convergence to equilibrium cannot be expected to be universal. Notice in particular that \eqref{eq:iR} implies that
\begin{equation}\label{eq:RR}
\lim_{\varepsilon \to 0^{+}}\left\|\Rs(\varepsilon+i\eta\,,\T_{\H})\right\|_{\mathscr{B}(L^{1}(\Omega\times V))}=\infty\end{equation}
where $\Rs(\l,\T_{\H})=(\l-\T_{\H})^{-1}$ denotes the resolvent of the transport operator $\T_{\H}$ and we denote simply 
$L^{1}(\Omega \times V)=L^{1}(\Omega \times V,\,\d x\otimes \bm{m}(\d v)).$
\medskip

The present paper continues \cite{LMR} by means of both qualitative and quantitative tauberian arguments addressing the following problems:
\begin{enumerate}[(P1)]
\item We show that asymptotic stability of the invariant density can
be derived from a classical Ingham tauberian theorem.
\item We provide rates of convergence to
equilibrium for solutions to \eqref{1} under mild assumptions on the initial datum $f_{0}$ by using recent quantified versions of Ingham's theorem \cite{chill}.
\end{enumerate}
\bigskip

Regarding (P1), asymptotic stability has been already investigated in our previous contribution \cite{LMR} using fine properties of partially integral stochastic semigroups
(see \cite[Appendix B]{LMR}). We give here a new and more general proof which   holds under weaker
assumptions. As far as (P2) is concerned, the use of a quantified tauberian result for collisionless
kinetic equations appears for the first time in \cite{mmkseifert}
for transport equations in slab geometry. The extension to multidimensional
geometries is much more involved and is the (main) object of the present paper.

As just said, the main ingredient for the above point  (P2) is a quantitative version of Ingham's Theorem recently obtained in \cite{chill}. We first recall the classical statement of Ingham's Theorem as stated in \cite{chill}:

\begin{theo}[Ingham]\label{theo:ingham}
Let $X$ be a Banach space and let $g \in \mathrm{BUC}(\R_{+};X)$. Assume there exists a function $F \in L^{1}_{\mathrm{loc}}(\R;X)$ such that
$$\lim_{\e\to0+}\int_{\R}\widehat{g}(\e+i\eta)\psi(\eta)\d s=\int_{\R}F(\eta)\psi(\eta)\d s \qquad \forall \psi \in \mathcal{C}_{c}(\R)$$
where $\widehat{g}(\l)=\int_{0}^{\infty}g(t)e^{-\l t}\d t$ denotes the Laplace transform of $g$, $\l=\e+i\eta \in \C_{+}$. Then $g \in \mathcal{C}_{0}(\R_{+},X)$.
\end{theo}
Here above, $\mathrm{BUC}(\R_{+};X)$ stands for the space of  bounded and  uniformly continuous functions from $\R_{+}$ to $X$ whereas $g \in \mathcal{C}_{0}(\R^{+};X)$ means that $\lim_{t\to\infty}g(t)=0.$ 

The \emph{quantitative version} of Theorem \ref{theo:ingham} we need has been derived in \cite{chill}. We extract from \cite{chill} the following \footnote{The statement here above corresponds to \cite[Theorem 2.1]{chill} for a constant  $M(R)=1$ $(R >0)$ so that $M_{k}(R)=(1+R)^{2/k}$, $R >0$.} where $\mathcal{C}_{b}(\R_{+};X)$ is the space of continuous and bounded functions whereas $\mathrm{Lip}(\R_{+};X)$ denotes the space of Lipschitz functions from $\R_{+}$ to $X$:
\begin{theo}\label{theo:quanti} Let $X$ be a Banach space and let $g \in \mathcal{C}_{b}(\R_{+};X) \cap \mathrm{Lip}(\R_{+};X)$. Suppose that
$\widehat{g}$ admits a boundary function $F \in L^{1}_{\mathrm{loc}}(\R;X)$ in the sense of the above Theorem \ref{theo:ingham}. Suppose that there is $k \geq 0$ such that $F \in \mathcal{C}^{k}(\R,X)$ and there is $C >0$ such that
\begin{equation}\label{eq:deri}
\left\|\frac{\d^{j}}{\d \eta^{j}}F(\eta)\right\|_{X} \leq C \qquad \forall \eta \in \R, \qquad 0 \leq j \leq k.\end{equation}
Then, 
\begin{equation}\label{eq:rateTheoric}
\|g(t)\|_{X}=\mathrm{O}(t^{-\frac{k}{2}}) \qquad \quad \text{ as $t \to \infty.$}\end{equation}
\end{theo}
Notice that, with the terminology of \cite{chill}, this theorem falls within the case where $F$ is \emph{nonsingular at zero}. In particular, the quantified tauberian theorem used here is \textit{different}
from the one used in \cite{mmkseifert} and improves even the one dimensional rates of
convergence given in \cite{mmkseifert}.

Our construction yielding to the answer of points (P1) and (P2) is quite
involved and relies on many new mathematical results of independent
interest. The main statements of the paper can be summarized as follows:
\begin{enumerate}[(a)]
\item If the boundary operator $\H$ is such that 
\begin{equation}\label{eq:H1m}
\H\in \mathscr{B}(L^{1}(\Gamma _{+},\mu _{+}(x,v)),\ L^{1}(\Gamma
_{-},|v|^{-1}\mu _{-}(x,v)))\end{equation}
then, for any $f \in L^{1}(\Omega \times V,)$ such that
$$|v|^{-1}f(x,v) \in L^{1}(\Omega \times V) \qquad \text{ and } \quad \int_{\Omega\times V}f(x,v)\d x\bm{m}(\d v)=0$$
the limit
$$\lim_{\varepsilon
\rightarrow 0^{+}}\Rs\left( \varepsilon +i\eta,\T_{\H}\right)f \;\;\text{
exists in } L^{1}(\Omega \times V)$$
and  the convergence is uniform in bounded $\eta$. We denote by $\mathbf{R}_{f}(\eta)$ this limit and refers to it as the \emph{boundary function} of  $\Rs(\l,\T_{\H})f$. The asymptotic stability
follows from Ingham's tauberian Theorem, see Theorem \ref{theo:qualit}. 
\item If the boundary operator $\H$ provides higher integrability than the mere \eqref{eq:H1m}, i.e. if 
\begin{equation}\label{eq:hk}
\H\in \mathscr{B}(L^{1}(\Gamma _{+},\d\mu _{+}(x,v)),\ L^{1}(\Gamma
_{-},\left\vert v\right\vert ^{-(k+1)}\d\mu _{-}(x,v)))\end{equation}
for some $k \in \N$, then for any $f \in L^{1}(\Omega \times V)$ such that 
\begin{equation}\label{eq:fk}
|v|^{-k-1}f(x,v) \in L^{1}(\Omega \times V) \qquad  \text{ and } \quad \int_{\Omega\times V}f(x,v)\d x\bm{m}(\d v)=0\end{equation}
the boundary function 
$$\eta \in \R \longmapsto \mathbf{R}_{f}(\eta) \in L^{1}(\Omega \times V)$$
is of class $\mathcal{C}^{k}$ and is uniformly bounded on $\R$ as well as its derivatives. One then can deduce from the \emph{quantified Ingham's Theorem} \ref{theo:quanti} that, in this case, the rate of convergence to equilibrium is 
$$\mathrm{O}\left(t^{-\frac{k}{2}}\right), \qquad \text{ as } t \to\infty.$$
We refer to Theorem \ref{theo:main} at the end of this Introduction for a precise statement.
\end{enumerate}

It is remarkable here that the rate of convergence (for some class of initial data) depends \emph{heavily} on the gain of integrability provided by the boundary operator in \eqref{eq:hk}. For instance, if the kernel
$$\bm{h}(x,\cdot,v') \emph{ is compactly supported away from zero} \qquad \forall (x,v') \in \Gamma_{+}$$ 
then one can take any $k >0$ in \eqref{eq:hk}. Another fundamental example is the one for which
$$\bm{h}(x,v,v')=\M_{\theta}(v), \qquad (x,v') \in \Gamma_{+},\,\,v \in V=\R^{d}$$
with 
\begin{equation}\label{eq:maxw}
\M_{\theta}(v)=(2\pi\theta)^{-d/2}\exp\left(-\frac{|v|^{2}}{2}\right), \qquad v \in \R^{d}, \qquad \theta >0\end{equation}
and if $\bm{m}(\d v)=\d v$ is the Lebesgue measure on $\R^{d}$, then \eqref{eq:hk} holds for $k < d $ and from this upper bound on $k$ we can derive the faster possible rate of convergence. 
 
The proofs of the above points (a) and (b) are quite technical and are derived
from a series of various mathematical results of independent interest; see
below.

\subsection{Related literature} Besides its fundamental role in the study of the Boltzmann equation with boundary conditions \cite{EGKM,guo03,briant}, the mathematical interest towards relaxation to equilibrium for collisionless equations is
relatively recent in kinetic theory starting maybe with numerical evidences obtained in \cite{aoki1}. A precise description of the relevance of the question as well as very interesting results have been obtained then in \cite{aoki}. We mention also the important contributions \cite{kuo,liu} which obtain optimal rate of convergence when the spatial domain is a ball. The two very recent works \cite{bernou1,bernou2} provide (optimal) convergence rate for general domains $\Omega$. All these works are dealing with partial diffuse boundary operator of Maxwell-type for which
\begin{multline}\label{eq:maxweBC}
\H\varphi(x,v)=\alpha (x)\varphi (x,v-(v\cdot n(x))n(x)) \\
+\frac{(1-\alpha
(x))}{\gamma(x)}\M_{\theta(x)}(v)\int_{v'\cdot n(x) >0}\varphi (x,v^{\prime })|v'\cdot n(x)|
\bm{m}(\d v^{\prime })\end{multline}
where, as above $\M_{\theta(x)}$ is a Maxwellian distribution given by \eqref{eq:maxw} for which the temperature $\theta(x)$ depends (continuously) on $x \in \partial\Omega$ and $\gamma(x)$ is a normalization factor ensuring $\H$ to be stochastic. Optimal rate of convergence for the boundary condition \eqref{eq:maxweBC} in dimension $d=2,3$ has been obtained recently in \cite{bernou1} thanks to a clever use of Harris's subgeometrical convergence theorem for Markov processes. A related probabilistic approach, based on coupling, has been addressed in \cite{bernou2} in dimension $d\geq2$ whenever $\theta(x)=\theta$ is constant. As mentioned, in \cite{bernou1}, the obtained rate of convergence is optimal and given by
$$\mathrm{O}\left(\log(1+t)^{d+1}t^{-(d)}\right) \quad \text{ as } t \to\infty.$$
Recall that, here $\H$ satisfies \eqref{eq:hk} with $k < d$ which confirms the fact that the optimal rate of convergence is prescribed by the gain of integrability $\H$ is able to provide. 

Of course the above rate $\mathrm{O}(\log(1+t)^{d+1}t^{-d})$ is much better than the one $\mathrm{O}(t^{-\frac{d}{2}})$ we can reach in our case  but our paper is not really 
comparable to \cite{bernou1,bernou2}. First, we deal here with different kind of boundary conditions (in any dimension $d \geq 1$) and, even if we restrict ourselves to diffuse boundary condition, the structure of the kernel $\bm{h}(x,v,v')$ is much more general than the Maxwellian case \eqref{eq:maxweBC}. Second,  the mathematical tools and
results are completely different. Finally, as we noted, our rates of
convergence are consequences of various results of independent interest for
kinetic theory; see below. It is clear however that the better rates obtained in \cite{bernou1}
suggest strongly that our rates are certainly not optimal. We have the feeling though that the results are the best one can derive from the quantitative Ingham tauberian Theorem \ref{theo:quanti}. In particular, any improvement of the theoretical results in Theorem \ref{theo:quanti} would lead to an improvement of the rate we obtain here. For instance, only integer values of the derivatives are allowed in Theorem \ref{theo:quanti} and one may wonder if there is room for some improvement in the rate \eqref{eq:rateTheoric} if one allows to consider \emph{fractional derivatives} in \eqref{eq:deri}.

\subsection{Mathematical framework}\label{sec:frame}

Let us describe more precisely our mathematical framework and the set of assumptions we adopt throughout the paper. First, the general assumptions on the phase space are  the following

\begin{hyp}\label{hypO} The phase space $\Omega \times V$ is such that
\begin{enumerate} 
\item $\Omega
\subset \R^{d}$ $(d\geq 2)$ is an open and \emph{bounded} subset with $\mathcal{C}^{1}$ boundary $\partial \Omega 
$.
\item  $V$ is the support of a nonnegative Borel measure $\bm{m}$  which is orthogonally invariant (i.e. invariant under the action of the orthogonal group of matrices in $\R^{d}$).
\item $0 \in V$, $\bm{m}(\{0\})=0$ and $\bm{m}\left(V \cap B(0,r)\right) >0$ for any $r >0$ where $B(0,r)=\{v \in \R^{d}\,,\,|v| < r\}.$
\end{enumerate}
We denote by 
$$\X_{0}:=L^{1}(\Omega \times V\,,\,\d x\otimes \bm{m}(\d v))$$
endowed with its usual norm $\|\cdot\|_{\X_{0}}.$ More generally, for any $k \in \N$, we set
$$\X_{k}:=L^{1}(\Omega \times V\,,\,\max(1,|v|^{-k})\d x \otimes \bm{m}(\d v))$$
with norm $\|\cdot\|_{\X_{k}}.$
\end{hyp}
Notice that the above Assumption \textit{(3)} is necessary to ensure that the transport operator with no-incoming boundary condition has at least the whole imaginary axis in its spectrum.\medskip

With respect to our previous contribution \cite{LMR}, as already mentioned, we do not consider abstract and general boundary operator here but focus our attention on the specific case of diffuse boundary operator satisfying the following Assumption where we define
\begin{equation*}
\begin{cases}
\mathsf{M}_{0} \::\:&L^1_- \longrightarrow L^1_+\\
&u \longmapsto
\mathsf{M}_{0}u(x,v)=u(x-\tau_{-}(x,v)v,v),\:\:\:(x,v) \in \Gamma_+\;;
\end{cases}
\end{equation*}
\begin{hyp}\label{hypH}
The boundary operator $\H\::\lp \to \lm$ is a bounded and stochastic operator which satisfies the following
\begin{enumerate}[1)]
\item There exists some $n \in \N$ such that
$$\H \in \mathscr{B}(\lp,\Y_{n+1}^{-})$$
where
$$\Y_{k}^{\pm}:=\{g \in L^{1}_{\pm}\;;\;\int_{\Gamma_{\pm}}\max(1,|v|^{-k})|g(x,v)|\d\mu_{+}(x,v) < \infty, \qquad k \in \N.$$
We will set
\begin{equation}\label{eq:HYn}
N_{\H}:=\sup\{k \in \N\;;\;\H \in \mathscr{B}(\lp,\Y_{k+1}^{-})\}.\end{equation}
\item The operator $\H\mathsf{M}_{0}\H \in \mathscr{B}(\lp)$
is weakly compact. 
\item $\mathsf{M}_{0}\H$ is irreducible.
\item There exists $\ell \in \N$ such that
\begin{equation}\label{eq:poweriml}
\lim_{|\eta|\to\infty}\left\|\left(\mathsf{M}_{\varepsilon+i\eta}\H\right)^{\ell}\right\|_{\mathscr{B}(\lp)}=0 \qquad \forall \varepsilon \geq 0.\end{equation}
\end{enumerate}
\end{hyp}

A few remarks are in order about our set of Assumptions:
\begin{itemize}
\item First, we gave in our previous contribution \cite[Theorem 5.1]{LMR} precise definition of a general class of boundary operator for which $\H\mathsf{M}_{0}\H$ is weakly-compact. This class of operators was defined in \cite{LMR} as the class of \emph{regular diffuse boundary} operators and we will simply say here that $\H$ is diffuse. We refer to Appendix \ref{appen:REGU} for details.
\item Moreover, practical criterion ensuring the above property \textit{3)} to occur are also given in \cite{LMR}. In practice, as observed already, the typical operator we have in mind are given by \eqref{eq:Hhkernel}.
Under some  \emph{strong positivity} assumption on $\bm{h}(\cdot,\cdot,\cdot)$, one can show that $\mathsf{M}_{0}\H$ is irreducible (see \cite[Section 4]{LMR}).
\item We believe that Assumption \textit{4)} is met for \emph{any} regular diffuse boundary operators. We have been able to prove the result with $\ell = 2$ for a slightly more restrictive class of boundary operators (see Theorem \ref{ass:4} in Appendix \ref{appen:REGU}).
\end{itemize}

 We refer to Appendix \ref{appen:REGU} for more details on this set of assumptions but we notice already that the assumptions are met for the following examples of diffuse boundary operators of physical relevance.

\begin{exe}\label{exe:gener} The most typical example corresponds to generalized Maxwell-type diffuse operator for which
$$\bm{h}(x,v,v')=\gamma^{-1}(x)\bm{G}(x,v)$$
where $\G\::\:\partial \Omega \times V \to \R^{+}$ is a measurable and nonnegative mapping such that 
\begin{enumerate}[($i$)]
\item $\G(x,\cdot)$ is radially symmetric for $\pi$-almost every $x \in \partial \Omega$; 
\item $\G(\cdot,v) \in L^{\infty}(\partial \Omega)$ for almost every $v \in V$;
\item The mapping $x \in \partial \Omega \mapsto \gamma(x)$ is continuous and \emph{bounded away from zero} where
\begin{equation}\label{eq:gamm}
\gamma(x):=\int_{\Gamma_{-}(x)}\G(x,v)|v \cdot n(x)| \bm{m}(\d v) \qquad \forall x \in \partial\Omega,\end{equation}
i.e. there exist $\gamma_{0} >0$ such that $\gamma(x) \geq \gamma_{0}$ for $\pi$-almost every $x \in \partial\Omega.$
\end{enumerate}
In that case, \eqref{eq:HYn} is satisfied for $n \in \N$ such that $\gamma(n,d) < \infty$ where, for all $s \geq0$,
$$\gamma(s,d):=\sup_{x \in \partial\Omega}\gamma^{-1}(x)\int_{v \cdot n(x)<0}|v|^{-s-1}\G(x,v)|v \cdot n(x)|\bm{m}(\d v)\in (0,\infty].$$
\end{exe}
\begin{exe}\label{exe:maxw} A particularly relevant example is a special case of the previous one for which, $\bm{m}(\d v)=\d v$ and $\G$ is a given maxwellian with temperature $\theta(x)$, i.e.
$$\G(x,v)=\mathcal{M}_{\theta(x)}(v), \qquad \mathcal{M}_{\theta}(v)=(2\pi\theta)^{-d/2}\exp\left(-\frac{|v|^{2}}{2\theta}\right), \qquad x \in \partial \Omega, \:\:v \in \R^{d}.$$ 
Then, 
$$\gamma(x)=\bm{\kappa}_{d}\sqrt{\theta(x)}\int_{\R^{d}}|w|\M_{1}(w)\d w, \qquad x \in \partial\Omega$$ 
for some positive constant $\bm{\kappa}_{d}$ depending only on the dimension. The above assumption \textit{$(iii)$} asserts that the temperature mapping $x \in \partial\Omega \mapsto \theta(x)$ is bounded away from zero and continuous.
\end{exe}

Notice that under Assumption \ref{hypH}, one can deduce directly the following from  \cite[Theorem 6.5]{LMR}:
\begin{theo}\label{theo:LMR} Under Assumption \ref{hypH}, the operator $(\T_{\H},\D(\T_{\H}))$ defined by
\begin{multline*}
\D(\T_{\H})=\left\{f \in \X_{0}\;;\;\,v
\cdot \nabla_x \psi \in \X_{0}\;;\;f_{\vert\Gamma_{\pm}} \in L^{1}_{\pm}\,\;\;\,\H\,f_{\vert \Gamma_{+}}=f_{\vert\Gamma_{-}}\right\},\\ 
\T_{\H}f=-v\cdot \nabla_{x}f, \qquad f \in \D(\T_{\H})\end{multline*}
 is the generator of a \emph{stochastic} $C_{0}$-semigroup  $(U_{\mathsf{H}}(t))_{t\geq 0}$. Moreover, $(U_{\mathsf{H}}(t))_{t\geq 0}$ is irreducible and has a unique invariant density ${\Psi}_{\mathsf{H}} \in \D(\mathsf{T}_{\mathsf{H}})$ with 
$${\Psi}_{\mathsf{H}}(x,v) >0 \qquad \text{ for a. e. } (x,v) \in \Omega \times \R^{d}, \qquad \|{\Psi}_{\mathsf{H}}\|_{\X_{0}}=1$$
and $\mathrm{Ker}(\mathsf{T}_{\mathsf{H}})=\mathrm{Span}({\Psi}_{\mathsf{H}}).$ Moreover, 
\begin{equation}\label{eq:ergodic}
\lim_{t \to \infty}\left\|\frac{1}{t}\int_{0}^{t}U_{\mathsf{H}}(s)f\d s-\mathbb{P}f\right\|_{\X_{0}}=0, \qquad  \forall f \in \X_{0}\end{equation}
 where $\mathbb{P}$ denotes the ergodic projection
$$\mathbb{P}f=\varrho_{f}\,{\Psi}_{\mathsf{H}}, \qquad \text{ with } \quad\varrho_{f}=\displaystyle\int_{\Omega\times\R^{d}}f(x,v)\d x \otimes \bm{m}(\d v), \qquad f \in \X_{0}.$$ 
\end{theo}

 \subsection{Main results and method of proof} 
 We describe more in details here the main results of the paper. As mentioned earlier, we obtain two kinds of results addressing the two problems (P1) and (P2) of Section \ref{sec:nut}. First, as far as the qualitative convergence to equilibrium is concerned, our main result is the following:
\begin{theo}\label{theo:qualit}
Under Assumptions \ref{hypH}, for any $f \in \X_{0}$, one has
$$\lim_{t\to\infty}\left\|U_{\H}(t)f-\varrho_{f}\Psi_{\H}\right\|_{\X_{0}}=0.$$
\end{theo}
As already mentioned, this stability result not only strengthen the ergodic convergence of Theorem \ref{theo:LMR} but also extend \cite[Theorem 7.5]{LMR} to general orthogonally invariant measure $\bm{m}$ with $\bm{m}(\{0\})=0$ (recall that \cite[Theorem 7.5]{LMR} was restricted to the Lebesgue measure $\bm{m}(\d v)=\d v$). Note that if Assumptions \ref{hypH} \textit{1)} is not satisfied then the invariant density need not exist; in this case, a sweeping phenomenon occurs, i.e. the total mass of the solution to the Cauchy problem concentrates near the zero velocity as $t \to \infty$ (see \cite[Theorem 8.5]{LMR}).

Regarding problem (P2), we can make the above convergence quantitative under additional assumption on the initial datum. Namely, our main result can be formulated as follows
\begin{theo}\label{theo:main}
Assume that $\H$ satisfies Assumptions \ref{hypH}. Let $k \in \N$ with $k < N_{\H}$ and 
$$f \in \D(\T_{\H}) \cap \X_{k+1}$$ be given. Then, there exists $C_{k} >0$ such that
$$\left\|U_{\H}(t)f-\varrho_{f}\Psi_{\H}\right\|_{\X_{0}} \leq C_{k}\,t^{-\frac{k}{2}}, \qquad \forall t >0.$$
In particular, 
\begin{enumerate}
\item If $N_{\H} <\infty$ and $f \in \D(\T_{\H}) \cap \X_{N_{\H}+1}$, then
$$\left\|U_{\H}(t)f-\varrho_{f}\Psi_{\H}\right\|_{\X_{0}}=\mathrm{O}(t^{-\frac{N_{\H}-1}{2}}) \qquad \text{ as } t \to \infty.$$
\item If $N_{\H}=\infty$, then for any $k \in \N$ and any $f \in \D(\T_{\H}) \cap \X_{k+1}$, it holds
$$\left\|U_{\H}(t)f-\varrho_{f}\Psi_{\H}\right\|_{\X_{0}}=\mathrm{O}(t^{-\frac{k}{2}}) \qquad \text{ as } t \to \infty.$$
\end{enumerate}
\end{theo}

Besides these two Theorems, the paper contains many technical results. For the sake of clarity and in
order to help the reading of the paper, we give here an idea of the main
steps of the proofs of the above two results. The precise definition of the involved objects is given in Section \ref{sec:functional}. The main mathematical object we have to study is the resolvent of $\T_{\H}$ which can written as 
\begin{equation}\label{eq:RTH}
\Rs(\lambda,\T_{\H})=\Rs(\lambda,\T_{0})+\mathsf{\Xi}_{\lambda }\H
\Rs(1,\mathsf{M}_{\lambda }\H)\mathsf{G}_{\lambda} \qquad \qquad \mathrm{Re}\lambda >0,\end{equation}
where $\T_{0}$ is the transport operator corresponding to $\H=0$
and $\mathsf{\Xi}_{\lambda },\mathsf{M}_{\lambda },\mathsf{G}_{\lambda }$ are bounded operators on suitable trace spaces (see Section \ref{sec:func}).
Note that 
$$r_{\sigma }\left(\mathsf{M}_{\lambda}\H\right) <1 \qquad \text{Re}\l >0.$$%
To apply Theorem \ref{theo:quanti}, the main issue is to understand for which $f\in L^{1}(\Omega \times V)$, the 
\textit{boundary function}
$$\mathbf{R}_{f}(\eta):=\lim_{\e\to0^{+}}\Rs(\e+i\eta,\T_{\H})f$$
is well-defined in $\X_{0}$, is a smooth vector-valued function of the parameter $\eta$ and is bounded as well as its derivatives. The simplest part of this program concerns the transport operator $\T_{0}$ and
$$\mathbf{R}_{f}^{0}(\eta):=\lim_{\e\to0^{+}}\Rs(\e+i\eta,\T_{0})f$$
exists in $\X_{0}$ with the mapping $\eta \in \R \mapsto \mathbf{R}_{f}^{0}(\eta)$ of class $\mathcal{C}^{k}$ and bounded as well as its derivatives
provided that 
$$f \in \X_{k+1}$$ 
(see Lemma \ref{lem:To}). The most tricky part is the understanding of the second part of the splitting \eqref{eq:RTH}
$$
\lim_{\varepsilon \rightarrow 0_{+}}\mathrm{R}(\varepsilon +i\eta )f, \qquad \text{ where } \quad
\mathrm{R}(\lambda ):=\mathsf{\Xi _{\lambda }H}\Rs(1,\mathsf{M_{\lambda
}H)G}_{\lambda } \qquad (\text{Re}\lambda >0).
$$
It turns out that $\lambda \mapsto \mathsf{M_{\lambda }H} \in \mathscr{B}(\lp)$ extends to the
imaginary axis with
$$r_{\sigma }\left(\mathsf{M_{i\eta }H}\right) <1 \qquad (\eta \neq 0);$$
(see Proposition \ref{prop:Meps}). It follows that 
$$
\lim_{\varepsilon \rightarrow 0_{+}}\mathsf{\Xi _{\varepsilon +i\eta }H}\Rs
(1,\mathsf{M_{\varepsilon +i\eta }H)G}_{\varepsilon +i\eta }f\text{ \ exists}\ \
(\eta \neq 0)
$$
and the convergence is locally uniform in $\eta \neq 0$ (see Lemma \ref{lem:converTH}).
Because $r_{\sigma }\left(\mathsf{M_{0}H}\right) =1$, the treatment of the case $%
\eta =0$  is very involved and the various technical results of Section \ref{sec:near0}  are
devoted to this delicate point. In particular, by exploiting the fact that
near $\lambda =0$, the eigenvalue of $\mathsf{M_{\lambda }H}$ of maximum modulus is
algebraically simple (converging to $1$ as $\lambda \rightarrow 0$, see Proposition \ref{prop:eigenMLH}), and
analyzing the corresponding spectral projection, we get the desired result
under the additional assumption that
$$\int_{\Omega\times V}f(x,v)\d x\,\bm{m}(\d v)=0.$$
(see Lemma \ref{lem:convZk}). All these results allow to show that the boundary function
$$\mathbb{R}
\ni \eta \longmapsto \mathbf{R}_{f}(\eta)\in \X_{0}\text{ is continuous}$$
and yields the proof of Theorem \ref{theo:qualit}. To deal with rates of convergence, we need \textit{first} to analyze the
smoothness of the boundary function. This technical point is related, through the well-know identity for derivatives of the resolvent
$$\dfrac{\d^{k}}{\d\l^{k}}\Rs(\l,\T_{\H})=(-1)^{k}\,k!\,\Rs(\l,\T_{H})^{k+1}, \qquad \mathrm{Re}\l >0$$ to the
existence of a boundary function for the \emph{iterates} of $\Rs(\lambda
,\T_{\H})$%
$$
\lim_{\varepsilon \rightarrow 0_{+}}\left[\Rs(\varepsilon +i\eta,
    \T_{\H})\right]^{k+1}f\text{ \  in  } \X_{0}
$$
(locally uniformly in $\eta $), see Theorem \ref{theo:regul}. 
The \textit{second} technical point is to show that the boundary
function and its derivatives are bounded on $\R$. This point is the crucial one where Assumption \ref{hypH} \textit{(4)} is fully exploited. 

As mentioned earlier, proving that Assumption \ref{hypH} \textit{4)} is met for a large class of diffuse boundary operators is a highly technical task and we devote Appendix \ref{appen:REGU} to this (see Theorem \ref{ass:4}).  The results of Appendix \ref{appen:REGU} are also related to a general change of variable formula transferring integrals in velocities into integrals over $\partial\Omega$. This change of variable formula is established in Appendix \ref{appen:chv} and, besides its use in Theorem \ref{ass:4}, has its own interest: it clarifies several computations scattered in the literature \cite{EGKM,guo03} and will be a fundamental tool for the analysis in the companion paper \cite{LM-iso}.

\subsection{Organization of the paper} In Section \ref{sec:func}, we introduce the functional setting and notations used in the rest of the paper and recall several known results mainly from our previous contribution \cite{LMR}. Section \ref{sec:fineR} is devoted to the fine analysis of the resolvent $\Rs(1,\mathsf{M}_{\l}\H)$ which is well-defined for $\mathrm{Re}\l >0$ but need to be carefully extended to the imaginary axis $\l=i\eta$, $\eta \in\R.$ Such an extension is a cornerstone in the construction of the boundary function $\lim_{\e \to 0^{+}}\Rs(\e+i\eta,\T_{\H})f$ (for suitable $f$) which is performed in Section \ref{sec:boudF}. This Section is the most technical one of the paper and we have to deal separately with the case $\eta \neq 0$ and $\eta=0.$ Section \ref{sec:regF} deals with the regularity of the boundary function and gives the full proof of our main results Theorems \ref{theo:qualit} and \ref{theo:main}.

The paper ends with two Appendices. A first one, Appendix \ref{appen:REGU} is aimed to provide practical criteria ensuring Assumptions \ref{hypH} to be met. It contains several results we believe to be of independent interest. Some of the results in Appendix \ref{appen:REGU} are derived thanks to a general change of variable formula which is established in the second Appendix \ref{appen:chv}.

\subsection*{Acknowledgements} BL gratefully acknowledges the financial
support from the Italian Ministry of Education, University and
Research (MIUR), ``Dipartimenti di Eccellenza'' grant 2018-2022. Part of this research was performed while the second author was visiting the ``Laboratoire de Math\'ematiques CNRS UMR 6623'' at Universit\'e de  Franche-Comt\'e in February 2020. He wishes to express his gratitude for the financial support and warm hospitality offered by this Institution.

 \section{Reminders of known results}\label{sec:func}

\subsection{Functional setting}\label{sec:functional}
We introduce the partial Sobolev space 
$$W_1=\{\psi \in \X_{0}\,;\,v
\cdot \nabla_x \psi \in \X_{0}\}.$$ It is known \cite{ces1,ces2}
that any $\psi \in W_1$ admits traces $\psi_{|\Gamma_{\pm}}$ on
$\Gamma_{\pm}$ such that 
$$\psi_{|\Gamma_{\pm}} \in
L^1_{\mathrm{loc}}(\Gamma_{\pm}\,;\,\d \mu_{\pm}(x,v))$$ where
$$\d \mu_{\pm}(x,v)=|v \cdot n(x)|\pi(\d x) \otimes \bm{m}(\d v),$$
denotes the "natural" measure on $\Gamma_{\pm}.$ Notice that, since $\d\mu_{+}$ and $\d\mu_{-}$ share the same expression, we will often simply denote it by 
$$\d \mu(x,v)=|v \cdot n(x)|\pi(\d x) \otimes \bm{m}(\d v),$$
the fact that it acts on $\Gamma_{-}$ or $\Gamma_{+}$ being clear from the context.  Note that
$$\partial
\Omega \times V:=\Gamma_- \cup \Gamma_+ \cup \Gamma_0,$$ where
$$\Gamma_0:=\{(x,v) \in \partial \Omega \times V\,;\,v \cdot
n(x)=0\}.$$
We introduce the set
$$W=\left\{\psi \in W_1\,;\,\psi_{|\Gamma_{\pm}} \in L^1_{{\pm}}\right\}.$$
One can show \cite{ces1,ces2} that $W=\left\{\psi \in
W_1\,;\,\psi_{|\Gamma_+} \in \lp\right\} =\left\{\psi \in
W_1\,;\,\psi_{|\Gamma_-} \in \lm\right\}.$ Then, the \textit{trace
operators} $\mathsf{B}^{\pm}$:
\begin{equation*}\begin{cases}
\mathsf{B}^{\pm}: \:&W_1 \subset \X_{0} \to L^1_{\mathrm{loc}}(\Gamma_{\pm}\,;\,\d \mu_{\pm})\\
&\psi \longmapsto \mathsf{B}^{\pm}\psi=\psi_{|\Gamma_{\pm}},
\end{cases}\end{equation*}
are such that $\mathsf{B}^{\pm}(W)\subseteq L^1_{\pm}$. Let us define the
{\it maximal transport operator } $\mathsf{T}_{\mathrm{max}}$  as follows:
\begin{equation*}\begin{cases} \mathsf{T}_{\mathrm{max}} :\:& \D(\mathsf{T}_{\mathrm{max}}) \subset \X_{0} \to \X_{0}\\
&\psi \mapsto \mathsf{T}_{\mathrm{max}}\psi(x,v)=-v \cdot \nabla_x
\psi(x,v),
\end{cases}\end{equation*}
with domain $\D(\mathsf{T}_{\mathrm{max}})=W_1.$
Now, for any \textit{
bounded boundary operator} $\mathsf{H} \in\mathscr{B}(L^1_+,L^1_-)$, define
$\mathsf{T}_{\mathsf{H}}$ as
$$\mathsf{T}_{\mathsf{H}}\varphi=\mathsf{T}_{\mathrm{max}}\varphi \qquad \text{ for any }
\varphi \in \D(\mathsf{T}_{\mathsf{H}}),$$ where 
$$\D(\mathsf{T}_{\mathsf{H}})=\{\psi \in
W\,;\,\psi_{|\Gamma_-}=\mathsf{H}(\psi_{|\Gamma_+})\}.$$ In particular, the
transport operator with absorbing conditions (i.e. corresponding
to $\mathsf{H}=0$) will be denoted by $\T_0$.

\subsection{Travel time and integration formula} 
Let us now introduce the \textit{travel time} of particles in $\Omega$ (with the notations of \cite{mjm1}),
defined as:
\begin{defi}\label{tempsdevol}
For any $(x,v) \in \overline{\Omega} \times V,$ define
\begin{equation*}
t_{\pm}(x,v)=\inf\{\,s > 0\,;\,x\pm sv \notin \Omega\}.
\end{equation*}
To avoid confusion, we will set $\tau_{\pm}(x,v):=t_{\pm}(x,v)$  if $(x,v) \in
\partial \Omega \times V.$
\end{defi}
 
With the notations of \cite{guo03}, $t_{-}$ is the \emph{backward exit time} $t_{\mathbf{b}}$.
From a heuristic viewpoint, $t_{-}(x,v)$ is the time needed by a
particle having the position $x \in \Omega$ and the velocity $-v
\in V$ to reach the boundary $\partial\Omega$. One can prove
\cite[Lemma 1.5]{voigt} that $t_{\pm}(\cdot,\cdot)$ is measurable on
$\Omega \times V$. Moreover $\tau_{\pm}(x,v)=0 \text{ for any } (x,v)
\in \Gamma_{\pm}$ whereas $\tau_{\mp}(x,v)> 0$ on $\Gamma_{\pm}.$ It holds
$$(x,v) \in \Gamma_{\pm} \Longleftrightarrow \exists y \in \Omega \quad \text{ with } \quad t_{\pm}(y,v) < \infty \quad \text{ and }\quad x=y\pm t_{\pm}(y,v)v.$$
In that case, $\tau_{\mp}(x,v)=t_{+}(y,v)+t_{-}(y,v).$  Notice also that,
\begin{equation}\label{eq:scale}
t_{\pm}(x,v)|v|=t_{\pm}\left(x,\omega\right), \qquad \forall (x,v) \in \overline{\Omega} \times V, \:v \neq 0, \:\omega=|v|^{-1}\,v \in \mathbb{S}^{d-1}.\end{equation}

We have the following integration formulae from \cite{mjm1}.
\begin{propo} For any $h \in \X_{0}$, it holds
\begin{equation}\label{10.47}
\int_{\Omega \times V}h(x,v)\d x \otimes \bm{m}(\d v)
=\int_{\Gamma_\pm}\d\mu_{\pm}(z,v)\int_0^{\tau_{\mp}(z,v)}h\left(z\mp\,sv,v\right)\d s,
\end{equation}
and for any $\psi \in L^1(\Gamma_-,\d\mu_{-})$,
\begin{equation}\label{10.51}
\int_{\Gamma_-}\psi(z,v)\d\mu_{-}(z,v)=\int_{\Gamma_+}\psi(x-\tau_{-}(x,v)v,v)\d\mu_{+}(x,v).\end{equation}
\end{propo}
\begin{nb} Notice that, because $\mu_{-}(\Gamma_{0})=\mu_{+}(\Gamma_{0})=0$, we can extend the above identity \eqref{10.51} as follows: for any $\psi \in L^{1}(\Gamma_{-} \cup \Gamma_{0},\d\mu_{-})$ it holds
\begin{equation}\label{10.52}
\int_{\Gamma_{-}\cup \Gamma_{0}}\psi(z,v)\d\mu_{-}(z,v)=\int_{\Gamma_{+}\cup\Gamma_{0} }\psi(x-\tau_{-}(x,v)v,v)\d\mu_{+}(x,v).\end{equation}
\end{nb}

\subsection{About the resolvent of $\mathsf{T}_{\mathsf{H}}$}

For any $\lambda \in \mathbb{C}$ such that $\mathrm{Re}\lambda
> 0$, define
\begin{equation*}
\begin{cases}
\mathsf{M}_{\lambda} \::\:&L^1_- \longrightarrow L^1_+\\
&u \longmapsto
\mathsf{M}_{\lambda}u(x,v)=u(x-\tau_{-}(x,v)v,v)e^{-\lambda\tau_{-}(x,v)},\:\:\:(x,v) \in \Gamma_+\;;
\end{cases}
\end{equation*}

\begin{equation*}
\begin{cases}
\mathsf{\Xi}_{\lambda} \::\:&L^1_- \longrightarrow \X_{0}\\
&u \longmapsto \mathsf{\Xi}_{\lambda}u(x,v)=u(x-t_{-}(x,v)v,v)e^{-\lambda
t_{-}(x,v)}\ind_{\{t_{-}(x,v) < \infty\}},\:\:\:(x,v) \in \Omega \times V\;;
\end{cases}
\end{equation*}

\begin{equation*}
\begin{cases}
\mathsf{G}_{\lambda} \::\:& \X_{0} \longrightarrow L^1_+\\
&\varphi \longmapsto \mathsf{G}_{\lambda}\varphi(x,v)=\displaystyle
\int_0^{\tau_{-}(x,v)}\varphi(x-sv,v)e^{-\lambda s}\d s,\:\:\:(x,v) \in
\Gamma_+\;;\end{cases}
\end{equation*}
and
\begin{equation*}
\begin{cases}
\mathsf{R}_{\lambda} \::\:&\X_{0} \longrightarrow \X_{0}\\
&\varphi \longmapsto \mathsf{R}_{\lambda}\varphi(x,v)=\displaystyle
\int_0^{t_{-}(x,v)}\varphi(x-tv,v)e^{-\lambda t}\d t,\:\:\:(x,v) \in
\Omega\times V;
\end{cases}
\end{equation*}
where $\ind_E$ denotes the characteristic function of the measurable
set $E$. The interest of these operator is related to the resolution of the boundary
value problem:
\begin{equation}\label{BVP1}
\begin{cases}
(\lambda- \mathsf{T}_{\mathrm{max}})f=g,\\
\mathsf{B}^-f=u,
\end{cases}
\end{equation}
where $\lambda > 0$, $g \in \X_{0}$ and $u$ is a given function over
$\Gamma_-.$ Such a boundary value problem, with $u \in \lm$ can be uniquely solved (see \cite[Theorem 2.1]{mjm1})
\begin{theo}\label{Theo4.2} Given  $\lambda >0$, $u \in \lm$ and $g \in \X_{0}$, the function
$$f=\mathsf{R}_{\lambda}g + \mathsf{\Xi}_{\lambda}u$$ is the \textbf{
unique} solution $f \in \D(\mathsf{T}_{\mathrm{max}})$ of the boundary value
problem \eqref{BVP1}. Moreover, $\B^{+}f \in \lp$ and
$$\|\B^{+}f\|_{\lp}+\l\,\|f\|_{\X_{0}} \leq \|u\|_{\lm} + \|g\|_{\X_{0}}.$$
\end{theo}

\begin{nb}\label{nb:lift} Notice that
$\mathsf{\Xi}_{\lambda}$ is a lifting operator which, to a given $u \in
\lm$, associates a function $f=\mathsf{\Xi}_{\lambda}u \in
\D(\mathsf{T}_{\mathrm{max}})$ whose trace on $\Gamma_-$ is exactly $u$.
More precisely,
\begin{equation}\label{propxil1}\T_\mathrm{max}\mathsf{\Xi}_{\lambda}u=\lambda \mathsf{\Xi}_\lambda u, \qquad
\mathsf{B}^-\mathsf{\Xi}_\lambda u=u,\:\:\mathsf{B}^+\mathsf{\Xi}_\lambda u = \mathsf{\ml} u, \qquad  \forall u
\in \lm.
\end{equation}
Moreover, for any $\lambda >0$, one sees with the choice $u=0$ that $\mathsf{R}_{\lambda}$ coincide with $\Rs(\lambda,\mathsf{T}_{0})$. The above Theorem also shows that, for any $\lambda
> 0$
\begin{equation}\label{eq:XiLRL}
\|\mathsf{\Xi}_{\lambda}\|_{\mathscr{B}(\lm,\,\X_{0})} \leq 
\lambda^{-1}\,\qquad 
\|\mathsf{R}_{\lambda}\|_{\mathscr{B}(\X_{0})} \leq \lambda^{-1}.
\end{equation}
Moreover, one has the obvious estimates
$$\|\mathsf{M}_{\lambda}\|_{\mathscr{B}(\lm,\lp)} \leq 1, \qquad
\|\mathsf{G}_{\lambda}\|_{\mathscr{B}(\X_{0},\lp)} \leq 1
$$
for any $\l >0.$
\end{nb}
We can complement the above result with the following whose proof can be extracted from \cite[Proposition 2.6]{LMR}:
\begin{propo}\phantomsection\label{propo:resolvante}
For any $\l \in \overline{\C}_{+}$ such that $r_{\sigma}(\mathsf{M}_{\l}\H) <1$, it holds
\begin{equation}\label{eq:lambdaTH}
\Rs(\lambda,\mathsf{T}_{\mathsf{H}})=\mathsf{R}_{\lambda}+\mathsf{\Xi}_{\lambda}\mathsf{H}\Rs(1,\mathsf{M}_{\lambda}\mathsf{H})\mathsf{G}_{\lambda}
=\Rs(\lambda,\T_{0})+\sum_{n=0}^{\infty}\mathsf{\Xi}_{\lambda}\mathsf{H}\left(\mathsf{M}_{\lambda}\mathsf{H}\right)^{n}\mathsf{G}_{\lambda}
\end{equation}
where the series  converges in $\mathscr{B}(\X_{0})$.
\end{propo}
\subsection{Some auxiliary operators}\label{sec:prelim}

For $\lambda=0$, we can extend the definition of these operators in an obvious way but not all the resulting operators are bounded in their respective spaces. However, we see from  the above integration formula \eqref{10.51}, that 
$$\mathsf{M}_{0} \in \mathscr{B}(\lm,\lp) \qquad \text{ with } \quad \|\mathsf{M}_{0}u\|_{\lp}=\|u\|_{\lm}, \qquad \forall u \in \lm.$$
In the same way, one deduces from   \eqref{10.47} that for any nonnegative $\varphi \in \X_{0}$:
\begin{equation}\begin{split}
\label{Eq:G0}
\int_{\Gamma_{+}}\mathsf{G}_{0}\varphi(x,v)\d\mu_{+}(x,v)&=\int_{\Gamma_{+}}\d\mu_{+}(x,v)\int_{0}^{\tau_{-}(x,v)}\varphi(x-sv,v)\d s\\
&=\int_{\Omega \times V}\varphi(x,v)\d x \otimes \bm{m}(\d v)\end{split}\end{equation}
which proves that 
$$\mathsf{G}_{0} \in \mathscr{B}(\X_{0},\lp) \qquad \text{ with } \quad  \|\mathsf{G}_{0}\varphi\|_{\lp}=\|\varphi\|_{\X_{0}}, \qquad \forall \varphi\in \X_{0}.$$
Notice that, more generally, for any $\eta \in \R$
$$\mathsf{G}_{i\eta} \in \mathscr{B}(\X_{0},\lp), \qquad \mathsf{M}_{i\eta} \in \mathscr{B}(\lm,\lp)$$
with 
$$\|\mathsf{G}_{i\eta}\|_{\mathscr{B}(\X_{0},\lp)} \leq 1,\qquad \|\mathsf{M}_{i\eta}\|_{\mathscr{B}(\lm,\lp)} \leq 1.$$

To be able to provide a rigorous definition of the operators $\mathsf{\Xi}_{0}$ and $\mathsf{R}_{0}$ we need the following
\begin{defi} For any $k \in \N$, we define the function spaces
$$\Y^{\pm}_{k}=L^{1}(\Gamma_{\pm}\,,\max(1,|v|^{-k}) \d\mu_{\pm})$$
with the norm
$$\|u\|_{\Y^{\pm}_{k}}=\int_{\Gamma_{\pm}}|u(x,v)|\,\max(1,|v|^{-k}) \d\mu_{\pm}(x,v).$$
In the same way, for any $k \in \mathbb{N}$, we introduce 
$$\X_{k}=L^{1}(\Omega \times V\,,\max(1,|v|^{-k})\d x\otimes \bm{m}(\d v))$$
with norm $\|f\|_{\X_{k}}:=\|\,\max(1,|v|^{-k})f\|_{\X_{0}},$ $f \in \X_{k}.$
\end{defi} 
\begin{nb}\label{nb:varpi} Of course, for any $k \in \N$, $\Y^{\pm}_{k}$ is continuously and densely embedded in $L^{1}_{\pm}.$ In the same way, $\X_{k}$ is continuously and densely embedded in $\X_{0}$. Introduce, for any $k \in \N$, the function
$$\varpi_{k}(v)=\max(1,|v|^{-k}), \qquad v \in V.$$
One will identify, without ambiguity, $\varpi_{k}$ with the multiplication operator acting on $L^{1}_{\pm}$ or on $\X_{0}$, e.g.
\begin{equation*}
\begin{cases}
\varpi_{k} \::\:&X \longrightarrow X\\
&f \longmapsto 
\varpi_{k}f(x,v)=\varpi_{k}(v)f,\:\:\:(x,v) \in \Omega \times V.
\end{cases}
\end{equation*}
Then, one sees that 
$$\Y^{\pm}_{k}=\{f \in L^{1}_{\pm}\;;\;\varpi_{k}f \in L^{1}_{\pm}\}, \qquad \X_{k}=\{f \in \X_{0}\;;\;\varpi_{k}f \in \X_{0}\}.$$
\end{nb}

The interest of the above boundary spaces lies in the following (see \cite[Lemma 2.8]{LMR} where \eqref{eq:M0+Y} is proven for $k=1$ but readily extends to $k \in \N$):
\begin{lemme}\phantomsection\label{lem:motau} 
For any $u \in \Y^{-}_{1}$  one has $\mathsf{\Xi}_{0}u \in \X_{0}$ with 
\begin{equation}\label{eq:Xio}
\|\mathsf{\Xi}_{0}u\|_{\X_{0}}=\int_{\Gamma_{-}}u(x,v)\tau_{+}(x,v)\d\mu_{+}(x,v) \leq D\|u\|_{\Y^{-}_{1}}, \qquad \forall u \in \Y^{-}_{1}\end{equation}
where we recall that $D$ is the diameter of $\Omega$.  
Moreover, given $k \geq 1$, if $u \in \Y^{-}_{k}$ then $\mathsf{M}_{0}u \in \Y^{+}_{k}$ and $\mathsf{\Xi}_{0}u \in \X_{k-1}$ with
\begin{equation}\label{eq:M0+Y}
\|\mathsf{M}_{0}u\|_{\Y^{+}_{k}}= \|u\|_{\Y^{-}_{k}} \qquad \text{ and } \quad \|\mathsf{\Xi}_{0}u\|_{\X_{k-1}} \leq D\|u\|_{\Y^{-}_{k}}\end{equation}
If $f \in \X_{1}$ then $\mathsf{G}_{0}f \in \Y^{+}_{1}$  and $\mathsf{R}_{0} f \in \D(\mathsf{T}_{0}) \subset X$ and $\mathsf{T}_{0}\mathsf{R}_{0}f=-f$. \end{lemme}

\begin{lemme}
The mapping
$$\eta \in \R \mapsto \mathsf{M}_{i\eta}\H \in \mathscr{B}(\lp)$$
is continuous.\end{lemme}

\begin{nb}\label{nb:PsiXn} We wish to emphasise here that, if $\H$ satisfies Assumptions \ref{hypH} \textit{1)}, then 
$$\Psi_{\H} \in \X_{n} \qquad \forall n \leq N_{\H}$$
Indeed, recall from \cite[Proposition 4.2]{LMR}, that $\Psi_{\H}=\mathsf{\Xi}_{0}\H\,\bar{\varphi}$ where $\bar{\varphi} \in \lp$ is such that
$$\mathsf{M_{0}H}\bar{\varphi}=\bar{\varphi}.$$
From Assumption \ref{hypH} \textit{1)}, $\H\bar{\varphi} \in \Y_{n+1}^{-}$ and from \eqref{eq:M0+Y}, $\bar{\varphi} \in \Y_{n+1}^{-}$ and $\Psi_{\H} \in \X_{n}.$
\end{nb}

 \subsection{About some useful derivatives} In all this Section, we establish several differentiability results regarding the various operators appearing  in the expression of the resolvent $\Rs(\l,\T_{\H})$. These results are technically not very difficult but will be fundamental for the rest of our analysis. 
We begin with the following
\begin{propo}\phantomsection\label{propo:To} Let $n \in \mathbb{N}$. There exists some  constant $C_{n} >0$ such that, for any $f \in \mathbb{X}_{n}$ it holds
$$\sup_{\lambda \in\overline{\C}_{+}}\left\|\dfrac{\d^{k}}{\d \lambda^{k}}\Rs(\lambda,\mathsf{T_{0}})f\right\|_{\X_{0}} \leq C_{n}\|f\|_{\X_{n}}, \qquad \forall k \in \{0,\ldots,n-1\}.$$
\end{propo}
\begin{proof} Given $f \in \X_{n}$ it is easy to check from the definition of $\Rs(\lambda,\mathsf{T_{0}})f$ that
$$\frac{\d^{k}}{\d \lambda^{k}}\Rs(\lambda,\mathsf{T_{0}})f(x,v)=(-1)^{k}\int_{0}^{t_{-}(x,v)}s^{k}f(x-sv,v)\exp(-\lambda s)\d s,$$
for a. e.  $(x,v) \in \Omega \times \R^{d}$. Thus $\left\|\frac{\d^{k}}{\d \lambda^{k}}\Rs(\lambda,\mathsf{T_{0}})f\right\|_{\X_{0}} \leq \|A_{k}f\|_{\X_{0}}$ where
\begin{equation}\label{eq:defAk}
A_{k}f(x,v)=\int_{0}^{t_{-}(x,v)}s^{k}f(x-sv,v)\d s \qquad  (x,v) \in \Omega \times \R^{d}.\end{equation}
Clearly, $\|A_{k}f\|_{\X_{0}} \leq \|A_{k}|f|\|_{\X_{0}}$ so to compute the norm, we can assume without loss of generality that $f$ is nonnegative.

One can compute the norm of $A_{k}f$ in the following way. First, thanks to \eqref{10.47},
\begin{equation*}\begin{split}
\int_{\Omega\times V}A_{k}f(x,v)\d x\bm{m}(\d v)&=\int_{\Gamma_{+} }\d\mu(z,v)\int_{0}^{\tau_{-}(z,v)}\d s\int_{s}^{\tau_{-}(z,v)}(t-s)^{k}f(z-tv,v)\d t\\
&=\int_{\Gamma_{+} }\d\mu(z,v)\int_{0}^{\tau_{-}(z,v)}\,f(z-tv,v)\d t\int_{0}^{t}(t-s)^{k}\d s\\
&=\frac{1}{k+1}\int_{\Gamma_{+} }\d\mu(z,v)\int_{0}^{\tau_{-}(z,v)}\,t^{k+1}f(z-tv,v)\d t\end{split}\end{equation*}
and this yields, still using \eqref{10.47},
$$
\|A_{k}f\|_{\X_{0}}=\frac{1}{k+1}\int_{\Omega\times V} t_{+}(x,v)^{k+1}f(x,v)\d x\bm{m}(\d v)$$
since $t_{+}(z-tv,v)=t\,$ for any $(z,v) \in \Gamma_{+}$ and any $t \in (0,\tau_{-}(z,v))$.
This, together with the bound $t_{+}(x,v) \leq D/|v|$ yields the estimate
$$\left\|\frac{\d^{k}}{\d \lambda^{k}}\Rs(\lambda,\mathsf{T_{0}})f\right\|_{\X_{0}} \leq \|A_{k}f\|_{\X_{0}} \leq \frac{D^{k+1}}{k+1}\int_{\Omega\times\R^{d}}|v|^{-k-1}|f(x,v)|\d x\d v \leq \frac{D^{k+1}}{k+1}\|f\|_{\X_{k+1}}$$
and the result follows with $C_{n}=\max_{0\leq k\leq n-1}\frac{D^{k+1}}{k+1}$ since $\|f\|_{\X_{k+1}} \leq \|f\|_{\X_{n}}$ for $k \leq n-1$
\end{proof}

In the same spirit, we have the following 
\begin{lemme} \phantomsection \label{lem:Gl}
Let $n \geq 0$ be given and $f \in \X_{n}$ be given. For any $j \in \{0,\ldots,n\}$ it holds 
$$\sup_{\lambda \in\overline{\C}_{+}}\left\|\dfrac{\d^{j}}{\d \lambda^{j}}\mathsf{\mathsf{G_{\lambda}}}f\right\|_{\lp} \leq  D^{j}\|f\|_{\X_{j}} \leq D^{j}\|f\|_{\X_{n}}$$
\end{lemme}
\begin{proof} Let $j \in \{0,\ldots,n\}$ be given. For $f \in \X_{n}$, it holds for $\mu$-a. e. $(x,v) \in \Gamma_{+}$
$$\dfrac{\d^{j}}{\d \lambda^{j}}\mathsf{\mathsf{G_{\lambda}}}f(x,v)=(-1)^{j}\int_{0}^{\tau_{-}(x,v)}s^{j}f(x-sv,v)\exp(-\lambda s)\d s.$$
Introducing $\varphi(x,v)=|f(x,v)|\,t_{+}(x,v)^{j}$, $(x,v) \in \Omega\times \R^{d}$, we get easily that
$$\left|\dfrac{\d^{j}}{\d \lambda^{j}}\mathsf{\mathsf{G_{\lambda}}}f(x,v)\right| \leq \int_{0}^{\tau_{-}(x,v)}\varphi(x-sv,v)\d s=\mathsf{G}_{0}\varphi(x,v).$$
Then, according to \eqref{Eq:G0},
$$\left\|\dfrac{\d^{j}}{\d \lambda^{j}}\mathsf{\mathsf{G_{\lambda}}}f\right\|_{\lp} \leq \|\mathsf{G}_{0}\varphi\|_{\lp} \leq \|\varphi\|_{\X_{0}}$$
For $j \leq n$, it is clear that $\|\varphi\|_{\X_{0}} \leq D^{j}\|f\|_{\X_{j}}\leq D^{j}\|f\|_{\X_{n}}$ and the conclusion follows.
\end{proof}
We also have the following
\begin{lemme}\label{lem:unifGeis} For any $f \in \X_{0}$, the limit
$$\lim_{\e\to0^{+}}\left\|\mathsf{G}_{\e+i\eta}f-\mathsf{G}_{i\eta}f\right\|_{\lp}=0$$
uniformly with respect to $\eta \in \R.$ \end{lemme} 
\begin{proof} Given $f \in \X_{0}$ and $(x,v) \in \Omega \times V$, 
\begin{multline*}
\left|\mathsf{G}_{\e+i\eta}f(x,v)-\mathsf{G}_{i\eta}f(x,v)\right|=\left|\int_{0}^{\tau_{-}(x,v)}\left(e^{-\e\,t}-1\right)e^{-i\eta\,t}f(x-tv,v)\d t\right|\\
\leq \int_{0}^{\tau_{-}(x,v)}\left(1-e^{-\e\,t}\right)|f(x-tv,v)|\d t,\end{multline*}
so that
$$\sup_{\eta\in \R}\left\|\mathsf{G}_{\e+i\eta}f-\mathsf{G}_{i\eta}f\right\|_{\lp} \leq \int_{\Gamma_{+}}\d\mu_{+}(x,v)\int_{0}^{\tau_{-}(x,v)}\left(1-e^{-\e\,t}\right)|f(x-tv,v)|\d t.$$
Since $1-e^{-\e\,t} \leq 1$ for any $\e >0,$ $t\geq0$, the dominated convergence theorem combined with \eqref{Eq:G0} gives the result. 
\end{proof}

Due to the regularizing effect of the boundary operator $\H$:
\begin{propo}\label{propo:regul} If $f \in \X_{k+1}$, $0\leq k\leq N_{\H}$, then
$$g_{\l}:=\Rs(\l,\T_{\H})f \in \X_{k}, \qquad \forall \l \in \C_{+}.$$
Moreover, if $\varrho_{f}=0$ then $\varrho_{g_{\l}}=0$ for all $\l \in \C_{+}.$
\end{propo}
\begin{proof} Assume that $\varrho_{f}=0$. The equation $\l\,g_{\l}-\T_{\H}g_{\l}=f$ implies, after integration, that
$$\l\,\int_{\Omega\times V}g_{\l}(x,v)\d x \otimes\bm{m}(\d v)=\int_{\Omega \times V}f(x,v)\d x\otimes\bm{m}(\d v)=0.$$
Because $\l \neq 0$, one sees that $\varrho_{g_{\l}}=0.$  
\end{proof}

\section{Fine estimates for $\Rs(1,\mathsf{M}_{\l}\H)$}\label{sec:fineR}

We aim to derive here estimates for the inverse $\Rs(1,\mathsf{M}_{\lambda}\mathsf{H})$ along the imaginary axis $\lambda=i\eta$, $\eta \in \R$, $\eta \neq 0.$
We introduce the half-planes
$$\C_{+}=\{z \in \C\;;\;\mathrm{Re}z >0\}, \qquad \overline{\C}_{+}=\{z \in \C\;;\;\mathrm{Re}z \geq 0\}.$$
In the sequel, the notion of differentiability of functions $h\::\:\l \in \overline{\C}_{+} \mapsto h(\l) \in Y$ (where $Y$ is a given Banach space) is the usual one but we have to emphasize the fact that  limits are always meant in $\overline{\C}_{+}$ \footnote{This means for instance that, if $\l_{0} \in \C_{+}$, $h$ is differentiabile means that it is holomorphic in a neighborhoud of $\l_{0}$ whereas, for $\l_{0}=i\eta_{0}$, $\eta_{0} \in \R$, the differentiability at $\l_{0}$ of $h$ at means that there exists $h'(\l_{0}) \in Y$ such that 
$$\underset{\l \in \overline{\C}_{+}}{\lim_{\l \to \l_{0}}}\left\|\frac{h(\l)-h(\l_{0})}{\l-\l_{0}}-h'(\l_{0})\right\|_{Y}=0$$
where $\|\cdot\|_{Y}$ is the norm on $Y$.}
\subsection{Fine properties of $\mathsf{M}_{\l}\H$ and $\mathsf{\Xi}_{\l}\H$}
The above result shows in particular that, for a given $f \in \X_{0}$, one cannot extend $\Rs(\l,\T_{\H})f$ to $\mathrm{Re}\l=0$. However, such an extension will exist under some additional integrability assumption of $f$. Before proving this, we begin with some important properties of $\mathsf{M}_{\l}\H$ and $\mathsf{\Xi}_{\l}\H$.
\begin{lemme}\label{lem:Meis} 
For any $\eta \in \R$, it holds
\begin{equation}\label{eq:Meis}
\left\|\mathsf{M}_{\varepsilon+i\eta}-\mathsf{M}_{i\eta}\right\|_{\mathscr{B}(\Y^{-}_{1},\lp)} \leq \varepsilon\,D, \qquad \left\|\mathsf{\Xi}_{\e+i\eta}-\mathsf{\Xi}_{i\eta}\right\|_{\mathscr{B}(\Y_{1}^{-},\X_{0})} \leq \e\,D\end{equation}
where $D$ is the diameter of $\Omega$. Consequently, 
\begin{equation}\label{eq:MeisH}
\left\|\mathsf{M}_{\varepsilon+i\eta}\H-\mathsf{M}_{i\eta}\H\right\|_{\mathscr{B}(\lp)} \leq  \varepsilon\,D\,\|\H\|_{\mathscr{B}(\lp,\Y_{1}^{-})} \qquad \forall \eta \in \R\end{equation}
and, $\left\|\mathsf{\Xi}_{\e+i\eta}\H-\mathsf{\Xi}_{i\eta}\H\right\|_{\mathscr{B}(\lp,\X_{0})} \leq \e\,D \,\|\H\|_{\mathscr{B}(\lp,\Y_{1}^{-})}$ for any $\eta \in \R.$
\end{lemme}

\begin{proof} We give the proof for $\mathsf{M}_{i\eta}$, the proof for $\mathsf{\Xi}_{i\eta}$ being exactly the same. Let $\eta \in \R$ be fixed. Let $\varphi \in \Y_{1}^{-}$ be given and $\varepsilon>0$. One has
\begin{multline*}
\left\|\mathsf{M}_{\e+i\eta}\varphi-\mathsf{M}_{i\eta}\varphi\right\|_{\lp}
=\int_{\Gamma_{+}}\left|e^{-(\e+i\eta)\tau_{-}(x,v)}-e^{-i\eta\tau_{-}(x,v)}\right|
\,|\mathsf{M}_{0}\varphi(x,v)|\d\mu_{+}(x,v)\\
=\int_{\Gamma_{+}}\left|\exp(-\e\tau_{-}(x,v))-1\right|\,|\mathsf{M}_{0}\varphi(x,v)|\d\mu_{+}(x,v)\\
\leq C_{0}\e\int_{\Gamma_{+}}\tau_{-}(x,v)\,|\mathsf{M}_{0}\varphi(x,v)|\d\mu_{+}(x,v)\end{multline*}
where 
$C_{0}:=\sup_{s >0}\frac{|\exp(-s)-1|}{s}=1.$
Now, because $\tau_{-}(x,v) \leq D|v|^{-1}$  we get
$$\left\|\mathsf{M}_{\e+i\eta}\varphi-\mathsf{M}_{i\eta}\varphi\right\|_{\lp} \leq \e\,D\int_{\Gamma_{+}}|v|^{-1}\,|\mathsf{M}_{0}\varphi(x,v)|\d\mu_{+}(x,v)=D\,C_{0}\e\|\mathsf{M}_{0}\psi\|_{\lp}$$
where $\psi(x,v)=|v|^{-1}\varphi(x,v)$. Because $\|\mathsf{M}_{0}\psi\|_{\lp}=\|\psi\|_{\lm}=\|\varphi\|_{\Y^{-}_{1}}$ we obtain
$$\left\|\mathsf{M}_{\e+i\eta}\varphi-\mathsf{M}_{i\eta}\varphi\right\|_{\lp} \leq \e\,D\,\|\varphi\|_{\Y^{-}_{1}}$$
which proves \eqref{eq:Meis}. Now,  since $\mathrm{Range}(\H) \subset \Y^{-}_{1}$, one deduces \eqref{eq:MeisH} directly from \eqref{eq:Meis}.\end{proof}
A first consequence of the above convergence result is the following which will play an important role in the sequel:
\begin{cor}\label{cor:uniformpower}
For any $j \in \N$, one has
$$\lim_{\varepsilon\to0}\left\|\left(\mathsf{M}_{\e+i\eta}\H\right)^{j}-\left(\mathsf{M}_{i\eta}\H\right)^{j}\right\|_{\mathscr{B}(\lp)}=0$$
uniformly with respect to $\eta \in \R.$\end{cor}
\begin{proof} The proof is done by induction on $j \in \N.$ For $j=1$, the result is true. Noticing that, for any $j \in \N$
\begin{multline*}
\left\|\left(\mathsf{M}_{\e+i\eta}\H\right)^{j+1}-\left(\mathsf{M}_{i\eta}\H\right)^{j+1}\right\|_{\mathscr{B}(\lp)}
\leq \left\|\left(\mathsf{M}_{\e+i\eta}\H\right)^{j}-\left(\mathsf{M}_{i\eta}\H\right)^{j}\right\|_{\mathscr{B}(\lp)}\|\mathsf{M}_{\e+i\eta}\H\|_{\mathscr{B}(\lp)} \\
+ \left\|\left(\mathsf{M}_{i\eta}\H\right)^{j}\right\|_{\mathscr{B}(\lp)}\,\left\| \mathsf{M}_{\e+i\eta}\H-\mathsf{M}_{i\eta}\H\right\|_{\mathscr{B}(\lp)}
\end{multline*}
we easily get the result since $\|\mathsf{M}_{i\eta}\H\|_{\mathscr{B}(\lp)}\leq1$.\end{proof}
\begin{nb}\label{cor:exten} An important consequence of the above Corollary \ref{cor:uniformpower} is that the holomorphic functions
$$\lambda \in \C_{+} \mapsto \mathsf{M}_{\l}\H \in \mathscr{B}(\lp) \quad { and } \quad 
\lambda \in \C_{+} \mapsto \mathsf{\Xi}_{\l}\H \in \mathscr{B}(\lp,\X_{0})$$
can be extended to continuous functions on $\overline{\C}_{+}.$
\end{nb}

One can extend easily the above result in the following
\begin{propo}\label{propo:convK} For any $\eta \in \R$ and any $k \in \N$, 
\begin{equation}\label{eq:Meisk}
\left\|\mathsf{M}_{\varepsilon+i\eta}-\mathsf{M}_{i\eta}\right\|_{\mathscr{B}(\Y^{-}_{k+1},\Y_{k})} \leq \varepsilon\,D, \qquad \left\|\mathsf{\Xi}_{\e+i\eta}-\mathsf{\Xi}_{i\eta}\right\|_{\mathscr{B}(\Y_{k+1}^{-},\X_{k})} \leq \e\,D\end{equation}
where $D$ is the diameter of $\Omega$. Consequently, for any $k \leq N_{\H}$ where $N_{\H}$ is defined in \eqref{eq:HYn}
\begin{multline}\label{eq:MeisHk}
\left\|\mathsf{M}_{\varepsilon+i\eta}\H-\mathsf{M}_{i\eta}\H\right\|_{\mathscr{B}(\lp,\Y_{k}^{+})} \leq  \varepsilon\,D\,\|\H\|_{\mathscr{B}(\lp,\Y_{k+1}^{-})} \\
\text{ and } \quad \left\|\mathsf{\Xi}_{\e+i\eta}\H-\mathsf{\Xi}_{i\eta}\H\right\|_{\mathscr{B}(\lp,\X_{k})} \leq \e\,D\,\|\H\|_{\mathscr{B}(\lp,\Y_{k+1}^{-})} 
\qquad \forall \eta \in \R.\end{multline}
\end{propo}
\begin{proof} The proof is exactly the same as Lemma \ref{lem:Meis} and is omitted here.\end{proof}
A consequence of this is the following 
\begin{propo}\label{prop:derMeis} For any $k \in \N$ and $\varphi \in \Y_{k+1}^{-}$, one has
$$\left\|\frac{\d^{k}}{\d \eta^{k}}\mathsf{M}_{\varepsilon+i\eta}\varphi-\frac{\d^{k}}{\d \eta^{k}}\mathsf{M}_{i\eta}\varphi\right\|_{\lp} \leq \varepsilon\,D\,\|\varphi\|_{\Y_{k+1}^{-}} \qquad \forall \eta \in \R, \qquad \varepsilon >0.$$
In particular, for any $k \leq N$ 
$$\left\|\frac{\d^{k}}{\d \eta^{k}}\mathsf{M}_{\varepsilon+i\eta}\H-\frac{\d^{k}}{\d \eta^{k}}\mathsf{M}_{i\eta}\H\right\|_{\mathscr{B}(\lp)} \leq \varepsilon\,D\,\|\H\|_{\mathscr{B}(\lp,\Y_{k+1}^{-})} \qquad \forall \eta \in \R, \qquad \varepsilon >0.$$
\end{propo}
\begin{proof} For $\varphi \in \Y_{k+1}^{-}$, $\e >0$, $\eta \in\R$ one checks easily that
\begin{multline*}
\dfrac{\d^{k}}{\d\eta^{k}}\mathsf{M}_{\varepsilon+i\eta}\varphi(x,v)-\dfrac{\d^{k}}{\d\eta^{k}}\mathsf{M}_{i\eta}\varphi(x,v)=(-i)^{k}\tau_{-}(x,v)^{k}\\
\left(\exp\left(-(\e+i\eta)\tau_{-}(x,v)\right)-\exp\left(-i\eta\tau_{-}(x,v)\right)\right)\mathsf{M}_{0}\varphi(x,v), 
\end{multline*}
for any $(x,v) \in \Gamma_{-}.$ Therefore
$$
\left|\dfrac{\d^{k}}{\d\eta^{k}}\mathsf{M}_{\varepsilon+i\eta}\varphi(x,v)-\dfrac{\d^{k}}{\d\eta^{k}}\mathsf{M}_{i\eta}\varphi(x,v)\right|
=\tau_{-}(x,v)^{k}|\mathsf{M}_{0}\varphi(x,v)|\left|\exp(-\e(\tau_{-}(x,v))-1\right|
$$
and, reasoning as in Lemma \ref{lem:Meis}, we get
$$\left|\dfrac{\d^{k}}{\d\eta^{k}}\mathsf{M}_{\varepsilon+i\eta}\varphi(x,v)-\dfrac{\d^{k}}{\d\eta^{k}}\mathsf{M}_{i\eta}\varphi(x,v)\right| \leq \e,D^{j+1}|v|^{-k-1}\left|\mathsf{M}_{0}\varphi(x,v)\right|$$
and the result follows.
\end{proof}
As a consequence
\begin{cor}
For any $k \in \N$ such that $\H \in \mathscr{B}(\lp,\Y_{k+1}^{-})$  (i.e. $k \leq N_{\H}$), the holomorphic function
$$\l=\alpha+i\eta  \in \C_{+} \longmapsto \frac{\d^{j}}{\d \eta^{j}}\mathsf{M}_{\l}\H \in \mathscr{B}(\lp), \qquad 0 \leq j \leq k$$
can be extended to a continuous functions on $\overline{\C}_{+}.$ In particular, the  mapping
$$\eta \in \R \mapsto \mathsf{M}_{i\eta}\H \in \mathscr{B}(\lp)$$
is of class $\mathcal{C}^{k}.$ Moreover, for any $\varphi\in \Y_{1}^{-}$, the limit $\lim_{\l\to 0}\dfrac{\d}{\d\l}\mathsf{M}_{\l}\varphi$
exists in $\lp$. In particular, 
$$\lim_{\l\to 0}\dfrac{\d}{\d\l}\mathsf{M}_{\l}\H=-\tau_{-}\mathsf{M}_{0}\H$$
exists in $\mathscr{B}(\lp)$ where, as before, we use the same symbol for the function $\tau_{-}(\cdot,\cdot)$ and the multiplication operator by that function.
\end{cor}
\begin{proof} Since the mapping $\l \in \C_{+} \mapsto \mathsf{M}_{\l}\H \in \mathscr{B}(\lp)$ is holomorphic, for any $\e >0$, the mapping
$$\eta \in\R \mapsto \frac{\d^{j}}{\d \eta^{j}}\mathsf{M}_{\e+i\eta}\H \in \mathscr{B}(\lp)$$
is continuous for any $0 \leq j \leq k$. Thanks to the previous  Proposition \ref{prop:derMeis}, we can let $\e \to 0$ and conclude that the derivatives exist and are continuous on $\R$. Now, for any $\l \in \overline{\C}_{+}$, one has
$$\dfrac{\d}{\d\l}\mathsf{M}_{\l}\varphi(x,v)=-\tau_{-}(x,v)\exp(-\lambda\tau_{-}(x,v))\varphi(x-\tau_{-}(x,v)v,v)=-\tau_{-}\mathsf{M}_{\l}\varphi$$
so that, by the dominated convergence theorem, 
$$\lim_{\l\to0}\dfrac{\d}{\d\l}\mathsf{M}_{\l}\varphi=-\tau_{-}\mathsf{M}_{0}\varphi$$
provided $\varphi \in \Y_{1}^{-}.$ The conclusion follows easily.
\end{proof}
In the same spirit
\begin{lemme} For any $f \in \X_{1}$, the limit
$$\lim_{\l \to 0}\dfrac{\d}{\d \l}\mathsf{G}_{\l}f=-\mathsf{G}_{0}(t_{+}f)$$
exists in $\lp$.
\end{lemme}
\begin{proof} For $\l \in \C_{+}$ and $f \in \X_{1}$, one has 
$$\dfrac{\d}{\d\l}\mathsf{G}_{\l}f(x,v)=-\int_{0}^{\tau_{-}(x,v)}tf(x-tv,v)e^{-\l t}\d t, \qquad \text{ for a. e.} (x,v)\in \Gamma_{+}.$$
Notice that, for any $f \in \X_{1}$ and any $(x,v) \in \Gamma_{+}$
$$\int_{0}^{\tau_{-}(x,v)}tf(x-tv,v)\d t=\int_{0}^{\tau_{-}(x,v)}t_{+}(x-tv,v)f(x-tv,v)\d t=\mathsf{G}_{0}(t_{+}f)(x,v)$$
since $t_{+}(x-tv,v)=t$. In particular, $\mathsf{G}_{0}(t_{+}f) \in \lp$ since $f \in \X_{1}$ and  we can invoke the dominated convergence theorem to get the conclusion.\end{proof}
\subsection{Spectral properties of $\mathsf{M}_{\l}\H$ along the imaginary axis}
We study here the properties of $\mathsf{M}_{i\eta}\H$ for $\eta \in\R$.
\begin{propo}\phantomsection \label{prop:Meps} 
For any $\lambda \in \C \setminus \{0\}$ with $\mathrm{Re}\l \geq 0$, 
$$r_{\sigma}(\mathsf{M}_{\l}\H) < 1.$$
In particular, for any $\eta \in \R$, $\eta \neq 0$, it holds $r_{\sigma}(\mathsf{M}_{i\eta}\H) < 1$.

\end{propo}
\begin{proof} We give the proof only for $\mathrm{Re}\l=0$, the case $\mathrm{Re}\l >0$ being similar. Since $\left|\mathsf{M}_{i\eta}\H\psi(x,v)\right|=\left|\mathsf{M}_{0}\H\psi(x,v)\right|$ for any $\psi \in \lp$, $(x,v) \in \Gamma_{+},$ $\eta \in \R$, one sees that 
$$\mathsf{M}_{i\eta}\H \in \mathscr{B}(\lp) \qquad \text{ with } \quad \left|\mathsf{M}_{i\eta}\H \right| \leq \mathsf{M}_{0}\H$$
where $\left|\mathsf{M}_{i\eta}\H\right|$ denotes the absolute value operator of $\mathsf{M}_{i\eta}\H$ (see \cite{chacon}). Being $\mathsf{M}_{0}\H$ power compact, the same holds for $\left|\mathsf{M}_{i\eta}\H \right|$ by a domination argument  so that
$$r_{\mathrm{ess}}(\left|\mathsf{M}_{i\eta}\H \right|)=0$$
where $r_{\mathrm{ess}}(\cdot)$ denotes the essential spectral radius. We prove that $r_{\sigma}(\left|\mathsf{M}_{i\eta}\H\right|) < 1$ by contradiction: assume, on the contrary, $r_{\sigma}(\left|\mathsf{M}_{i\eta}\H\right|) \geq 1 > r_{\mathrm{ess}}(\left|\mathsf{M}_{i\eta}\H\right|)$, then $r_{\sigma}(\left|\mathsf{M}_{i\eta}\H\right|)$ is an isolated eigenvalue of $\left|\mathsf{M}_{i\eta}\H\right|$ with finite algebraic multiplicity and also an eigenvalue of the dual operator, associated to nonnegative eigenfunction. From the fact that $\left|\mathsf{M}_{i\eta}\H\right| \leq \mathsf{M}_{0}\H$ with $\left|\mathsf{M}_{i\eta}\H\right| \neq \mathsf{M}_{0}\H$, one can invoke \cite[Theorem 4.3]{marek} to get that
$$r_{\sigma}(\left|\mathsf{M}_{i\eta}\H\right|) < r_{\sigma}(\mathsf{M}_{0}\H)=1$$
which is a contradiction. Therefore, $r_{\sigma}(\left|\mathsf{M}_{i\eta}\H\right|) < 1$ and, by comparison, the conclusion holds true. 
\end{proof}
 
We deduce the following 
\begin{cor}\label{cor:ResMei} For any $\eta_{0} \in \R\setminus\{0\}$, there is $0 < \delta < \tfrac{1}{2}|\eta_{0}|$ such that
$$\lim_{\e\to0^{+}}\sup_{|\eta-\eta_{0}| < \delta}\bigg\|\Rs(1,\mathsf{M}_{\e+i\eta}\H)-\Rs(1,\mathsf{M}_{i\eta}\H)\bigg\|_{\mathscr{B}(\lp)}=0.$$
\end{cor}
\begin{proof} Notice that, if $0 <\delta<\tfrac{|\eta_{0}|}{2}$ then, any $\eta \neq 0$ whenever  $|\eta-\eta_{0}| < \delta$. Without loss of generality, we can assume $\eta_{0} >0$. From Proposition \ref{prop:Meps}, there is $\varrho \in (0,1)$ such that $r_{\sigma}(\mathsf{M}_{i\eta_{0}}\H) < \varrho <1$. In particular, there is $\ell \in \N$ such that
$$\left\|\left(\mathsf{M}_{i\eta_{0}}\H\right)^{\ell}\right\|_{\mathscr{B}(\lp)}^{\frac{1}{\ell}} < \varrho<1.$$
Since $\mathsf{M}_{i\eta}\H$ converges to $\mathsf{M}_{i\eta_{0}}\H$ in operator norm as $\eta \to \eta_{0}$, there is $\delta < \frac{\eta_{0}}{2}$ such that
$$\left\|\left(\mathsf{M}_{i\eta}\H\right)^{\ell}\right\|_{\mathscr{B}(\lp)} < \varrho^{\ell} \qquad \forall \eta \in (\eta_{0}-\delta,\eta_{0}+\delta).$$
Because of Corollary \ref{cor:uniformpower}, there is $\e_{0} >0$ small enough, such that, for any $0<\e<\e_{0}$ we also have
$$\left\|\left(\mathsf{M}_{\e+i\eta}\H\right)^{\ell}\right\|_{\mathscr{B}(\lp)} < \varrho^{\ell} \qquad \forall \eta \in (\eta_{0}-\delta,\eta_{0}+\delta).$$
One has then
$$\Rs(1,\mathsf{M}_{\e+i\eta}\H)=\sum_{n=0}^{\infty}\left(\mathsf{M}_{\e+i\eta}\H\right)^{n}=\sum_{k=0}^{\infty}\sum_{j=0}^{\ell-1}\left(\mathsf{M}_{\e+i\eta}\H\right)^{k\ell+j}$$
and, similarly
$$\Rs(1,\mathsf{M}_{i\eta}\H)=\sum_{k=0}^{\infty}\sum_{j=0}^{\ell-1}\left(\mathsf{M}_{\e+i\eta}\H\right)^{k\ell+j} \qquad \forall \e \in (0,\e_{0}), \qquad \eta \in (\eta_{0}-\delta_{0},\eta_{0}+\delta_{0}).$$
Therefore,
$$\Rs(1,\mathsf{M}_{\e+i\eta}\H)-\Rs(1,\mathsf{M}_{i\eta}\H)=\sum_{k=0}^{\infty}\sum_{j=0}^{\ell-1}\left[\left(\mathsf{M}_{\e+i\eta}\H\right)^{k\ell+j}-\left(\mathsf{M}_{i\eta}\H\right)^{k\ell+j}\right]$$
for any $\e \in (0,\e_{0}),$ $\eta \in (\eta_{0}-\delta_{0},\eta_{0}+\delta_{0}).$ Using again Corollary \ref{cor:uniformpower}, each term of the series converges to $0$ as $\eta \to 0$ \emph{uniformly} with respect to $\eta \in  (\eta_{0}-\delta_{0},\eta_{0}+\delta_{0}).$ To prove the result, it is enough therefore to show that the reminder of the series can be made arbitrarily small in operator norm \emph{uniformly} on $(\eta_{0}-\delta_{0},\eta_{0}+\delta_{0}).$ Since, for any $k,j \geq 0$
$$\left\|\left(\mathsf{M}_{\e+i\eta}\H\right)^{k\ell+j}\right\|_{\mathscr{B}(\lp)} \leq \left\|\left(\mathsf{M}_{\e+i\eta}\H\right)^{k\ell}\right\|_{\mathscr{B}(\lp)} \leq \left\|\left(\mathsf{M}_{\e+i\eta}\H\right)^{\ell}\right\|_{\mathscr{B}(\lp)}^{k} \leq \varrho^{k\ell}$$
for any $\e \in [0,\e_{0})$, we get that, for any $n \geq 0$
$$\sup_{|\eta-\eta_{0}| < \delta_{0}}\sup_{\e \in (0,\e_{0})}\sum_{k=n}^{\infty}\sum_{j=0}^{\ell-1}\left\|\left(\mathsf{M}_{\e+i\eta}\H\right)^{k\ell+j}-\left(\mathsf{M}_{i\eta}\H\right)^{k\ell+j}\right\|_{\mathscr{B}(\lp)} \leq 2\ell\sum_{k=n}^{\infty}r^{\ell\,k}$$
which goes to $0$ as $n \to \infty$. This combined with the term-by-term convergence of the series as $\e \to 0$  gives the result.
\end{proof}

\section{About the boundary function of $\Rs(\l,\T_{\H})$}\label{sec:boudF}

We recall the following result (see \cite[Theorem 1.1.(c)]{voigt84}): 
\begin{theo}\label{theo:spectT0}
Under Assumption \ref{hypO} \textit{3)},  $\mathfrak{S}(\T_{0})=\{\l\in \C\;;\,\mathrm{Re}\l \leq 0\}.$
\end{theo} 
We first deduce from this property of $\T_{0}$ the following:
\begin{theo}\phantomsection \label{theo:spectTH} If $\mathsf{H} \in \mathscr{B}(\lp,\lm)$ satisfies Assumption \ref{hypH} then $i\R \subset \mathfrak{S}(\mathsf{T_{H}}).$\end{theo}
\begin{proof} According to Theorem \ref{theo:spectT0}, it holds 
\begin{equation}\label{eq:epsi-s-T0}
\lim_{\varepsilon\to 0^{+}}\left\|\Rs(\varepsilon+i\eta,\mathsf{T_{0}})\right\|_{\mathscr{B}(\X_{0})}=+\infty, \qquad \forall \eta \in \R.\end{equation}
From Proposition \ref{prop:Meps} and Banach-Steinhaus Theorem \cite[Theorem 2.2, p. 32]{brezis}, for any $\eta \neq 0$,
$$\limsup_{\varepsilon \to 0^{+}}\left\|\Rs(1,\mathsf{M}_{\varepsilon+i\eta}\mathsf{H})\right\|_{\mathscr{B}(\lp)} < \infty.$$
Since, under \eqref{eq:HYn}, the range of $\mathsf{H}$ is included in $\Y^{-}_{1}$ and $\|\mathsf{\Xi_{0}}u\|_{\X_{0}} \leq \|u\|_{\Y^{-}_{1}}$ for any $u \in \Y^{-}_{1}$ (see Lemma  \ref{lem:motau}), one has
$$\sup_{\varepsilon >0,\eta\in \R}\|\mathsf{\Xi_{\varepsilon+i\eta}H}\|_{\mathscr{B}(\lp,\X_{0})} < \infty.$$ Moreover $\sup_{\varepsilon >0,\eta\in \R}\|\mathsf{G}_{\varepsilon+i\eta}\|_{\mathscr{B}(\X_{0},\lp)} < \infty$ we get that, for any $\eta \in \R$, $\eta \neq 0$, it holds:
$$\limsup_{\varepsilon \to 0^{+}}\left\|\mathsf{\Xi_{\varepsilon+i\eta}H}\Rs(1,\mathsf{M}_{\varepsilon+i\eta}\mathsf{H})\mathsf{G}_{{\varepsilon+i\eta}}\right\|_{\mathscr{B}(\X_{0})} < \infty.$$
This, together with \eqref{eq:epsi-s-T0} and \eqref{eq:lambdaTH} proves that, for any $\eta \in \R$, $\eta \neq 0$, it holds
$$\limsup_{\varepsilon \to 0^{+}}\left\|\Rs(\varepsilon+i\eta,\mathsf{T_{H}})\right\|_{\mathscr{B}(\X_{0})}=\infty$$
whence $i\eta\in \mathfrak{S}(\mathsf{T_{H}})$ for any $\eta \neq 0$. Recalling that $0 \in \mathfrak{S}_{p}(\T_{\H})$ we get the conclusion.\end{proof}

\subsection{Definition of the boundary function away from zero}

One has first the basic observation:
\begin{lemme}\phantomsection\label{lem:To} For any $f \in \X_{1}$ and any $\eta \in \R$, 
$$\lim_{\e \to 0^{+}}\Rs(\e+i\eta,\T_{0})f \qquad (\l \in \C_{+})$$
exists in $\X_{0}$ and is denoted $\mathbf{R}^{0}_{f}(\eta).$ Moreover, given $k \geq 0$, if $f \in \X_{k+1}$ then $\mathbf{R}_{f}^{0}(\eta) \in \X_{k}$ and
\begin{equation}\label{eq:R0fk}
\lim_{\e \to 0^{+}}\|\Rs(\e+i\eta,\T_{0})f-\mathbf{R}_{f}^{0}(\eta)\|_{\X_{k}}=0 \qquad (\l \in \C_{+}).\end{equation}
\end{lemme}
\begin{proof} The existence of the strong limit then is deduced from the pointwise convergence:  for almost every $(x,v) \in \Omega \times V$
$$\lim_{\e\to0^{+}}\int_{0}^{t_{-}(x,v)}\exp(-(\e+i\eta) t)f(x-tv,v)\d t= \mathbf{R}_{f}^{0}(\eta)(x,v):=\int_{0}^{t_{-}(x,v)}\exp(-i\eta t)f(x-tv,v)\d t$$
and the dominated convergence theorem. Now, if $f \in \X_{k+1}$ then $g_{k}=\varpi_{k}f \in \X_{1}$ where we recall that $\varpi_{k}$ has been defined in Remark \ref{nb:varpi}. Moreover, the first part of the result implies that
\begin{equation}\label{eq:covgk}
\lim_{\e\to 0^{+}}\|\Rs(\e+i\eta,\T_{0})g_{k}-\mathbf{R}_{g_{k}}^{0}(\eta)\|_{\X_{0}}=0.\end{equation}
Since $\varpi_{k}$ and $\Rs(\e+i\eta,\T_{0})$ are commuting operator (recall that $\varpi_{k}$ is a multiplication operator by the function $\varpi_{k}(v)=\max(1,|v|^{-k})$ which is independent of $x$), one has $\Rs(\e+i\eta,\T_{0})g_{k}=\varpi_{k}\Rs(\e+i\eta,\T_{0})f$ and 
 $\mathbf{R}_{g_{k}}^{0}(\eta)(x,v)=\varpi_{k}\mathbf{R}_{f}^{0}(\eta)$. The convergence convergence \eqref{eq:covgk} is then equivalent to
$$\lim_{\e\to0^{+}}\left\|\varpi_{k}\left(\Rs(\e+i\eta,\T_{0})f-\mathbf{R}_{f}^{0}(\eta)\right)\right\|_{\X_{0}}=0$$
which is exactly \eqref{eq:R0fk}.
\end{proof}
\begin{nb}\label{nb:boundedR0} Notice that, arguing as in Prop. \ref{propo:To}, one sees that
$$\sup_{\eta\in\R}\|\mathbf{R}_{f}^{0}(\eta)\|_{\X_{k}} \leq \|\mathbf{R}_{f}^{0}(0)\|_{\X_{k}} \leq \|A_{0}f\|_{\X_{k}} \leq \|f\|_{\X_{k}}$$
where $A_{0}$ is defined through \eqref{eq:defAk}.
\end{nb}
One sees from the above result and the fact that $\Rs(\l,\T_{\H})$ splits as
\begin{equation}\label{eq:split}
\Rs(\l,\T_{\H})=\Rs(\l,\T_{0})+\mathrm{R}(\l), \qquad \text{ with } \quad \mathrm{R}(\l):=\mathsf{\Xi}_{\l}\H(1-\mathsf{M}_{\l}\H)^{-1}\mathsf{G}_{\l}\end{equation}
that, to be able to define
$$\lim_{\e \to 0^{+}}\Rs(\e+i\eta,\T_{\H})f, \qquad \eta \in \R$$
in $\X_{0}$, we will at least need to assume $f \in \X_{1}$. This is actually enough if we consider only $\eta \in \R\setminus\{0\}$. Namely
\begin{lemme}\label{lem:converTH} Let $\eta_{0} \in \R\setminus\{0\}$ and let $k \in \N$, $k \leq N_{\H}.$ Then, for any $f \in \X_{k+1}$, the limit
$$\lim_{\e\to0^{+}}\Rs(\e+i\eta,\T_{\H})f$$
exists in $\X_{k}$ \emph{uniformly} on some neighbourhood of $\eta_{0}$. We denote $\mathbf{R}_{f}(\eta)$ this limit.
\end{lemme}
\begin{proof} According to Lemma \ref{lem:To} and the above splitting \eqref{eq:split} of $\Rs(\l,\T_{\H})$, we only need to prove that
$$\lim_{\e\to0^{+}}\mathrm{R}(\e+i\eta) f$$
converges uniformly with respect to $\eta$ in some neighbourhood of $\eta_{0}.$ We know from Corollary \ref{cor:ResMei} and Lemma \ref{lem:unifGeis} that, in some neighbourhood $I_{0}:=(\mu_{0}-\delta_{0},\mu_{0}+\delta_{0})$ of $\mu_{0}$ it holds
\begin{equation}\label{eq:convunifr}
\lim_{\e\to0^{+}}\sup_{\eta \in I_{0}}\left\|\Rs(1,\mathsf{M}_{\e+i\eta}\H)\mathsf{G}_{\e+i\eta}f-\Rs(1,\mathsf{M}_{i\eta}\H)\mathsf{G}_{i\eta}f\right\|_{\lp}=0.\end{equation}
Since we know from \eqref{eq:MeisHk} that
$$\lim_{\e\to 0^{+}}\sup_{\eta \in \R}\left\|\mathsf{\Xi}_{\e+i\eta}\H-\mathsf{\Xi}_{i\eta}\H\right\|_{\mathscr{B}(\lp,\X_{k})}=0$$
we get the conclusion.
\end{proof}

\subsection{Definition of the boundary function near zero} \label{sec:near0}
To consider the more delicate case $\eta=0$, we will first need a careful study of the spectral properties of $\mathsf{M}_{\l}\H$ for $\l \in \overline{\C}_{+}$ with $|\l|$ small.

\subsubsection{{Spectral properties of $\mathsf{M}_{\l}\H$ in the vicinity of $\l=0$.}}
We recall that, being $\mathsf{M}_{0}\H$ stochastic and irreducible, the spectral radius $r_{\sigma}(\mathsf{M}_{0}\H)=1$ is a algebraically simple and isolated eigenvalue of $\mathsf{M}_{0}\H$ and there exists $0 < r < 1$ such that
$$\mathfrak{S}(\mathsf{M}_{0}\H) \setminus \{1\} \subset \{z \in \C\;;\;|z| < r\}$$
and there is a normalised and positive eigenfunction $\varphi_{0}$ such that
$$\mathsf{M}_{0}\H\,\varphi_{0}=1, \qquad \int_{\Gamma_{+}}\varphi_{0}\,\d\mu_{+}=1.$$
Because $\mathsf{M}_{0}\H$ is stochastic, the dual operator $\left(\mathsf{M}_{0}\H\right)^{\star}$ (in $L^{\infty}(\Gamma_{+},\d\mu_{+})$) admits the eigenfunction 
$$\varphi_{0}^{\star}=\mathbf{1}_{\Gamma_{+}}$$ 
associated to the algebraically simple eigenvalue $1.$ The spectral projection of $\mathsf{M}_{0}\H$ associated to the eigenvalue $1$ is then defined as
$$\mathsf{P}(0)=\frac{1}{2i\pi}\oint_{\{|z-1|=r_{0}\}}\Rs(z,\mathsf{M}_{0}\H)\d z$$
where $r_{0} >0$ is chosen so that $\{z \in \C\;;\;|z-1|=r_{0}\} \subset \{z \in \C\;;\;|z| >r\}.$ Such a spectral structure is somehow inherited by $\mathsf{M}_{\l}\H$ for $\l$ small enough:
\begin{propo}\label{prop:eigenMLH}
For any $\l \in \overline{\C}_{+}$  the spectrum of $\mathsf{M}_{\l}\H$ is given by
$$\mathfrak{S}(\mathsf{M}_{\l}\H)=\{0\} \cup \{\nu_{n}(\l)\;;\;n \in \N_{\l} \subset \N\}$$
where, $\N_{\l}$ is a (possibly finite) subset of $\N$ and, for each $n \in \N_{\l}$, $\nu_{n}(\l)$ is an isolated eigenvalue of $\mathsf{M}_{\l}\H$ of finite algebraic multiplicities and $0$ being the only possible accumulation point of the sequence $\{\nu_{n}(\l)\}_{n\in \N_{\l}}$. Moreover,
$$|\nu_{n}(\l)| < 1 \qquad \text{ for any } n \in \N_{\l}, \qquad \l \neq 0.$$
Finally, there exists $\delta_{0} >0$ such that, for any $|\lambda| \leq \delta_{0}$, $\l \in \overline{\C}_{+}$,
$$\mathfrak{S}(\mathsf{M}_{\l}\H) \cap \{z \in \C\;;\;|z-1| <\epsilon\}=\{\nu(\l)\}$$
where $\nu(\l)$ is an algebraically simple eigenvalue of $\mathsf{M}_{\l}\H$ such that
$$\lim_{\l \to 0}\nu(\l)=1$$
and there exist an eigenfunction $\varphi_{\l}$ of $\mathsf{M}_{\l}\H$ and an  eigenfunction $\varphi^{\star}_{\l}$ of $\left(\mathsf{M}_{\l}\H\right)^{\star}$ associated to $\nu(\l)$ such that
$$\lim_{\l\to0}\|\varphi(\l)-\varphi_{0}\|_{\lp}=0, \qquad \lim_{\l\to0}\|\varphi_{\l}^{\star}-\varphi_{0}^{\star}\|_{L^{\infty}(\Gamma_{+},\d\mu_{+})}=0.$$
\end{propo}
\begin{proof} Since $\left|\left(\mathsf{M}_{\l}\H\right)^{2}\right| \leq \left(\mathsf{M}_{0}\H\right)^{2}$, one has that $(\mathsf{M}_{\l}\H)^{2}$ is weakly compact and the structure of $\mathfrak{S}(\mathsf{M}_{\l}\H)$ follows. The fact that all eigenvalues have modulus less than one comes from Proposition \ref{prop:Meps}. This gives the first part of the Proposition. For the second part, because $\mathsf{M}_{\l}\H$ converges in operator norm towards $\mathsf{M}_{0}\H$ as $\l \to 0$ $(\l \in \overline{\C}_{+})$, it follows from general results about the separation of the spectrum \cite[Theorem 3.16, p.212]{kato} that, for $|\l| < \delta_{0}$ small enough, the curve $\{z \in \C\;;\;|z-1|=r_{0}\}$ is separating the spectrum $\mathfrak{S}(\mathsf{M}_{\l}\H)$ into two disjoint parts, say
$$\mathfrak{S}(\mathsf{M}_{\l}\H)=\mathfrak{S}_{\text{in}}(\mathsf{M}_{\l}\H) \cup \mathfrak{S}_{\text{ext}}(\mathsf{M}_{\l}\H)$$
where $\mathfrak{S}_{\text{in}}(\mathsf{M}_{\l}\H) \subset \{z\in \C\;;|z-1| < r_{0}\}$ and $\mathfrak{S}_{\text{ext}}(\mathsf{M}_{\l}\H) \subset \{z\in \C\;;|z-1|>r_{0}\}.$ Moreover, the spectral projection of $\mathsf{M}_{\l}\H$ associated to $\mathfrak{S}_{\text{in}}(\mathsf{M}_{\l}\H)$, defined as,
\begin{equation}\label{eq:Pl}
\mathsf{P}(\l)=\frac{1}{2i\pi}\oint_{\{|z-1|=r_{0}\}}\Rs(z,\mathsf{M}_{\l}\H)\d z,\end{equation}
is converging in operator norm to $\mathsf{P}(0)$ as $\l \to 0$ $(\mathrm{Re}\l \geq 0)$ so that, in particular, up to reduce again $\delta_{0}$,
$$\mathrm{dim}(\mathrm{Range}(\mathsf{P}(\l)))=\mathrm{dim}(\mathrm{Range}(\mathsf{P}(0)))=1, \qquad |\l| < \delta_{0}, \mathrm{Re}\l \geq 0.$$
This shows that
$$\mathfrak{S}_{\text{in}}(\mathsf{M}_{\l}\H)=\mathfrak{S}(\mathsf{M}_{\l}\H) \cap \{z \in \C\;;\;|z-1| < \epsilon\}=\{\nu(\l)\}, \qquad |\l| < \delta_{0}, \mathrm{Re}\l \geq0,$$
where $\nu(\l)$ is a \emph{algebraically simple} eigenvalue of $\mathsf{M}_{\l}\H$. Notice that, clearly
$$\lim_{\l\to 0}\nu(\l)=1 \qquad (\mathrm{Re}\l \geq0).$$
In the same way, defining
$$\mathsf{P}(\l)^{\star}=\frac{1}{2i\pi}\oint_{\{|z-1|=r_{0}\}}\Rs(z,\left(\mathsf{M}_{\l}\H\right)^{\star})\d z, \qquad |\l| \leq \delta_{0}, \mathrm{Re}\l \geq0$$
it holds that
$$\lim_{\l\to0}\|\mathsf{P}(\l)^{\star}-\mathsf{P}(0)^{\star}\|_{\mathscr{B}(L^{\infty}(\Gamma_{+},\d\mu_{+}))}=0.$$
Set 
$$\varphi_{\l}:=\mathsf{P}(\l)\varphi_{0}, \qquad \l \in \C_{+}.$$
Since $\varphi_{\l}$ converges to $\mathsf{P}(0)\varphi_{0}=\varphi_{0} \neq 0$, we get that $\varphi_{\l} \neq 0$ for $\l$ small enough and, since $\nu(\l)$ is algebraically simple, $\varphi(\l)$ is an eigenfunction of $\mathsf{M}_{\l}\H$ for $|\l|$ small enough. In the same way, for $|\l|$ small enough,
$$\varphi^{\star}_{\l}:=\mathsf{P}(\l)^{\star}\varphi^{\star}_{0} \longrightarrow \mathsf{P}(0)^{\star}\varphi^{\star}_{0}=1$$
as $\l \to 0$ and $\varphi^{\star}_{\l}$ is an eigenfunction of $\left(\mathsf{M}_{\l}\H\right)^{\star}$ associated to the eigenvalue $\nu(\l).$
\end{proof}

From now, we define $\delta >0$ small enough so that the rectangle
$$\mathcal{C}_{\delta}:=\{\l \in \C\;;\;0 \leq \mathrm{Re}\l \leq \delta\,,\,|\mathrm{Im}\l|\leq\delta\} \subset \{\l \in \C\;;\;|\l| < \delta_{0}\}$$
where $\delta_{0}$ is introduced in the previous Proposition \ref{prop:eigenMLH}. 
\begin{lemme}\label{lem:p'0}
The mapping 
$$\l \in \mathcal{C}_{\delta} \longmapsto \mathsf{P}(\l) \in \mathscr{B}(\lp)$$
is differentiable with
$$\mathsf{P}'(0)=-\frac{1}{2i\pi}\oint_{\{|z-1|=r_{0}\}}\Rs(z,\mathsf{M}_{0}\H)(\tau_{-}\mathsf{M}_{0}\H)\Rs(z,\mathsf{M}_{0}\H)\d z.$$
More generally, for any $\eta \in (-\delta,\delta)$,
$$\frac{\d}{\d\eta}\mathsf{P}(i\eta)=-\frac{1}{2i\pi}\oint_{\{|z-1|=r_{0}\}}\Rs(z,\mathsf{M}_{i\eta}\H)\left(\dfrac{\d}{\d\eta}\mathsf{M}_{i\eta}\H\right)\Rs(z,\mathsf{M}_{i\eta}\H)\d z.$$
\end{lemme}
\begin{proof} The only difficulty is to prove the differentiability on the imaginary axis. As soon as $z \notin \mathfrak{S}\left(\mathsf{M}_{\l}\H\right)$ for any $\l \in \mathcal{C}_{\delta}$, one has
$$\dfrac{\d}{\d\l}\Rs(z,\mathsf{M}_{\l}\H)=-\Rs(z,\mathsf{M}_{\l}\H)\left(\frac{\d}{\d\l}\mathsf{M}_{\l}\H\right)\Rs(z,\mathsf{M}_{\l}\H),$$
so that
$$\dfrac{\d}{\d\l}\mathsf{P}(\l)=-\frac{1}{2i\pi}\oint_{\{|z-1|=r_{0}\}}\Rs(z,\mathsf{M}_{\l}\H)\left(\frac{\d}{\d\l}\mathsf{M}_{\l}\H\right)\Rs(z,\mathsf{M}_{\l}\H)\d z\qquad \forall \l \in \mathcal{C}_{\delta}$$
and, since $\lim_{\l \to 0}\frac{\d}{\d\l}\left(\mathsf{M}_{\l}\H\right)=-\lim_{\l\to0}\left(\tau_{-}\mathsf{M}_{\l}\H\right)=-\tau_{-}\mathsf{M}_{0}\H$ we easily get the differentiability in $0.$ The same computations also give
$$\dfrac{\d}{\d\eta}\mathsf{P}(\e+i\eta)=-\frac{1}{2i\pi}\oint_{\{|z-1|=r_{0}\}}\Rs(z,\mathsf{M}_{\e+i\eta})\left(\frac{\d}{\d\eta}\mathsf{M}_{\e+i\eta}\H\right)\Rs(z,\mathsf{M}_{\e+i\eta}\H)\d z, \qquad \forall \eta \in \R\setminus\{0\}.$$
Using now Prop. \ref{prop:derMeis}  which asserts that $\tfrac{\d}{\d\eta}\mathsf{M}_{\e+i\eta}\H$ converges to $\tfrac{\d}{\d\eta}\mathsf{M}_{i\eta}\H$ as $\e\to0^{+}$ uniformly with respect to $\eta$, we deduce the second part of the Lemma. 
\end{proof}

We can complement the above result with the following:
\begin{lemme}\label{lem:deriv} With the notations of Proposition \ref{prop:eigenMLH}, the function $\l \in \mathcal{C}_{\delta} \mapsto \nu(\l)$ is differentiable with derivative $\nu'(\l)$ such that the limit
$$\nu'(0)=\lim_{\l\to0}\nu'(\l)$$
exists and is given by
$$\nu'(0)=-\int_{\Gamma_{+}}\tau_{-}(x,v)\varphi_{0}(x,v)\d \mu_{+}(x,v)<0.$$
\end{lemme}
\begin{proof} Recall that we introduced in the proof of Proposition \ref{prop:eigenMLH} the functions
$$\varphi_{\l}=\mathsf{P}(\l)\varphi_{0}, \qquad \varphi_{\l}^{\star}=\mathsf{P}(\l)^{\star}\varphi_{0}^{\star}, \qquad \l \in \mathcal{C}_{\delta}$$
which are such that $\lim_{\l \to 0}\varphi_{\l}=\varphi_{0}$ and $\lim_{\l\to0}\varphi_{\l}^{\star}=\varphi^{\star}_{0}=\mathbf{1}_{\Gamma_{+}}$. Introducing the duality bracket $\langle\cdot,\cdot\rangle$ between $\lp$ and its dual $(\lp)^{\star}=L^{\infty}(\Gamma_{+},\d\mu_{+})$, we have in particular
$$\lim_{\l\to0}\langle \varphi_{\l},\varphi_{\l}^{\star}\rangle=\langle \varphi_{0},\varphi_{0}^{\star}\rangle=\int_{\Gamma_{+}}\varphi_{0}\d\mu_{+}=1.$$
Moreover, the mappings $\l \in \mathcal{C}_{\delta}\mapsto \varphi_{\l} \in \lp$ and $\l\in \mathcal{C}_{\delta} \mapsto \varphi_{\l}^{\star} \in (\lp)^{\star}$ are differentiable with 
$$\dfrac{\d}{\d\l}\varphi_{\l}=\dfrac{\d}{\d\l}\mathsf{P}(\l)\varphi_{0}, \qquad \dfrac{\d}{\d\l}\varphi_{\l}^{\star}=\dfrac{\d}{\d\l}\mathsf{P}(\l)^{\star}\varphi_{0}^{\star}.$$
Since 
$$\mathsf{M}_{\l}\H\varphi_{\l}=\nu(\l)\varphi_{\l}$$ 
so that $\langle \mathsf{M}_{\l}\H\varphi_{\l}, \varphi_{\l}^{\star}\rangle=\nu(\l)\langle\varphi_{\l},\varphi_{\l}^{\star}\rangle$
we deduce first that $\l \in \mathcal{C}_{\delta} \mapsto \nu(\l)$ is differentiable and, differentiating the above identity yields
$$\dfrac{\d}{\d\l}\left(\mathsf{M}_{\l}\H\varphi_{\l}\right)=\nu'(\l)\varphi_{\l}+\nu(\l)\dfrac{\d}{\d\l}\varphi_{\l}.$$
Computing the derivatives and multiplying with $\varphi_{\l}^{\star}$ and integrating over $\Gamma_{+}$ we get
$$\langle \left(\tfrac{\d}{\d\l}\mathsf{M}_{\l}\H\right)\varphi_{\l} + \mathsf{M}_{\l}\H\tfrac{\d}{\d\l}\varphi_{\l},\varphi_{\l}^{\star}\rangle =
\nu'(\l)\langle\varphi_{\l},\varphi_{\l}^{\star}\rangle + \nu(\l)\langle \tfrac{\d}{\d\l}\varphi_{\l},\varphi_{\l}^{\star}\rangle.$$
Using that $\tfrac{\d}{\d\l}\mathsf{M}_{\l}\H=-\tau_{-}\mathsf{M}_{\l}\H$ whereas 
$$\langle \mathsf{M}_{\l}\H\tfrac{\d}{\d\l}\varphi_{\l},\varphi_{\l}^{\star}\rangle=\langle\tfrac{\d}{\d\l}\varphi_{\l},(\mathsf{M}_{\l}\H)^{\star}\varphi_{\l}^{\star}\rangle=\nu(\l)\langle \tfrac{\d}{\d\l}\varphi_{\l},\varphi_{\l}^{\star}\rangle$$ we obtain
$$-\langle \tau_{-}\mathsf{M}_{\l}\H\varphi_{\l},\varphi_{\l}^{\star}\rangle + \nu(\l)\langle \tfrac{\d}{\d\l}\varphi_{\l},\varphi_{\l}^{\star}\rangle
=\nu'(\l)\langle \langle\varphi_{\l},\varphi_{\l}^{\star}\rangle + \nu(\l)\langle \tfrac{\d}{\d\l}\varphi_{\l},\varphi_{\l}^{\star}\rangle.$$
Thus
$$-\langle \tau_{-}\mathsf{M}_{\l}\H\varphi_{\l},\varphi_{\l}^{\star}\rangle=\nu'(\l)\langle \varphi_{\l},\varphi_{\l}^{\star}\rangle, \qquad \forall \l \in \mathcal{C}_{\delta}.$$
Letting $\l \to 0$, we get that
$$\lim_{\l\to0}\nu'(\l)=-\langle \tau_{-}\mathsf{M}_{0}\H\varphi_{0},\varphi_{0}^{\star}\rangle$$
which is the desired result since $\mathsf{M}_{0}\H\varphi_{0}=\varphi_{0}$. \end{proof}

\subsubsection{Existence of the boundary function at $0$}

Recall that, to prove the existence of the limit $\lim_{\e\to0^{+}}\Rs(\e+i\eta,\T_{\H})f$ for $\eta \neq 0$, we need that $f \in \X_{1}.$ For the delicate case $\eta=0$, we will need the additional assumption that $f$ has zero mean:
\begin{equation}\label{eq:0mean}
\varrho_{f}:=\int_{\Omega\times V}f(x,v)\d x \otimes \bm{m}(\d v)=0.\end{equation}
Namely,
\begin{lemme}\label{lem:convZk}
Let $k \in \N$, $k \leq N_{\H}.$ Then, for any $f \in \X_{k+1}$ satisfying \eqref{eq:0mean}, the limit
$$\lim_{\e\to0^{+}}\Rs(\e+i\eta,\T_{\H})f$$
exists in $\X_{k}$ \emph{uniformly} on some neighbourhood of $0$. We denote $\mathbf{R}_{f}(\eta)$ this limit.
\end{lemme}
\begin{proof} According to Lemma \ref{lem:To}, if $f \in \X_{k+1}$ then
$$\lim_{\e\to0+}\Rs(\e+i\eta,\T_{0})f$$
exists in $\X_{k}$ \emph{uniformly} on some neighbourhood of $0$. So, as in the proof of Lemma \ref{lem:converTH}, it is enough to prove 
the convergence in $\X_{k}$ of 
$$\mathrm{R}(\e+i\eta)f=\mathsf{\Xi}_{\e+\,i\eta}\H\Rs(1,\mathsf{M}_{\e+i\eta}\H)\mathsf{G}_{\e+i\eta}f.$$
This convergence will hold thanks to \eqref{eq:0mean}.   
We split $\Rs(\e+i\eta)f$ as follows
\begin{multline*}
\mathrm{R}(\e+i\eta)f=\mathsf{\Xi}_{\e+i\eta}\H\Rs(1,\mathsf{M}_{\e+i\eta}\H)(\mathsf{I-P}(\e+i\eta))\mathsf{G}_{\e+i\eta}f\\
+\mathsf{\Xi}_{\e+i\eta}\H\Rs(1,\mathsf{M}_{\e+i\eta}\H)\mathsf{P}(\e+i\eta)\mathsf{G}_{\e+i\eta}f\\
=\mathsf{\Xi}_{\e+i\eta}\H\Rs(1,\mathsf{M}_{\e+i\eta}\H(\mathsf{I-P}(\e+i\eta)))\mathsf{G}_{\e+i\eta}f\\
+\mathsf{\Xi}_{\e+i\eta}\H\Rs(1,\mathsf{M}_{\e+i\eta}\H\mathsf{P}(\e+i\eta))\mathsf{G}_{\e+i\eta}f\end{multline*}
since $\mathsf{P}(\l)$ commutes with $\mathsf{M}_{\l}\H$. Notice that, with the notations of Proposition \ref{prop:eigenMLH}, 
$$\mathfrak{S}\left(\mathsf{M}_{\e+i\eta}\H\left[\mathsf{I-P}(\e+i\eta)\right]\right) \subset \{z \in \C\;;\;|z| <r\}$$
so that
$$r_{\sigma}\left(\mathsf{M}_{\e+i\eta}\H(\mathsf{I-P}(\e+i\eta))\right) \leq r <1.$$
One has then, for $r < r'< r_{0}$
$$\mathsf{I-P}(\e+i\eta))=\frac{1}{2i\pi}\oint_{\{|z|=r'\}}\Rs(z,\mathsf{M}_{\e+i\eta}\H)\d z$$
so that $\lim_{\e\to0^{+}}\mathsf{I-P}(\e+i\eta))=\mathsf{I-P}(i\eta))$ in $\mathscr{B}(\lp)$ uniformly with respect to $|\eta| < \delta.$ Consequently,
\begin{equation}\label{eq:convI-P}
\lim_{\e\to0^{+}}\sup_{|\eta| \leq \delta}\left\|\Rs(1,\mathsf{M}_{\e+i\eta}\H(\mathsf{I-P}(\e+i\eta)))\mathsf{G}_{i\eta}f-\Rs(1,\mathsf{M}_{i\eta}\H(\mathsf{I-P}(i\eta)))\mathsf{G}_{i\eta}f\right\|_{\lp}=0.\end{equation}
On the other hand, 
$$\Rs(1,\mathsf{M}_{\e+i\eta}\H\mathsf{P}(\e+i\eta))\mathsf{G}_{\e+i\eta}f=\frac{1}{1-\nu(\e+i\eta)}\mathsf{P}(\e+i\eta)\mathsf{G}_{\e+i\eta}f$$
so that, by the continuity of $\l \in \mathcal{C}_{\delta} \mapsto \nu(\l)$ gives easily that
$$\lim_{\e\to0^{+}}\Rs(1,\mathsf{M}_{\e+i\eta}\H\mathsf{P}(\e+i\eta))\mathsf{G}_{\e+i\eta}f=\frac{1}{1-\nu(i\eta)}\mathsf{P}(i\eta)\mathsf{G}_{i\eta}f, \qquad \forall \eta \neq 0$$
where the limit is meant in $\lp.$ Whenever $\eta=0$, we have
$$\Rs(1,\mathsf{M}_{\e}\H\mathsf{P}(\e))\mathsf{G}_{\e}f=\frac{1}{1-\nu(\e)}\mathsf{P}(\e)\mathsf{G}_{\e}f=\frac{\e}{1-\nu(\e)}\frac{\mathsf{P}(\e)\mathsf{G}_{\e}f-\mathsf{P}(0)\mathsf{G}_{0}f}{\e}$$
where we used the fact that, since $\varrho_{f}=0$ and $\mathsf{G}_{0}$ is stochastic, one has 
$$\int_{\Gamma_{+}}\mathsf{G}_{0}f\d\mu_{+}=0 \qquad \text{ so } \quad \mathsf{P}(0)\mathsf{G}_{0}f=0.$$ 
As already seen, the derivative $\mathsf{G}'(0)f$ exists since $f \in \X_{1}$ and therefore, by virtue of Lemma \ref{lem:p'0}, 
$$\lim_{\e \to 0^{+}}\frac{\mathsf{P}(\e)\mathsf{G}_{\e}f-\mathsf{P}(0)\mathsf{G}_{0}f}{\e}=\mathsf{P}'(0)\mathsf{G}_{0}f+\mathsf{P}(0)\mathsf{G}'_{0}f.$$
According to Lemma \ref{lem:deriv}, 
$$\lim_{\e\to0^{+}}\frac{\e}{1-\nu(\e)}=-\frac{1}{\nu'(0)} > 0$$
so that
$$\lim_{\e\to0^{+}}\Rs(1,\mathsf{M}_{\e}\H\mathsf{P}(\e))\mathsf{G}_{\e}f=-\frac{1}{\nu'(0)}\left[\mathsf{P}'(0)\mathsf{G}_{0}f+\mathsf{P}(0)\mathsf{G}'_{0}f\right].$$
Finally, we obtain that $\lim_{\e\to0^{+}}\Rs(1,\mathsf{M}_{\e}\H)\mathsf{G}_{\e}f$ exists in $\lp$ and is given by
$$\Rs(1,\mathsf{M}_{0}\H(\mathsf{I-P}(0)))\mathsf{G}_{0}f-\frac{1}{\nu'(0)}\left[\mathsf{P}'(0)\mathsf{G}_{0}f+\mathsf{P}(0)\mathsf{G}'_{0}f\right].$$
We just  proved that, for any $\eta \in\R$, 
$$\lim_{\e\to0^{+}}\Rs(1,\mathsf{M}_{\e+i\eta})\mathsf{G}_{\e+i\eta}f$$
exists in $\X_{0}$ and is given by 
$$\Rs_{f}(\eta):=\begin{cases}
\Rs(1,\mathsf{M}_{i\eta}\H)\mathsf{G}_{i\eta}f \qquad &\text{ if } \eta \neq 0\\
\Rs\left(1,\mathsf{M}_{0}\H\left(\mathsf{I-P}(0)\right)\right)\mathsf{G}_{0}f  -\frac{1}{\nu'(0)}\left[\mathsf{P}'(0)\mathsf{G}_{0}f+\mathsf{P}(0)\mathsf{G}'_{0}f)\right] \qquad &\text{ if } \eta =0.\end{cases}
$$
Let us prove that the convergence is uniform with respect to $|\eta| \leq \delta.$  According to \eqref{eq:convI-P}, we only need to prove that the convergence
$$\lim_{\e\to0^{+}}\Rs(1,\mathsf{M}_{\e+i\eta}\H\mathsf{P}(\e+i\eta))\mathsf{G}_{\e+i\eta}f$$
towards
$$F(\eta)=\begin{cases}
\Rs(1,\mathsf{M}_{i\eta}\H\mathsf{P}(i\eta))\mathsf{G}_{i\eta}f \qquad &\text{ if } \eta \neq 0\\
-\frac{1}{\nu'(0)}\left[\mathsf{P}'(0)\mathsf{G}_{0}f+\mathsf{P}(0)\mathsf{G}'_{0}f)\right] \qquad &\text{ if } \eta =0.\end{cases}
$$
is uniform with respect to $|\eta| < \delta.$ We argue by contradiction, assuming that there exist $c >0$, a sequence $(\e_{n})_{n} \subset (0,\infty)$ converging to $0$ and a sequence $(\eta_{n})_{n} \subset (-\delta,\delta)$ such that
\begin{equation}\label{eq:absurd}
\left\|\Rs(1,\mathsf{M}_{\e_{n}+i\eta_{n}}\H\mathsf{P}(\e_{n}+i\eta_{n}))\mathsf{G}_{\e_{n}+i\eta_{n}}f-F(\eta_{n})\right\|_{\lp} \geq c >0.\end{equation}
Up to consider subsequence if necessary, we can assume without loss of generality that $\lim_{n}\eta_{n}=\eta_{0}$ with $|\eta_{0}| \leq \delta.$ First, one sees that then $\eta_{0}=0$ since the convergence of $\Rs(1,\mathsf{M}_{\e+i\eta}\H\mathsf{P}(\e+i\eta)\mathsf{G}_{\e+i\eta}f$ to $F(\eta)$ is actually uniform in any neighbourhood around $\eta_{0} \neq 0$ (see \eqref{eq:convunifr}). Because $\eta_{0}=0$, defining $\lambda_{n}:=\e_{n}+i\eta_{n}$, $n \in \N$, the sequence $(\l_{n})_{n}\subset \mathcal{C}_{\delta}$ is converging to $0$. Now, as before,
$$\Rs(1,\mathsf{M}_{\l_{n}}\H\mathsf{P}_{\l_{n}})\mathsf{G}_{\l_{n}}f=\frac{\l_{n}}{1-\nu(\l_{n})}\frac{\mathsf{P}(\l_{n})\mathsf{G}_{\l_{n}}f-\mathsf{P}(0)\mathsf{G}_{0}f}{\l_{n}}, \qquad n\in \N$$
with
$$\lim_{n\to \infty}\frac{\l_{n}}{1-\nu(\l_{n})}=-\frac{1}{\nu'(0)}, \qquad \lim_{n\to\infty}
\frac{\mathsf{P}(\l_{n})\mathsf{G}_{\l_{n}}f-\mathsf{P}(0)\mathsf{G}_{0}f}{\l_{n}}=\left[\mathsf{P}'(0)\mathsf{G}_{0}f+\mathsf{P}(0)\mathsf{G}'_{0}f\right].$$
Therefore, 
$$\lim_{n\to\infty}\Rs(1,\mathsf{M}_{\l_{n}}\H\mathsf{P}(\l_{n}))\mathsf{G}_{\l_{n}}f=-\frac{1}{\nu'(0)}\left[\mathsf{P}'(0)\mathsf{G}_{0}f+\mathsf{P}(0)\mathsf{G}'_{0}f\right].$$
One also has
$$F(\eta_{n})=\Rs(1,\mathsf{M}_{i\eta_{n}}\H\mathsf{P}({i\eta_{n}}))\mathsf{G}_{i\eta_{n}}f=\frac{i\eta_{n}}{1-\nu(i\eta_{n})}\frac{\mathsf{P}(i\eta_{n})\mathsf{G}_{i\eta_{n}}f-\mathsf{P}(0)\mathsf{G}_{0}f}{i\eta_{n}}, \qquad n\in \N$$
so that $F(i\eta_{n})$ has the same limit $-\frac{1}{\nu'(0)}\left[\mathsf{P}'(0)\mathsf{G}_{0}f+\mathsf{P}(0)\mathsf{G}'_{0}f\right]$ as $n \to \infty$. This contradicts \eqref{eq:absurd}. Therefore, one has
$$\lim_{\e\to0^{+}}\sup_{|\eta|<\delta}\left\|\Rs(1,\mathsf{M}_{\e+i\eta}\H)\mathsf{G}_{\e+i\eta}f-\Rs_{f}(\eta)\right\|_{\lp}=0.$$
As in the proof of Lemma \ref{lem:converTH}, using now \eqref{eq:MeisHk} in Proposition \ref{propo:convK}, we deduce the desired uniform convergence in $\X_{k}$ with $\mathbf{R}_{f}(\eta)=\mathbf{R}_{f}^{0}(\eta)+\mathsf{\Xi}_{i\eta}\H\Rs_{f}(\eta).$
\end{proof}

\subsubsection{Existence of the boundary function}

Combining Lemma \ref{lem:convZk} with  Lemma \ref{lem:converTH} we get the following
\begin{theo}\label{theo:existtrace} For any $k \in \N$, $k \leq N_{\H}$ and any $f \in \X_{k+1}$ satisfying \eqref{eq:0mean}, the limit
$$\lim_{\e\to0^{+}}\Rs(\e+i\eta,\T_{\H})f$$
exists in $\X_{k}$ uniformly with respect to $\eta$ on any compact of $\R$. The limit is denoted $\mathbf{R}_{f}(\eta)$. The mapping
$$\eta \in \R \longmapsto \mathbf{R}_{f}(\eta) \in \X_{k}$$
is continuous and is uniformly bounded in $\X_{k}$, i.e. $\sup_{\eta\in \R}\|\mathbf{R}_{f}(\eta)\|_{\X_{k}} <\infty.$
\end{theo}
\begin{proof} Let $[a,b]$ be a compact interval of $\R$ and $a \leq \eta_{0} \leq b.$ We know from Lemmas \ref{lem:convZk}--\ref{lem:converTH} and Lemma \ref{lem:To} that the convergence of $\Rs(\e+i\eta,\T_{\H})f$ to $\mathbf{R}_{f}(\eta)$ is uniform on some neighbourhood of $\eta_{0}$. Covering $[a,b]$ with a finite number of such neighbourhoods, we get the first part of the Theorem. The mapping $\eta \in \R \mapsto \mathbf{R}_{f}(\eta) \in \X_{k}$ is clearly continuous as the uniform limit of the family of continuous mappings $\eta \in \R \mapsto \Rs(\e+i\eta,\T_{\H})f$, $\e >0$. To prove the limit is uniformly bounded, we see from the first part that we only need to prove that, for $R >0$ large enough,
$$\sup_{|\eta| >R}\|\mathbf{R}_{f}(\eta)\|_{\X_{k}} < \infty.$$
For any $\eta\neq0$, one has
$$\mathbf{R}_{f}(\eta)=\mathbf{R}_{f}^{0}(\eta)+\mathsf{\Xi}_{i\eta}\H\Rs(1,\mathsf{M}_{i\eta}\H)\mathsf{G}_{i\eta}f$$
so that
\begin{multline*}
\|\mathbf{R}_{f}(\eta)\|_{\X_{k}} \leq \|\mathbf{R}_{f}^{0}(\eta)\|_{\X_{k}} + \left\|\mathsf{\Xi}_{i\eta}\H\Rs(1,\mathsf{M}_{i\eta}\H)\mathsf{G}_{i\eta}f\right\|_{\X_{k}}
\\
\leq \|f\|_{\X_{k+1}}+ \|\mathsf{\Xi}_{0}\|_{\mathscr{B}(\Y_{k}^{-},\X_{k})}\|\H\|_{\mathscr{B}(\lp,\Y_{k}^{-})}\left\|\Rs(1,\mathsf{M}_{i\eta}\H)\right\|_{\mathscr{B}(\lp)}\|f\|_{\X_{0}}
\end{multline*}
where we used that $|\mathsf{\Xi}_{i\eta}| \leq \mathsf{\Xi}_{0}$ and Remark \ref{nb:boundedR0}. To prove the result, we only need to show that $\|\Rs(1,\mathsf{M}_{i\eta}\H)\|_{\mathscr{B}(\lp)}$ is bounded for $|\eta|$ large enough. Recall (see \eqref{eq:poweriml}) that there exists $\ell \in \N$ such that
$$\lim_{|\eta|\to\infty}\left\|\left(\mathsf{M}_{i\eta}\H\right)^{\ell}\right\|_{\mathscr{B}(\lp)}=0,$$
and, arguing exactly as in Corollary \ref{cor:ResMei}, one proves easily that, for $|\eta|$ large enough $\Rs(1,\mathsf{M}_{i\eta}\H)=\sum_{n=0}^{\infty}\sum_{j=0}^{\ell-1}\left(\mathsf{M}_{i\eta}\H\right)^{n\ell+j}$ from which
$$\|\Rs(1,\mathsf{M}_{i\eta}\H)\|_{\mathscr{B}(\lp)} \leq \ell \sum_{n=0}^{\infty}\|\left(\mathsf{M}_{i\eta}\H\right)^{\ell}\|_{\mathscr{B}(\lp)}^{n}=\frac{\ell}{1-\left\|\left(\mathsf{M}_{i\eta}\H\right)^{\ell}\right\|_{\mathscr{B}(\lp)}}$$
which gives the conclusion.
\end{proof}

\section{Regularity of the boundary function}\label{sec:regF}

\subsection{Continuity and qualitative convergence theorem} Before proving Theorem \ref{theo:qualit} we need the following technical result
\begin{lemme}\label{lem:THX1} For any $k \in \N$ with $k \leq \mathsf{N}_{\H}$,  $\D(\T_{\H}) \cap \X_{k}$ is dense in $\X_{k}$.
\end{lemme}
\begin{proof} Let  $k \leq \mathsf{N}_{\H}$ be fixed. Recall that the resolvent $\Rs(\l,\T_{\H})$ leaves invariant $\X_{k}$ $(\l >0)$. To prove the result, let us show that
\begin{equation*}
\label{eq:RlTf}
\lim_{\l\to\infty}\left\|\l\Rs(\l,\T_{\H})f-f\right\|_{\X_{k}}=0 \qquad \forall f \in \X_{k}\end{equation*}
which will clearly prove the result since $\l\Rs(\l,\T_{\H})f \in \D(\T_{\H})\cap \X_{k}$ for any $\l >0.$ Let $f \in \X_{k}$ be fixed. A general property of the resolvent implies that the above convergence holds true in $\X_{0}$ i.e.
$$\lim_{\l\to\infty}\left\|\l\Rs(\l,\T_{\H})f-f\right\|_{\X_{0}}=0.$$
To prove the claim, one has to prove that
\begin{equation}\label{eq:limvarpi}
\lim_{\l \to \infty}\left\|\left(\l\Rs(\l,\T_{\H})f-f\right)\varpi_{k}\right\|_{\X_{0}}=0\end{equation}
where we recall that $\varpi_{k}(v)=\min(1,|v|^{-k})$ has been introduced in Remark \ref{nb:varpi}. Let $\l >1$ be fixed. We already saw (see the proof of Lemma \ref{lem:To}) that $\varpi_{k}$ commutes with $\Rs(\l,\T_{0})$ so,  
$$\varpi_{k}\left(\l\,\Rs(\l,\T_{\H})f-f\right)=\l\Rs(\l,T_{0})\left(\varpi_{k}f\right) + \varpi_{k}\l\,\mathsf{\Xi}_{\l}\H\Rs(1,\mathsf{M}_{\l}\H)\mathsf{G}_{\l}f$$
and, because $\l\Rs(\l,\T_{0})$ converges strongly to the identity (in $\X_{0})$ as $\l \to \infty$, for $f \in \X_{k}$ we have
$$\lim_{\l\to\infty}\left\|\l\Rs(\l,\T_{0})\left(\varpi_{k}f\right)-\varpi_{k}f\right\|_{\X_{0}}=0.$$
So to prove \eqref{eq:limvarpi}, it is enough to show that
\begin{equation}\label{eq:Xil0}
\lim_{\l\to\infty}\left\|\varpi_{k}\l\mathsf{\Xi}_{\l}\H\Rs(1,\mathsf{M}_{\l}\H)\mathsf{G}_{\l}f\right\|_{\X_{0}}=0.\end{equation}
Because $\varpi_{k}$ is independent of $x$, one has, $\varpi_{k}\mathsf{\Xi}_{\l}\H=\mathsf{\Xi}_{\l}\left(\varpi_{k}\H\right)$ and
$$\left\|\varpi_{k}\l\mathsf{\Xi}_{\l}\H\Rs(1,\mathsf{M}_{\l}\H)\mathsf{G}_{\l}f\right\|_{\X_{0}} \leq \|\l\,\mathsf{\Xi}_{\l}\|_{\mathscr{B}(\lm,\X_{0})}\,
\left\|\varpi_{k}\H\right\|_{\mathscr{B}(\lp,\lm)}\,\left\|\Rs(1,\mathsf{M}_{\l}\H)\right\|_{\mathscr{B}(\lp)}\,\|\mathsf{G}_{\l}f\|_{\lp}.$$
For $\l > 1$, one has $\left\|\l\mathsf{\Xi}_{\l}\right\|_{\mathscr{B}(\lm,\X_{0})} \leq 1$ from \eqref{eq:XiLRL} while
$$\left\|\Rs(1,\mathsf{M}_{\l}\H)\right\|_{\mathscr{B}(\lp)} \leq \left\|\Rs(1,\mathsf{M}_{1}\H)\right\|_{\mathscr{B}(\lp)} \quad \text{ 
and } \quad  \|\varpi_{k}\H\|_{\mathscr{B}(\lp,\lm)}=\|\H\|_{\mathscr{B}(\lp,\Y_{k}^{-})} < \infty$$
since $k \leq \mathsf{N}_{\H}$. Therefore, there exists a positive constant $C_{k} >0$ independent of $f$ and $\l$ such that
$$\left\|\varpi_{k}\l\mathsf{\Xi}_{\l}\H\Rs(1,\mathsf{M}_{\l}\H)\mathsf{G}_{\l}f\right\|_{\X_{0}} \leq C_{k}\,\left\|\mathsf{G}_{\l}f\right\|_{\lp}, \qquad \forall \l >1.$$
One has
$$\|\mathsf{G}_{\l}f\|_{\lp} \leq \int_{\Gamma_{+}}\d\mu_{+}(x,v)\int_{0}^{\tau_{-}(x,v)}\,|f(x-sv,v)|e^{-\l\,s}\d s$$
and, since $t_{+}(x-sv,v)=s$ for any $(x,v) \in \Gamma_{+}$, one has
$$\|\mathsf{G}_{\l}f\|_{\lp} \leq \int_{\Gamma_{+}}\d\mu_{+}(x,v)\int_{0}^{\tau_{-}(x,v)}\,|f(x-sv,v)|e^{-\l\,t_{+}(x-sv,v)}\d s$$
which, thanks to \eqref{10.47},  yields
$$\|\mathsf{G}_{\l}f\|_{\lp} \leq \int_{\Omega\times V}e^{-\l\,t_{+}(x,v)}\,|f(x,v)|\d x\bm{m}(\d v).$$
The dominated convergence theorem implies then that 
$$\lim_{\l\to\infty}\|\mathsf{G}_{\l}f\|_{\lp}=0$$
which gives \eqref{eq:Xil0} and proves the result.\end{proof}
We have all in hands to use directly Ingham's Theorem \ref{theo:ingham} to prove Theorem \ref{theo:qualit}.
\begin{proof}[Proof of Theorem \ref{theo:qualit}] Since $\X_{1}$ is dense in $\X_{0}$, to prove the result it is enough to consider $f \in \X_{1}.$ Using then Lemma \ref{lem:THX1}, we can assume without loss of generality that 
$$f \in \D(\T_{\H}) \cap \X_{1}.$$
Replacing $f$ with $f-\varrho_{f}\Psi_{\H} \in \X_{1}$ (recall that $\Psi_{\H} \in \D(\T_{\H}) \cap \X_{1}$, see Remark \ref{nb:PsiXn}) we can assume without loss of generality that $\varrho_{f}=0.$ We apply then Ingham Theorem \ref{theo:ingham} with 
$$g(t)=U_{\H}(t)f \qquad t \geq 0.$$
Since $\|U_{\H}(t)\|_{\mathscr{B}(\X_{0})} \leq 1$ and $f \in \D(\T_{\H})$ one has
$$g(\cdot)   \in \mathrm{BUC}(\R_{+},\X_{0})$$
and
$$\widehat{g}(\alpha+is)=\Rs(\alpha+is,\T_{\H})f \qquad \text{ and } \quad F(s)=\mathbf{R}_{f}(s), \qquad \alpha >0,\;s \in \R.$$
From Theorem \ref{theo:existtrace}, $F(\cdot)$ is indeed the (weak) limit of $\widehat{g}(\alpha+i\cdot)$ as $\alpha \to 0^{+}$ and, since $F \in \mathcal{C}(\R)$, one has in particular $F \in L^{1}_{\mathrm{loc}}(\R,\X_{0}).$ We deduce directly from Ingham's Theorem that $g(\cdot) \in \mathcal{C}_{0}(\R_{+},\X_{0})$ which is the desired conclusion.\end{proof}

\subsection{Higher order regularity and proof of Theorem \ref{theo:main}}

In order to make the above convergence Theorem \ref{theo:qualit} quantitative, we need to estimate the derivatives of $\mathbf{R}_{f}(\eta)$ (i.e. estimate the derivatives of $\Rs(\l,\T_{\H})f$ along the imaginary axis under specific integrability properties of $f$.  We begin with showing that the boundary function is indeed regular under suitable integrability property of $f$.

To show $\mathbf{R}_{f}(\cdot)$ is regular for $f\in \X_{k}$, the key observation is the following general property of the resolvent
$$\dfrac{\d^{k}}{\d \l^{k}}\Rs(\l,\T_{\H}) =(-1)^{k}k!\Rs(\l,\T_{\H})^{k+1}, \qquad \l \in \C_{+}.$$
We introduce 
$$\X_{k}^{0}:=\{f \in \X_{k}\;;\;\varrho_{f}=0\}, \qquad k \in \N.$$
which is a closed subset of $\X_{k}$. Notice that, endowed with the $\X_{k}$-norm, $\X_{k}^{0}$ is a Banach space.  
Since 
$$\int_{\Omega \times V}U_{\H}(t)f\d x\otimes \bm{m}(\d v)=\int_{\Omega\times V}f\d x\otimes \bm{m}(\d v), \qquad \forall t \geq0, \quad f \in \X_{0}$$
one has
\begin{equation*}\begin{split}
\int_{\Omega\times V}\Rs(\l,\T_{\H})f\d x\otimes \bm{m}(\d v)&=\int_{\Omega\times V}\left(\int_{0}^{\infty}e^{-\l\,t}U_{\H}(t)f\d t\right)\d x\otimes \bm{m}(\d v)\\
&=\frac{1}{\l}\int_{\Omega\times V}f\d x\otimes \bm{m}(\d v), \qquad \forall \l \in \C_{+}\end{split}\end{equation*}
and therefore the resolvent and all its iterates $\Rs(\l,\T_{\H})^{k}$ $k \geq 0$ leave invariant $\X_{0}^{0}$.

\begin{lemme}\label{lem:UPS}  Assume that $\H$ satisfies Assumptions \ref{hypH}. For any $k \in \N$, $k \leq N_{\H}$ and any $f \in \X_{k}^{0}$ the limit
$$\lim_{\e\to 0^{+}}\left[\Rs(\e+i\eta,\T_{\H})\right]^{k}f:=\Upsilon_{k}(\eta)f$$
exists in $\X_{0}$ uniformly with respect to $\eta$ on any compact of $\R$. Moreover, the mapping
$$\eta \in \R \longmapsto \Upsilon_{k}(\eta)f$$
is continuous and bounded over $\R.$
\end{lemme}
\begin{proof} The proof is made by induction over $k \in \N$, $k \leq N_{\H}.$ For $k=1$, the result holds true by  Theorem \ref{theo:existtrace}. Let $k \in \N$, $k < N_{\H}$ and assume the result to be true for any $j\in \{1,\ldots,k\}.$ According to Banach-Steinhaus Theorem \cite[Theorem 2.2, p. 32]{brezis}, for any $j\in \{1,\ldots,k\}$ and any $c >0$,
\begin{equation}\label{eq:Cj}
M_{j}(c)=\sup\left\{\left\|\Rs(\e+i\eta,\T_{\H})^{j}\right\|_{\mathscr{B}(\X_{j}^{0},\X_{0})}\;;\;0 < \e < 1,\;|\eta| \leq c\right\} < \infty\end{equation}
whereas, for any $g \in \X_{j}^{0}$, the limit
\begin{equation}\label{eq:upsilon}
\lim_{\e\to0^{+}}\left[\Rs(\e+i\eta,\T_{\H})\right]^{j}g=:\Upsilon_{j}(\eta)g\end{equation}
exists in $\X_{0}$  uniformly with respect to $\eta$ on any compact of $\R$ with moreover
$$
\sup_{\eta \in \R}\left\|\Upsilon_{j}(\eta)g\right\|_{\X_{0}} < \infty \qquad \forall g \in \X_{j}^{0}.$$
Notice that, using Banach-Steinhaus Theorem again, this implies that $\Upsilon_{j}(\eta) \in \mathscr{B}(\X_{j}^{0},\X_{0})$ for any $j \in \{1,\ldots,k\}$ with
\begin{equation}\label{eq:inducbounde}\sup_{\eta\in \R}\left\|\Upsilon_{j}(\eta)\right\|_{\mathscr{B}(\X_{j}^{0},\X_{0})} < \infty.\end{equation}
Moreover, the mapping $\eta \in \R \mapsto \Upsilon_{j}(\eta)f \in \X_{0}$ is continuous over $\R$ as the uniform limit of the continuous functions $\eta \mapsto \left[\Rs(\e+i\eta,\T_{\H})\right]^{j}f \in \X_{0}.$ 
 Clearly, for any $\eta \in \R$, the mapping $g \in \X_{j}^{0} \mapsto \Upsilon_{j}(\eta)g$ is linear. Let $f \in \X_{k+1}^{0}$ be given. One has, for any $0 < \e < 1$ and $\eta \in \R$: 
\begin{multline*}
\left[\Rs(\e+i\eta,\T_{\H})\right]^{k+1}f-\Upsilon_{k}(\eta)\mathbf{R}_{f}(\eta)=\left[\Rs(\e+i\eta,\T_{\H})\right]^{k}\Rs(\e+i\eta,\T_{\H})f-\Upsilon_{k}(\eta)\mathbf{R}_{f}(\eta)\\
=\left(\left[\Rs(\e+i\eta,\T_{\H})\right]^{k}-\Upsilon_{k}(\eta)\right)\mathbf{R}_{f}(\eta)\\
+\left[\Rs(\e+i\eta,\T_{\H})\right]^{k}\left(\Rs(\e+i\eta,\T_{\H})f-\mathbf{R}_{f}(\eta)\right).\end{multline*}
Notice that, for $f \in \X_{k+1}^{0}$, $\Rs(\e+i\eta,\T_{\H})f \in \X_{k}^{0}$ and so is $\mathbf{R}_{f}(\eta) \in \X_{k}^{0}$. We can use then the above induction hypothesis \eqref{eq:Cj} first to obtain that
\begin{multline*}
\left\|\left[\Rs(\e+i\eta,\T_{\H})\right]^{k}\left(\Rs(\e+i\eta,\T_{\H})f-\mathbf{R}_{f}(\eta)\right)\right\|_{\X_{0}} \\
\leq M_{k}(c)\left\|\Rs(\e+i\eta,\T_{\H})f-\mathbf{R}_{f}(\eta)\right\|_{\X_{k}^{0}}, \qquad \forall 0 < \e <1, \quad |\eta| \leq c\end{multline*}
where the right-hand-side converges to $0$ as $\e\to 0^{+}$ uniformly on $\{|\eta| \leq c\}$ thanks to Theorem \ref{theo:existtrace}.
Moreover, applying \eqref{eq:upsilon} with $g=\mathbf{R}_{f}(\eta)$, we get
$$\lim_{\e \to 0^{+}}\sup_{|\eta| \leq c}\left\|\left(\left[\Rs(\e+i\eta,\T_{\H})\right]^{k}-\Upsilon_{k}(\eta)\right)\mathbf{R}_{f}(\eta)\right\|_{\X_{0}}=0$$
where we used the fact that the mapping $\eta \in \R \mapsto \mathbf{R}_{f}(\eta) \in \X_{k}^{0}$ is bounded and continuous according to Theorem \ref{theo:existtrace}. This is enough to prove that, for any $c >0$, 
$$\lim_{\e\to0^{+}}\sup_{|\eta| \leq c}\left\|\left[\Rs(\e+i\eta,\T_{\H})\right]^{k+1}f-\Upsilon_{k}(\eta)\mathbf{R}_{f}(\eta)\right\|_{\X_{0}}=0$$
which is the first part of the result with $\Upsilon_{k+1}(\eta)f=\Upsilon_{k}(\eta)\mathbf{R}_{f}(\eta)$. Consequently
$$\left\|\Upsilon_{k+1}(\eta)f\right\|_{\X_{0}} \leq \|\Upsilon_{k}(\eta)\|_{\mathscr{B}(\X_{k}^{0},\X_{0})}\,\|\mathbf{R}_{f}(\eta)\|_{\X_{k}}$$
and we deduce that
$$\sup_{\eta \in \R}\left\|\Upsilon_{k+1}(\eta)f\right\|_{\X_{0}} <\infty$$
according to the induction hypothesis \eqref{eq:inducbounde} and the fact that $\sup_{\eta\in\R}\|\mathbf{R}_{f}(\eta)\|_{\X_{k}} <\infty.$ This achieves the proof.
\end{proof}

\begin{theo}\label{theo:regul} Assume that $\H$ satisfy Assumptions \ref{hypH}. Let $0 \leq k < N_{\H}$ and $f \in \X_{k+1}$ with $\varrho_{f}=0$. Then, the boundary function 
$$\eta \in \R \longmapsto \mathbf{R}_{f}(\eta)$$
defined by Theorem \ref{theo:existtrace} belongs to $\mathcal{C}^{k}(\R\,;\,\X_{0})$ with 
$$\sup_{\eta\in \R}\max_{0\leq j\leq k}\left\|\frac{\d^{j}}{\d\eta^{j}}\mathbf{R}_{f}(\eta)\right\|_{\X_{0}} <\infty.$$
\end{theo}
\begin{proof} The result is true for $k=0$ according to Theorem \ref{theo:existtrace}. Notice now that, for $f \in \X_{k+1}^{0}$ and $0\leq j\leq k$
$$\dfrac{\d^{j}}{\d \eta^{j}}\Rs(\e+i\eta,\T_{\H})f=(-i)^{j}j!\left[\Rs(\e+i\eta,\T_{\H})\right]^{j+1}f \qquad \forall 0< \e<1,\qquad \eta \in \R.$$
From Lemma \ref{lem:UPS}, we deduce that, for any $c >0$,
$$\lim_{\e\to^{+}}\sup_{|\eta| \leq c}\left\|\dfrac{\d^{j}}{\d \eta^{j}}\Rs(\e+i\eta,\T_{\H})f-(-i)^{j}j!\Upsilon_{j+1}(\eta)f\right\|_{\X_{0}}=0$$
i.e. the first $k$ derivatives of $\eta \mapsto \Rs(\e+i\eta,\T_{\H})f$ converge as $\e\to0$ uniformly on any compact. This proves that the boundary function $\mathbf{R}_{f}(\eta)$ is of class $\mathcal{C}^{k}(\R,\X_{0})$ with 
$$\frac{\d^{j}}{\d\eta^{j}}\mathbf{R}_{f}(\eta)=(-i)^{j}j!\Upsilon_{j+1}(\eta)f, \qquad \forall 0 \leq j \leq k.$$
This last identity, together with Lemma \ref{lem:UPS} shows also that $\eta \mapsto \mathbf{R}_{f}(\eta)$ is uniformly bounded in $\mathcal{C}^{k}(\R,\X_{0}).$
\end{proof}

With this and Theorem \ref{theo:quanti}, we are in position to prove Theorem \ref{theo:main} stated in the Introduction.
\begin{proof}[Proof of Theorem \ref{theo:main}] Let $f \in \D(\T_{\H}) \cap \X_{k+1}.$ One has
$$U_{\H}(t)f-\mathbb{P}f=U_{\H}(t)\tilde{f}$$
with $\tilde{f}=f-\mathbb{P}f=f-\varrho_{f}\mathsf{\Psi}_{\H}$. Notice that $\varrho_{\tilde{f}}=0$ and, since $\mathsf{\Psi}_{\H} \in \D(\T_{\H}) \cap \X_{k+1}$ (see Remark \ref{nb:PsiXn}), it holds
$$\tilde{f} \in \D(\T_{\H}) \cap \X_{k+1}^{0}.$$
We introduce, as in Theorem \ref{theo:qualit},
$$g(t)=U_{\H}(t)\tilde{f} \qquad t \geq 0.$$
Because $\tilde{f} \in \D(\T_{\H})$ one has that $g(\cdot)   \in \mathrm{BUC}(\R_{+},\X_{0}) \cap \mathrm{Lip}(\R_{+};\X_{0})$ and
$$\widehat{g}(\e+i\eta)=\Rs(\alpha+i\eta,\T_{\H})\tilde{f} \qquad \text{ and } \quad F(\eta)=\mathbf{R}_{\tilde{f}}(\eta), \qquad \e >0,\;\eta \in \R.$$
According to Theorem \ref{theo:regul}, the boundary function $F$ is uniformly bounded in $\mathcal{C}^{k}(\R,\X_{0})$. We conclude thanks to Theorem \ref{theo:quanti} that there exists $C_{k} >0$ such that
$$\|g(t)\|_{\X_{0}} \leq C_{k}t^{-\frac{k}{2}},\qquad \forall t \geq 0$$
which is exactly the conclusion of Theorem \ref{theo:main}. \end{proof}

\appendix

\section{About Assumptions \ref{hypH}}\label{appen:REGU}

\subsection{Reminders about regular diffuse boundary operators}

We recall here that, in our previous contribution \cite{LMR}, we introduce a general class of diffuse boundary operators that we called \emph{regular}.  First, we begin with the definition of stochastic diffuse boundary operator.
\begin{defi}\phantomsection\label{defi:Hdiff} One says that $\mathsf{H} \in \mathscr{B}(\lp,\lm)$ is a \emph{stochastic diffuse boundary operator} if
\begin{equation}\label{eq:Hdif}
\H\,\psi(x,v)=\int_{\Gamma_{+}(x)}\bm{h}(x,v,v')\psi(x,v')\bm{\mu}_{x}(\d v'), \qquad (x,v) \in \Gamma_{-}, \quad \psi \in \lp\end{equation}
where the kernel $\bm{h}(\cdot,\cdot,\cdot)$ induces a field of nonnegative measurable functions
$$x \in \partial\Omega \longmapsto \bm{h}(x,\cdot,\cdot)$$
where 
$$\bm{h}(x,\cdot,\cdot)\::\:\Gamma_{-}(x) \times \Gamma_{+}(x) \to \R^{+}$$
is such that
$$\int_{\Gamma_{-}(x)}\bm{h}(x,v,v')\bm{\mu}_{x}(\d v)=1, \qquad \forall (x,v') \in \Gamma_{+}.$$
\end{defi}
Notice that we can identify $\H \in \mathscr{B}(\lp,\lm)$ to a field of integral operators
$$x \in \partial\Omega \longmapsto \mathsf{H}(x) \in \mathscr{B}(L^{1}(\Gamma_{+}(x),\Gamma_{-}(x))$$
by the formula
$$[\mathsf{H}\psi](x,v)=\left[\mathsf{H}(x)\psi(x,\cdot)\right](v)$$
where, for any $x \in \partial\Omega$
$$\mathsf{H}(x)\::\:\psi \in L^{1}(\Gamma_{+}(x)) \longmapsto \left[\mathsf{H}(x)\psi\right](v)=\int_{\Gamma_{+}(x)}\bm{h}(x,v,v')\psi(v')\bm{\mu}_{x}(\d v') \in L^{1}(\Gamma_{-}(x))$$
where
$$L^{1}(\Gamma_{\pm}(x))=L^{1}(\Gamma_{\pm}(x),\bm{\mu}_{x}(\d v)), \qquad \Gamma_{\pm}(x)=\{v \in V\;;\;(x,v) \in \Gamma_{\pm}\}, \quad x \in \partial\Omega.$$

We recall the class of diffuse boundary operators as introduced in \cite{LMR}.
\begin{defi}\phantomsection \label{defi:regul}
We say that a diffuse boundary operator $\mathsf{H} \in \mathscr{B}(\lp,\lm)$ is \emph{regular} if the family of operators 
$$\mathsf{H}(x) \in \mathscr{B}(L^{1}(\Gamma_{+}(x)),L^{1}(\Gamma_{-}(x))), \qquad x \in \partial\Omega$$
is collectively weakly compact in the sense that 
\begin{equation}\label{eq:Sm}
\lim_{m\to\infty}\sup_{v'\in \Gamma_{+}(x)}\int_{S_{m}(x,v')}\bm{h}(x,v,v')\,\bm{\mu}_{x}(\d v)=0, \qquad \forall x \in \partial\Omega\end{equation}
where the convergence is \emph{uniform} with respect to $x \in \partial\Omega$. Here above, for any $m \in \N$ and any $(x,v') \in \Gamma_{+}$, we set
$S_{m}(x,v')=\{v \in \Gamma_{-}(x)\;;\;|v| \geq m\} \cup \{v \in \Gamma_{-}(x)\;;\;\bm{h}(x,v,v') \geq m\}.$
\end{defi}
\begin{nb} If for instance $\bm{h}(x,v,v') \leq F(v)$ for any $(x,v,v') \in \partial\Omega\times V\times V$ with $F \in L^{1}(V,\,|v|\bm{m}(\d v))$ then $\H$ is regular.\end{nb}

We can combine \cite[Lemma 3.8]{LMR} and \cite[Lemma 3.7]{LMR} to prove the following approximation result 
\begin{lemme}\phantomsection
\label{lemme-1-add-sweeping} 
Let $\H$ be a regular \emph{stochastic} diffuse boundary operator
with kernel $\bm{h}(x,v,v^{\prime })$. There exists a sequence $\left(\H_{n}\right)_{n}$ of regular diffuse boundary operator such that
$$\lim_{n\to\infty}\|\H_{n}-\H\|_{\mathscr{B}(\lp,\lm)}=0$$
where, for any $n \in \N$ the kernel $\bm{h}_{n}(x,v,v')$ of $\H_{n}$ is such that, for  any $(x,v') \in \Gamma_{-}$, 
$$\bm{h}_{n}(x,v,v')=0 \qquad \text{ if} \quad |v| \leq \tfrac{1}{n} \quad \text{ or } \quad |v| \geq n.$$
\end{lemme}
According to \cite[Theorem 5.1]{LMR}, if $\H$ is a regular and stochastic diffuse boundary operator then
$$\H\mathsf{M}_{0}\H \in \mathscr{B}(\lp,\lm) \text{ is weakly-compact,}$$ i.e. Assumptions \ref{hypH} \textit{2)} is met. 
Moreover, to prove that Assumption \ref{hypH}\textit{4)} is satisfied, one can work with a dense family of regular boundary operators, which means, thanks to Lemma \ref{lemme-1-add-sweeping} that there is no loss of generality assuming that velocities are bounded and bounded away from zero. In the sequel, we shall therefore always assume that
\begin{equation}\label{eq:Vc}
V=\{v \in V\;;\;c < |v| < c^{-1}\,\}, \qquad \qquad 0 < c < 1.\end{equation}

\subsection{A subclass of regular diffuse operators.} 
To give a practical criterion ensuring that Assumption \ref{hypH} \textit{4)} is met,   we introduce a smaller class of diffuse boundary operators for which the collective weak-compactness is replaced with the collective compactness. To do so, it will be convenient not to work directly with $\H$ but with a different operator. Namely, assuming $\H$ is stochastic diffuse boundary operator as defined in Definition \ref{defi:Hdiff} with kernel $h(\cdot,\cdot,\cdot)$, 
notice that
$$|v\cdot n(x)|\H\psi(x,v)=\int_{\Gamma_{+}(x)}|v \cdot n(x)|\bm{h}(x,v,v')\psi(x,v')\,|v'\cdot n(x)|\bm{m}(\d v')$$
and that, for given $x \in \partial\Omega$, the mapping
$$v' \in \Gamma_{+}(x) \longmapsto \psi(x,v')|v'\cdot n(x)|$$
belongs to $L^{1}(\Gamma_{+}(x),\bm{m})$ whereas 
$$v \in \Gamma_{-}(x) \longmapsto |v\cdot n(x)|\H\psi(x,v)$$
belongs to $L^{1}(\Gamma_{-}(x),\bm{m})$. For simplicity, we assume that $\bm{h}(x,v,v')$ is defined on all $\partial\Omega\times V\times V$ and introduce then the kernel
\begin{equation}\label{eq:kxvv}
k(x,v,v')=|v\cdot n(x)|\bm{h}(x,v,v').\end{equation}
 We introduce then the following
$$\mathsf{K}\::\:L^{1}(\partial\Omega\times V\,,\,\pi(\d x)\otimes \bm{m}(\d v)) \longrightarrow L^{1}(\partial\Omega\times V\,,\,\pi(\d x)\otimes \bm{m}(\d v))$$
defined by
$$\mathsf{K}\psi(x,v)=\int_{V}k(x,v,v')\psi(x,v')\bm{m}(\d v'), \qquad \psi \in L^{1}(\partial\Omega\times V\,,\,\pi(\d x)\otimes \bm{m}(\d v)).$$
As before, we can identify 
$\mathsf{K} \in \mathscr{B}(L^{1}(\partial\Omega\times V\,,\,\pi(\d x)\otimes \bm{m}(\d v)))$ to a field of integral operators
$$x \in \partial\Omega \longmapsto \mathsf{K}(x) \in \mathscr{B}(L^{1}(\partial\Omega\,,\,\pi(\d x))$$
by the formula $[\mathsf{K}\psi](x,v)=\left[\mathsf{K}(x)\psi(x,\cdot)\right](v).$ From now, we will always simply write
$$L^{1}(\partial\Omega\times V)=L^{1}(\partial\Omega\times V\,,\,\pi(\d x)\otimes \bm{m}(\d v)).$$
To make the link between $\H$ and $\mathsf{K}$, we introduce, for any $x \in \partial\Omega$ the following \emph{isometry}
$$
\mathsf{I}_{x}^{+}\::\:L^{1}(\Gamma_{+}(x)\,,\bm{\mu}_{x}(\d v)) \longrightarrow L^{1}(\Gamma_{+}(x)\,,\,\bm{m}(\d v))$$
with
$$\mathsf{I}_{x}^{+}\psi(v)=|v\cdot n(x)|\psi(v), \qquad v \in \Gamma_{+}(x)$$
and the extension by zero
$$\mathsf{E}_{x}\::\:\varphi \in L^{1}(\Gamma_{+}(x)\,,\bm{m}(\d v)) \longmapsto \mathsf{E}_{x}\varphi \in L^{1}(V,\bm{m}(\d v))$$
with 
$$\mathsf{E}_{x}\varphi(v)=\begin{cases}\varphi(v) \quad &\text{ if } v \in \Gamma_{+}(x)\\
0 \quad &\text{ else }.\end{cases}$$
We also define the restriction operator
$$\mathsf{R}_{x}\::\:L^{1}(V,\bm{m}(\d v))\rightarrow L^{1}(\Gamma_{-}(x);\bm{m}(\d v)).$$
Finally, 
we introduce the \emph{isometry}
$$\mathsf{I}_{x}^{-}:\varphi \in L^{1}(\Gamma _{-}(x);\bm{m}(\d v))\longrightarrow L^{1}(\Gamma_{-}(x)\,,\,\bm{\mu}_{x}(\d v))$$
with
$$\mathsf{I}_{x}^{-}\varphi(v)=\frac{1}{|v\cdot n(x)|}\varphi(v), \qquad v \in \Gamma_{+}(x).$$
It holds then
\begin{equation}\label{eq:Kx}
\H(x):=\mathsf{I}_{x}^{-}\mathsf{R}_{x}\mathsf{K}(x)\mathsf{E}_{x}\mathsf{I}_{x}^{+}\end{equation}
with
\begin{equation}
\|\H\|_{\mathscr{B}(\lp,\lm)} \leq \|\mathsf{K}\|_{\mathscr{B}(L^{1}(\partial \Omega
\times V))}.  \label{Estimation}
\end{equation}
We have then the following approximation result which is easily deduced from \cite[Proposition 1]{mmkop} (replacing there $\Omega$ with $\partial\Omega$):
\begin{lemme}\label{lem:approx}
We assume that the operator $\mathsf{K}$ with kernel given by \eqref{eq:kxvv} is such that
\begin{equation}\label{eq:dire}
\left\{\mathsf{K}(x)\in \mathcal{L}(L^{1}(V\,;\bm{m}(\d v)));\ x\in \partial \Omega
\right\} \quad \text{ is collectively  compact}\end{equation}
and, for any  $\psi \in L^{\infty}(V,\bm{m}(\d v))$,
\begin{equation}\label{eq:dual}
\left\{\mathsf{K}(x)^{\star}\psi\,;\,x \in \partial\Omega\right\} \subset L^{\infty}(V,\bm{m}(\d v)), \qquad \text{ is relatively compact }\end{equation}
where $\mathsf{K}(x)$ is defined through \eqref{eq:Kx}. Then, there exist a sequence $(\mathsf{K}_{n})_{n} \subset \mathscr{B}(L^{1}(\partial\Omega\times V))$ such that
$$\lim_{n\to\infty}\|\mathsf{K}_{n}-\mathsf{K}\|_{\mathscr{B}(L^{1}(\partial\Omega\times V))}=0$$
where, for any $n \in \N$, the kernel $k_{n}(x,v,v')$ of $\mathsf{K}_{n}$ is of the form
\begin{equation}\label{eq:kernelkn}
k_{n}(x,v,v')=\sum_{j \in I_{n}}\beta_{j}(x)\,f_{j}(v)\,g_{j}(v'), \qquad x \in \partial \Omega, v \in \Gamma_{-}(x), \;v'\in \Gamma_{+}(x)\end{equation}
where $I_{n}$ is a \emph{finite} subset of $\N$ and
$$
\beta_{j}(\cdot) \in L^{\infty}(\partial\Omega,\d\pi), \quad f_{j} \in L^{1}(V\,;\,\bm{m}(\d v)),\quad
g_{j} \in L^{\infty}(V\,;\,\bm{m}(\d v)),$$
where we recall that $V$ is given by \eqref{eq:Vc}.\end{lemme}

An important feature of the above Lemma together with \eqref{Estimation} is that, to prove Assumption \ref{hypH} \textit{4)}, we can approximate the boundary operator $\H$ by a sequence $(\H_{n})_{n}$ such that
\begin{equation}\label{eq:kernL}
\H_{n}\varphi(x, v)=\sum_{j \in I_{n}}\beta_{j}(x)\,f_{j}(v)\int_{\Gamma_{+}(x)}g_{j}(v')\varphi(x,v')\bm{\mu}_{x}(\d v'), \qquad \varphi \in \lp.\end{equation}

Before explaining how assumptions \eqref{eq:dire} and \eqref{eq:dual} ensures Assumptions \ref{hypH} \textit{4)}, we give  a practical criterion ensuring that the boundary operator $\H$ satisfies \eqref{eq:dire} and \eqref{eq:dual}. 

\begin{propo}
\label{Proposition Hyp 1 et Hyp 2} The operator $\mathsf{K}$ related to $\H$ through \eqref{eq:Kx} satisfies Assumption \eqref{eq:dire} if, for any $\ 0<c<1$ 
\begin{equation}\label{eq:kol}
\lim_{z\rightarrow 0}\sup_{x\in\partial \Omega ,c\leq |v'|
 \leq c^{-1}\ }\int_{\left\{ c\leq |v| \leq
c^{-1}\right\} }\left| k(x,v+z,v')-k(x,v,v')\right| \bm{m}(\d v)=0.\end{equation}
Moreover, \eqref{eq:dual} is satisfied if, for any $\varphi \in
L^{\infty }(V;\bm{m}(\d v))$, the mapping
$$x\in \partial \Omega \longmapsto \int_{V}k(x,v',.)\varphi
(v')\bm{m}(\d v')\in L^{\infty }(V;\bm{m}(\d v)) $$
is piecewise continuous (with a finite number of pieces). In particular, both assumptions \eqref{eq:dire} and \eqref{eq:dual} are met whenever
\begin{equation}
(x,v,v^{\prime })\in \partial \Omega \times V\times V \longmapsto
k(x,v,v^{\prime })\text{ is continuous.}  \label{Continuity}
\end{equation}
\end{propo}

\begin{proof} The fact that \eqref{eq:kol} implies \eqref{eq:dire} is a simple consequence of Kolmogorov compactness criterion (see \cite[Theorem 4.26, p. 126]{brezis}) whereas the second part of the proposition comes from the fact that the image of $\partial\Omega$ through a piecewise continuous mapping is compact.
\end{proof}
\begin{nb}\label{nb:bound} If $k(x,v,v') =\alpha(x)|v\cdot n(x)|\bar{h}(v,v')$ with $\alpha(\cdot) \in L^{\infty}(\partial\Omega)$ and $\bar{h}(\cdot,\cdot)$ continuous then both \eqref{eq:dire} and \eqref{eq:dual} are satisfied. Notice that we do not require $\alpha(\cdot)$ to be continuous in this case.\end{nb}

\subsection{About Assumptions \ref{hypH} \textit{4)}.} Our scope in this subsection is to prove the following
\begin{theo}\label{ass:4} Assume that the diffuse boundary operator $\H$ as in Definition \ref{defi:Hdiff} is related through \eqref{eq:Kx} to $\mathsf{K}(x)$  satisfying \eqref{eq:dire} and \eqref{eq:dual}. Assume also that
$$\bm{m}(\d v)=\varpi(|v|)\d v$$
is absolutely continuous with respect to the Lebesgue measure over $\R^{d}$ with a radial and nonnegative weight function $\varpi(|v|)$.
Then, 
$$\lim_{s\to\infty}\left\|\left(\mathsf{M}_{\varepsilon+is}\H\right)^{2}\right\|_{\mathscr{B}(\lp)}=0 \qquad \forall \varepsilon \geq0$$
i.e. \eqref{eq:poweriml} in Assumption \ref{hypH} \textit{4)} holds true for $\ell=2$.
\end{theo}

\begin{proof}  On the basis of Lemma \ref{lem:approx} and since $\|\mathsf{M}_{\varepsilon+is}\|_{\mathscr{B}(\lm,\lp)} \leq 1$ for any $\varepsilon \geq 0, s \in \R,$ it is enough to prove that
\begin{equation}\label{eq:KMK2}
\lim_{s\to\infty}\left\|\mathsf{H}_{1}\mathsf{M}_{\e+is}\mathsf{H}_{2}\right\|_{\mathscr{B}(\lp,\lm)}=0 \qquad \forall \varepsilon \geq0\end{equation}
where $\mathsf{H}_{1},\mathsf{H}_{2}$ are operators of the form \eqref{eq:kernL}. Actually, using linearity, one can simply assume that the kernels $h_{1}$ and $h_{2}$ of $\mathsf{H}_{1},\mathsf{H}_{2}$ are given by
$$h_{1}(x,v,v')=\beta_{1}(x)f_{1}(v)g_{1}(v'), \qquad \qquad h_{2}(x,v,v')=\beta_{2}(x)f_{2}(v)g_{2}(v)$$
where
$$\beta_{j}(\cdot) \in L^{\infty}(\partial\Omega,\d\pi), \quad f_{j} \in L^{1}({V}\,;\,\bm{m}(\d v)),\quad
g_{j} \in L^{\infty}({V}\,;\,\bm{m}(\d v)), \qquad j=1,2.$$
In this case, one can show (see \cite{LM-iso} for details and similar computations) that, for all $\mathrm{Re}\l \geq0$ and any $\varphi \in \lp$
\begin{multline*}
\mathsf{H_{1}M_{\l}H_{2}}\varphi(x,v)={\beta_{1}(x)f_{1}(v)}\int_{\Gamma_{+}(x)}\exp(-\lambda\,\tau_{-}(x,v'))h(v')\beta_{2}(x-\tau_{-}(x,v')v')|v'\cdot n(x)|\bm{m}(\d v')\\
\int_{w \cdot n(x-\tau_{-}(x,v')v')>0}g_{2}(w)\varphi(x-\tau_{-}(x,v')v',w)|w\cdot n(x-\tau_{-}(x,v')v')|\bm{m}(\d w),\end{multline*}
where
$$h=g_{1}f_{2} \in L^{1}(V\,;\,\bm{m}(\d v)).$$ 
Introducing polar coordinates $v'=\varrho\,\sigma$, $\varrho=|v'| >0$ and $\sigma \in \S^{d-1}$ and using that 
$$\bm{m}(\d v)=\varpi(\varrho)\varrho^{d-1}\d\varrho \otimes \d\sigma$$ where $\d\sigma$ is the \emph{normalized} Lebesgue measure on $\S^{d-1}$, we get
\begin{multline*}
\mathsf{H}_{1}\mathsf{M}_{\l}\mathsf{H}_{2}\varphi(x,v)
={\beta_{1}(x)f_{1}(v)}\\
\int_{0}^{\infty} \varrho^{d}\varpi(\varrho)\d\varrho
\int_{\mathbb{S}_{+}(x)}\exp(-\lambda\,\tau_{-}(x,\sigma)\varrho^{-1})h(\varrho\sigma)G_{\varphi}(x-\tau_{-}(x,\sigma)\sigma)|\sigma \cdot n(x)|\d\sigma\,,\end{multline*}
where for $\varphi \in \lp$,
$$G_{\varphi}(y)=\beta_{2}(y)\int_{w\cdot n(y)>0}g_{2}(w)\varphi(y,w)|w\cdot n(y)|\bm{m}(\d w), \qquad y \in \partial\Omega.$$ 
Notice that
$$\int_{\partial\Omega}|G_{\varphi}(y)|\pi(\d y) \leq \|g_{2}\|_{L^{\infty}(V,\bm{m}(\d v))}\,\|\beta_{2}\|_{L^{\infty}(\partial\Omega)}\,\|\varphi\|_{\lp}$$
i.e. the mapping $\varphi \in \lp \mapsto G_{\varphi} \in L^{1}(\partial\Omega,\pi(\d x))$ is bounded, whereas the linear mapping
$$\psi \in L^{1}(\partial\Omega,\pi(\d x)) \mapsto \psi(x)\beta_{1}(x)f_{1}(v) \in \lp$$ is also bounded. Therefore, to prove \eqref{eq:KMK2}, one only needs to prove that
\begin{equation}\label{eq:Ll}
\lim_{|\mathrm{Im}|\to\infty}\left\|L(\l)\right\|_{\mathscr{B}(L^{1}(\partial\Omega,\pi(\d x)))}=0, \qquad \mathrm{Re}\l \geq 0\end{equation}
where $L(\l)\::\:G \in L^{1}(\partial\Omega,\pi(\d x)) \mapsto L(\l)G \in L^{1}(\partial\Omega,\pi(\d x))$ defined by
\begin{multline}\label{eq:LlG}
L(\l)G(x)=\int_{c}^{1/c}\varrho^{d}\varpi(\varrho)\d\varrho\int_{\mathbb{S}_{+}(x)}\exp(-\lambda\,\tau_{-}(x,\sigma)\varrho^{-1})\\
h(\varrho\sigma)G (x-\tau_{-}(x,\sigma)\sigma)|\sigma \cdot n(x)|\d\sigma\,,\end{multline}
for any $x \in \partial\Omega$ and  $G \in L^{1}(\partial\Omega,\pi(\d x))$ where we used \eqref{eq:Vc}. 
One notice first that the operator $L(\l)$ depends continuously on $h \in L^{1}(V,\bm{m}(\d v))$ (see Lemma \ref{lem:deph} after this proof) and since the set of continuous compactly supported mappings is dense in $L^{1}(V,\bm{m}(\d v))$, to prove \eqref{eq:Ll} we can assume without loss of generality that 
$h \in \mathcal{C}_{c}(V)$
with $\mathrm{Supp}\,h \subset V.$ Under such an assumption,  \eqref{eq:Ll} holds true as shown in the subsequent Proposition \ref{prop:Leieta}.
\end{proof}

We now give the results allowing to complete the above proof. First, one has the following
\begin{lemme}\label{lem:deph} For any $\mathrm{Re}\l \geq 0$, the mapping 
$$h \in L^{1}({V},\bm{m}(\d v)) \mapsto L(\l) \in \mathscr{B}(\lp)$$
is linear and continuous.\end{lemme}
\begin{proof} The fact that the mapping $h \in L^{1}(V,\bm{m}(\d v)) \mapsto L(\l)h$ is linear is clear. For any $x \in \partial\Omega$, one first notices that
$$L(\l)G(x)=\int_{\Gamma_{+}(x)}e^{-\l\tau_{-}(x,v)}h(v)\,G(x-\tau_{-}(x,v)v)\,|v\cdot n(x)|\bm{m}(\d v), \qquad x \in \partial\Omega$$
so that
$$\left|L(\l)G(x)\right| \leq \int_{\Gamma_{+}(x)}\,|h(v)|\,|G(x-\tau_{-}(x,v)v)|\,|v\cdot n(x)|\bm{m}(\d v)$$
and, integrating over $\partial \Omega$ and using \eqref{10.51}, we get
$$\|L(\l)G\|_{L^{1}(\partial\Omega,\pi(\d x))} \leq \int_{\Gamma_{-}}|h(v)|\,|G(z)|\,\d\mu_{-}(z,v)=\int_{\Gamma_{-}}|h(v)|\,|G(z)|\,|v\cdot n(z)|\bm{m}(\d v)\d\pi(z)$$
so that
$$\|L(\l)G\|_{L^{1}(\partial\Omega,\pi(\d x))} \leq \frac{1}{c}\|h\|_{L^{1}(V,\bm{m}(\d v))}\,\|G\|_{L^{1}(\partial\Omega,\pi(\d x))}, \qquad \forall \mathrm{Re}\l \geq0$$
where we use the bound $|v\cdot n(x)| \leq 1/c$ for $v \in \mathrm{Supp}h \subset V$. In other words
$$\sup_{\mathrm{Re}\l \geq 0}\|L(\l)\|_{\mathscr{B}(L^{1}(\partial\Omega,\pi(\d x)))} \leq \frac{1}{c}\,\|h\|_{L^{1}(V,\bm{m}(\d v))}.$$
This proves the result.\end{proof}

We can prove the following which completes the proof of Theorem \ref{ass:4}:
\begin{propo}\label{prop:Leieta} Assume that $h \in \mathcal{C}_{c}(V)$
with $\mathrm{Supp}\,h \subset V$ where we recall $V$ satisfies \eqref{eq:Vc}. Then,
$$\lim_{|\eta| \to \infty}\|L(\e+i\eta)\|_{\mathscr{B}(L^{1}(\partial\Omega,\pi(\d x)))}=0 \qquad \forall \e \geq0.$$
\end{propo}
\begin{proof} First, using Lemma \ref{lem:ChVa} and \eqref{eq:LlG}, one sees that, for any $G \in L^{1}(\partial\Omega,\pi(\d x))$
$$L(\l)G(x)=\int_{c}^{\frac{1}{c}}\varrho^{d}\varpi(\varrho)\d\varrho\int_{\Gamma_{+}(x)}e^{-\l|x-y|\varrho^{-1}}h\left(\varrho\frac{x-y}{|x-y|}\right)\mathcal{J}(x,y)\,G(y)\pi(\d y)\,, \qquad x \in \partial\Omega.$$
Therefore,
$$\|L(\l)\|_{\mathscr{B}(L^{1}(\partial\Omega,\pi(\d x)))} \leq \sup_{y \in \partial\Omega}\,\int_{\partial\Omega}\left|\int_{c}^{\frac{1}{c}}e^{-\l|x-y|\varrho^{-1}}h\left(\varrho\frac{x-y}{|x-y|}\right)\varrho^{d}\varpi(\varrho)\d\varrho\right|\mathcal{J}(x,y)\pi(\d x).$$
We denote then, for any fixed $\e \geq0$,
$$H(\eta,x):=\int_{c}^{\frac{1}{c}}e^{-(\e+i\eta)|x-y|\varrho^{-1}}h\left(\varrho\frac{x-y}{|x-y|}\right)\varrho^{d}\varpi(\varrho)\d\varrho\,, \quad x \in \partial\Omega$$
and need to prove that
$$\lim_{|\eta|\to\infty}\sup_{y \in \partial\Omega}\int_{\partial\Omega}\left|H(\eta,x)\right|\mathcal{J}(x,y)\pi(\d x)=0.$$
Given $y \in \partial\Omega$ one can split the above integral as
\begin{multline*}
\int_{\partial\Omega}\left|H(\eta,x)\right|\mathcal{J}(x,y)\pi(\d x)=\int_{|x-y| >\delta}\left|H(\eta,x)\right|\,\mathcal{J}(x,y)\pi(\d x)\\
+\int_{|x-y|\leq\delta}\left|H(\eta,x)\right|\,\mathcal{J}(x,y)\pi(\d x), \qquad \delta >0.\end{multline*}
First, one sees that
$$\sup_{\eta \in \R}\int_{|x-y|\leq\delta}\left|H(\eta,x)\right|\,\mathcal{J}(x,y)\pi(\d x) \leq \|h\|_{L^{\infty}(V,\d\bm{m})}\left(\int_{c}^{\frac{1}{c}}\varrho^{d}
\varpi(\varrho)\d\varrho\right)\,\int_{|x-y|\leq\delta}\mathcal{J}(x,y)\pi(\d x)$$
and thanks to Lemma \ref{lem:convede},
$$\lim_{\delta\to 0^{+}}\sup_{y \in \partial\Omega}\int_{|x-y|\leq\delta}\mathcal{J}(x,y)\pi(\d x)=0.$$
Thus, to prove our claim, we only have to show that, for any $\delta >0$,
$$\lim_{|\eta|\to\infty}\sup_{y \in \partial\Omega}\int_{|x-y|>\delta}\left|H(\eta,x)\right|\,\mathcal{J}(x,y)\pi(\d x)=0.$$
For $|x-y| > \delta$, one has $\mathcal{J}(x,y) \leq |x-y|^{1-d} \leq \delta^{1-d}$  and we are reduced to prove that
\begin{equation}\label{eq:final}
\lim_{|\eta|\to\infty}\sup_{y\in \partial\Omega}\int_{|x-y| > \delta}\left|H(\eta,x)\right|\pi(\d x)=0.\end{equation}
For a fixed $y \in \partial\Omega$ and $r:=|x-y| >\delta$, we perform the change of variable 
$\varrho \mapsto z:=|x-y|\varrho^{-1}$ with $\d\varrho=|x-y|z^{-2}\d z$ so that
$$H(\eta,x)=r^{d}\int_{rc}^{r/c}e^{-i\eta z}e^{-\e z}h\left(z^{-1}(x-y)\right)\varpi(r/z)\frac{\d z}{z^{2+d}} \qquad r=|x-y|\,.$$
which means that $H(\eta,x)$ is the Fourier transform of the continuous mapping $z \mapsto \zeta_{x,y}(z)$ given by
$$\zeta_{x,y}(z):=r^{d}\mathbf{1}_{(rc,r/c)}(z)e^{-\e z}h\left(z^{-1}(x-y)\right)\varpi(r/z)z^{-(2+d)}.$$
From Rieman-Lebesgue Theorem, one has
\begin{equation}\label{eq:RL}
\lim_{|\eta|\to\infty}\int_{\R_{+}}\big|\zeta_{x,y}(z)\big|\d z=0\end{equation}
for any $(x,y) \in A_{\delta}:=\{(x,y) \in \partial\Omega\times\partial\Omega\,;\,|x-y| >\delta\}$. Since the mapping
$$(x,y) \in A_{\delta}\longmapsto \zeta_{x,y}(\cdot) \in L^{1}(rc,r/c)$$
is \emph{continuous}, the above convergence \eqref{eq:RL} is uniform with respect to $(x,y)$, i.e.   
$$\lim_{|\eta|\to\infty}\sup_{y \in \partial\Omega}\sup_{|x-y|>\delta}\left|\int_{\R_{+}}e^{i\eta\,z}\zeta_{x,y}(z)\d z\right|=0.$$
This exactly means that
$$\lim_{|\eta|\to\infty}\sup_{y \in \partial\Omega}\sup_{|x-y|>\delta}\bigg|H(\eta,x)\bigg|=0$$
and, of course, this implies \eqref{eq:final}.
\end{proof}

\subsection{A few examples} We give some examples of practical application in the kinetic theory of gases \cite{aoki,bernou1}.

\begin{enumerate}
\item First, we consider an operator $\H$ with kernel
$$\bm{h}(x,v,v')=\alpha(x)\,h(v,v')$$
where 
$$\alpha(\cdot) \in L^{\infty}(\partial\Omega,\d\pi), \qquad \text{ and } \quad h(\cdot,\cdot) \text{ is continuous on } V\times V.$$
In that case, 
$$k(x,v,v')=\alpha(x)|v \cdot n(x)|\,h(v,v')$$
is such that, for $x \in \partial\Omega$, 
$$k(x,\cdot,\cdot) \text{ is continuous over } V \times V$$
and \eqref{eq:dual} and \eqref{eq:dire} are satisfied thanks to Proposition \ref{Proposition Hyp 1 et Hyp 2} and Remark \ref{nb:bound}.
\item Another example is the one given in Example \ref{exe:gener}:
$$\bm{h}(x,v,v')=\gamma^{-1}(x)\bm{G}(x,v)$$
where $\G\::\:\partial \Omega \times V \to \R^{+}$ is a measurable and nonnegative mapping such that 
\begin{enumerate}[($i$)]
\item $\G(x,\cdot)$ is radially symmetric for $\pi$-almost every $x \in \partial \Omega$; 
\item $\G(\cdot,v) \in L^{\infty}(\partial \Omega)$ for almost every $v \in V$;
\item The mapping $x \in \partial\Omega \mapsto G(x,\cdot) \in L^{1}(V,|v|\bm{m}(\d v))$ is piecewise continuous,
\item The mapping $x \in \partial \Omega \mapsto \gamma(x)$ is \emph{bounded away from zero} where
\begin{equation*}
\gamma(x):=\int_{\Gamma_{-}(x)}\G(x,v)|v \cdot n(x)| \bm{m}(\d v) \qquad \forall x \in \partial\Omega,\end{equation*}
i.e. there exist $\gamma_{0} >0$ such that $\gamma(x) \geq \gamma_{0}$ for $\pi$-almost every $x \in \partial\Omega.$\end{enumerate}
In that case, it is easy to show that the associated boundary operator $\H$ is satisfying \eqref{eq:dire} and \eqref{eq:dual} and, if 
\begin{equation}\label{eq:bmvarpi}
\bm{m}(\d v)=\varpi(|v|)\d v\end{equation}
for some radially symmetric and nonnegative function $\varpi(|v|)$, we deduce from Theorem \ref{ass:4} that Assumptions \ref{hypH} are met. 

\item A particular example of this kind of boundary operator corresponds to the case in which
$$\bm{G}(x,v)=\mathcal{M}_{\theta(x)}(v), \qquad \mathcal{M}_{\theta}(v)=(2\pi\theta)^{-d/2}\exp\left(-\frac{|v|^{2}}{2\theta}\right), \qquad x \in \partial \Omega, \:\:v \in \R^{d}.$$ 
Then, 
$$\gamma(x)=\bm{\kappa}_{d}\sqrt{\theta(x)}\int_{\R^{d}}|w|\M_{1}(w)\d w, \qquad x \in \partial\Omega$$ 
for some positive constant $\bm{\kappa}_{d}$ depending only on the dimension. The above assumption \textit{$(iii)$} asserts that the temperature mapping $x \in \partial\Omega \mapsto \theta(x)$ is bounded away from zero and piecewise continuous.

In that case, again, if $\bm{m}$ is given by \eqref{eq:bmvarpi}, Assumptions \ref{hypH} are met. Moreover, for, say
$$\varpi(|v|)=1, \qquad \bm{m}(\d v)=\d v$$
we see that
$$\H \in \mathscr{B}(\lp,\Y_{k+1}^{-}) \Longleftrightarrow  \int_{0}^{\infty}|v|^{-k}\mathcal{M}_{\theta(x)} (v)\d v < \infty \Longleftrightarrow  k < d,$$
which means that $N_{\H}=d-1$ (since $N_{\H}$ needs to be an integer).
\end{enumerate}
\section{Some useful change of variables}\label{appen:chv}
We establish now a fundamental change of variable formula which has its own interest and that we used in the previous Appendix. A systematic use of such a change of variable will be made in a companion paper \cite{LM-iso}. We begin with the following technical Lemma:

\begin{lemme}\phantomsection\label{lem:jaco}
Let $x \in \partial{\Omega}$ be fixed. We denote by $\mathbf{B}_{d-1}$ the closed unit ball of $\R^{d-1}$ and define
$$\mathfrak{p}\::\:z \in \mathbf{B}_{d-1} \longmapsto \mathfrak{p}(z)=x-\tau_{-}(x,\sigma(z))\sigma(z) \in \partial{\Omega}$$
where
$$\sigma(z)=\left(\sigma_{1}(z),\ldots,\sigma_{d}(z)\right)=\left(z_{1},\ldots,z_{d-1},\sqrt{1-|z|^{2}}\right) \in \S^{d-1}, \qquad z \in \mathbf{B}_{d-1}.$$
Defining
$$\mathcal{O}_{x}:=\left\{z \in \mathbf{B}_{d-1}\;;\;\sigma(z) \cdot n(x) >0\;;\;\sigma(z) \cdot n(\mathfrak{p}(z)) < 0\right\}$$ 
it holds that $\mathfrak{p}$ is differentiable on $\mathcal{O}_{x}$ and
$$\mathrm{det}\Bigg(\Bigg(\left\langle \frac{\partial \mathfrak{p}(z)}{\partial z_{i}},\frac{\partial \mathfrak{p}(z)}{\partial z_{j}}\right\rangle\Bigg)_{1\leq i,j\leq d-1}\Bigg)=\left(\frac{\tau_{-}(x,\sigma(z))^{d-1}}{\left(\sigma(z)\cdot n(\mathfrak{p}(z))\right)\sigma_{d}(z)}\right)^{2} \qquad \forall z \in \mathcal{O}_{x}.$$\end{lemme}
\begin{proof} The fact that $\mathfrak{p}(\cdot)$ is differentiable on $\mathcal{O}_{x}$ is a consequence of \cite[Lemma A.4]{LMR}. We recall in particular here that the set $\widehat{\Gamma}_{+}(x)=\{\omega \in \S^{d-1}\;;\;(x,\omega) \in \Gamma_{+}\text{ and } \bm{\xi}(x,\omega)=(x-\tau_{-}(x,\omega)\omega,\omega) \in \Gamma_{-}\}$ is an open subset of $\{\omega \in \S^{d-1}\;;\;(x,\omega) \in \Gamma_{+}\}$ and $\tau_{-}(x,\cdot)$ is differentiable on $\widehat{\Gamma}_{+}(x)$ with
\begin{equation}\label{eq:difftau}
\partial_{\omega}\tau_{-}(x,\omega)=\frac{\tau_{-}(x,\omega)}{w\cdot n(\bm{\xi}(x,\omega))}n(\bm{\xi}_{s}(x,\omega)), \qquad \omega \in \widehat{\Gamma}_{+}(x)
\end{equation}
where $\bm{\xi}_{s}(x,\omega)=x-\tau_{-}(x,\omega)\omega$. Since, for any $z \in \mathcal{O}_{x}$, $\mathfrak{p}(z)=\bm{\xi}_{s}(x,\sigma(z))$, this translates in a straightforward way to the differentiability of $\mathfrak{p}$. Moreover, one deduces from \eqref{eq:difftau} that
\begin{equation*}\begin{split}
\partial_{i}\mathfrak{p}(z)&=-\left(\nabla_{\omega}\tau_{-}(x,\sigma(z))\cdot\partial_{i}\sigma(z)\right)\sigma(z)-\tau_{-}(x,\sigma(z))\partial_{i}\sigma(z)\\
&=-\frac{\tau_{-}(x,\sigma(z))}{\sigma(z)\cdot n(\mathfrak{p}(z))}\left[\left(n(\mathfrak{p}(z))\cdot\partial_{i}\sigma(z)\right)\sigma(z)+\left(n(\mathfrak{p}(z))\cdot\sigma(z)\right)\partial_{i}\sigma(z)\right]\end{split}\end{equation*}
where we denote for simplicity $\partial_{i}=\frac{\partial}{\partial z_{i}}$. We therefore get
\begin{multline*}
\left\langle \partial_{i}\mathfrak{p}(z)\,;\,\partial_{j}\mathfrak{p}(z)\right\rangle=\left(\frac{\tau_{-}(x,\sigma(z))}{\sigma(z)\cdot n(\mathfrak{p}(z))}\right)^{2}\bigg(\left(n(\mathfrak{p}(z))\cdot\partial_{i}\sigma(z)\right)\left(n(\mathfrak{p}(z))\cdot\partial_{j}\sigma(z)\right) \\
 +\left(n(\mathfrak{p}(z))\cdot\sigma(z)\right)^{2}\left(\partial_{i}\sigma(z)\cdot \partial_{j}\sigma(z)\right)\bigg).\end{multline*}
Let us fix $z \in \mathcal{O}_{x}$. We denote by $\mathfrak{P}(z)$ the matrix whose entries are $\mathfrak{P}_{ij}(z)=\left\langle \partial_{i}\mathfrak{p}(z)\,;\,\partial_{j}\mathfrak{p}(z)\right\rangle$, $1 \leq i,j \leq d-1$. Using that
$$\partial_{i}\sigma(z)\cdot\partial_{j}\sigma(z)=\delta_{ij}+\frac{z_{i}z_{j}}{\sigma_{d}^{2}(z)}$$
where $\sigma_{d}(z)$ is the last component of $\sigma(z)$, i.e. $\sigma_{d}(z)=\sqrt{1-|z|^{2}}$, one sees that
$$\mathfrak{P}_{ij}(z)=\tau_{-}^{2}(x,\sigma(z))\Bigg[\delta_{ij}+\frac{z_{i}z_{j}}{\sigma_{d}^{2}(z)}+\frac{1}{(\sigma(z)\cdot n(\mathfrak{p}(z)))^{2}}\left(n(\mathfrak{p}(z))\cdot\partial_{i}\sigma(z)\right)\left(n(\mathfrak{p}(z))\cdot\partial_{j}\sigma(z)\right)\Bigg].$$
For simplicity, we will simply denote by $n$ the unit vector $n(\mathfrak{p}(z))$ and $\sigma=\sigma(z)$. Introducing the vectors $u,p\in \R^{d-1}$ whose components are 
$$u_{i}:=\frac{n\cdot\partial_{i}\sigma}{\sigma\cdot n}, \qquad p_{i}:=\frac{z_{i}}{\sigma_{d}}, \qquad i=1,\ldots,d-1$$
we have $\mathfrak{P}_{ij}(z)=\tau_{-}(x,\sigma)^{2}\left[\delta_{ij}+p_{i}p_{j}+u_{i}u_{j}\right]$ so that
$$\mathrm{det}\left(\mathfrak{P}(z)\right)=\left(\tau_{-}(x,\sigma)\right)^{2(d-1)}\mathrm{det}\left(\mathbf{I}_{d-1}+p\otimes p +u \otimes u\right).$$
Recall that, for any invertible matrix $\bm{A}$ and any vectors $\bm{x},\bm{y} \in \R^{d-1}$, then
\begin{equation}\label{detA}
\mathrm{det}\left(\bm{A}+\bm{x}\otimes\bm{y}\right)=\mathrm{det}(A)\left(1+\langle \bm{y},\bm{A}^{-1}\bm{x}\rangle\right).\end{equation}
We apply this identity first by considering $\bm{A}=\mathbf{I}_{d-1}+p\otimes p$ and get
$$\mathrm{det}\left(\mathfrak{P}(z)\right)=\left(\tau_{-}(x,\sigma)\right)^{2(d-1)}\mathrm{det}(\bm{A})\left(1+\langle u,\bm{A}^{-1}u\rangle\right).$$
To compute $\mathrm{det}(\bm{A})$, one uses again \eqref{detA} to deduce
$$\mathrm{det}(\bm{A})=1+\langle p,p\rangle=1+\frac{|z|^{2}}{\sigma^{2}_{d}}=\sigma_{d}^{-2}.$$
One also can compute in a direct way the inverse of $\bm{A}$ given by $\bm{A}^{-1}=\mathbf{I}_{d-1}-\frac{1}{1+|p|^{2}}p\otimes p$ from which
$$\langle u,\bm{A}^{-1}u\rangle=|u|^{2}-\frac{\langle p,u\rangle^{2}}{1+|p|^{2}}.$$
This results in 
\begin{equation}\label{eq:detPz}
\mathrm{det}\left(\mathfrak{P}(z)\right)=\left(\tau_{-}(x,\sigma)\right)^{2(d-1)}\left(1+|u|^{2}-\frac{\langle p,u\rangle^{2}}{1+|p|^{2}}\right).\end{equation}
Let us make this more explicit. One easily checks that
$$u_{i}=\frac{1}{\sigma\cdot n}\left(n_{i}-\frac{n_{d}}{\sigma_{d}}z_{i}\right) \qquad \text{ and } \qquad \langle p,u\rangle=\frac{1}{(\sigma\cdot n)\sigma_{d}}\sum_{i=1}^{d-1}\left(n_{i}z_{i}-\frac{n_{d}}{\sigma_{d}}z_{i}^{2}\right).$$
Noticing that $\sum_{i=1}^{d-1}n_{i}z_{i}=(\sigma\cdot n)-\sigma_{d}n_{d}$, it holds
$$\langle p,u\rangle=\frac{1}{(\sigma\cdot n)\sigma_{d}}\left(\sigma\cdot n - \sigma_{d}n_{d}-\frac{n_{d}}{\sigma_{d}}|z|^{2}\right)=\frac{1}{\sigma_{d}^{2}}\left(\sigma_{d}-\frac{n_{d}}{\sigma\cdot n}\right)$$
where we used again that $\sigma_{d}^{2}+|z|^{2}=1$. Since one also has
$$|u|^{2}=\frac{1}{(\sigma\cdot n)^{2}}\sum_{i=1}^{d-1}\left(n_{i}-\frac{n_{d}}{\sigma_{d}}z_{i}\right)^{2}=\frac{1}{\sigma_{d}^{2}(\sigma\cdot n)^{2}}\sum_{i=1}^{d-1}\left(\sigma_{d}n_{i}-z_{i}n_{d}\right)^{2}$$
we get easily after expanding the square and using that $\sum_{i=1}^{d-1}n_{i}^{2}=1-n_{d}^{2}$, 
$$|u|^{2}=\frac{1}{\sigma_{d}^{2}(\sigma\cdot n)^{2}}\left(\sigma_{d}^{2}+n_{d}^{2}-2(\sigma\cdot n)\sigma_{d}n_{d}\right).$$
One finally obtains, using \eqref{eq:detPz},
$$\mathrm{det}\left(\mathfrak{P}(z)\right)=\frac{\tau_{-}(x,\sigma)^{2(d-1)}}{\sigma_{d}^{2}}\left(1+\frac{1}{\sigma_{d}^{2}(\sigma\cdot n)^{2}}\left(\sigma_{d}^{2}+n_{d}^{2}-2(\sigma\cdot n)\sigma_{d}n_{d}\right)-\frac{1}{\sigma_{d}^{2}}\left(\sigma_{d}-\frac{n_{d}}{\sigma\cdot n}\right)^{2}\right)$$
and little algebra gives
$$\mathrm{det}\left(\mathfrak{P}(z)\right)=\frac{\tau_{-}(x,\sigma)^{2(d-1)}}{\sigma_{d}^{2}(\sigma\cdot n)^{2}}$$
which is the desired result.\end{proof}

We complement the above with the following
\begin{lemme}\phantomsection\label{lem:sard}
For any $x \in \partial\Omega$, introduce
$$\mathcal{G}_{x}:=\{\omega \in \S^{d-1}\;;\;(x,\omega) \in \Gamma_{+}\;;\;\omega \cdot n(x-\tau_{-}(x,\omega)\omega)=0\}.$$
Then, 
$$|\mathcal{G}_{x}|=0$$
where here $|\cdot|$ denotes the Lebesgue surface measure over $\S^{d-1}$. Moreover, with the notations of Lemma \ref{lem:jaco}, the set
$$\mathbf{G}_{x}=\left\{z \in \mathbf{B}_{d-1}\;;\;\sigma(z)=(z,\sqrt{1-|z|^{2}}) \in \mathcal{G}_{x}\right\}$$
has zero Lebesgue measure (in $\R^{d-1}$).
\end{lemme}
\begin{proof} The proof resorts to Sard Theorem. Let $x \in \partial\Omega$ be fixed. Introducing the function
$$\Psi\::\:y \in \partial\Omega \setminus\{x\} \longmapsto \Psi(y)=\frac{x-y}{|x-y|}$$
it holds that $\Psi$ is a $\mathcal{C}^{1}$ function. For any $\omega \in \S^{d-1}$ setting $y_{\omega}=x-\tau_{-}(x,\omega)\omega \in \partial\Omega$ one has
$$\omega=\Psi(y_{\omega})$$
Let us prove that $\mathcal{G}_{x}$ is included in the set of \emph{critical values} of $\Psi$. To do this, we show that, if $\omega \in \mathcal{G}_{x}$ then $y_{\omega}$ is a critical point of $\Psi$, i.e. the differential $\d\Psi(y_{\omega})$ is not injective. Since actually $\Psi$ is defined and smooth on $\R^{d}\setminus\{x\}$, its differential on $\partial\Omega \setminus \{x\}$ is nothing but the restriction of its differential on $\R^{d}\setminus\{x\}$ on the tangent hyperplane to $\partial\Omega$, i.e., for any $y \in \partial\Omega \setminus\{x\}$, one has
$$\d\Psi(y)\::\:h \in \mathcal{T}_{y} \longmapsto -\frac{1}{|x-y|}\mathds{P}_{z_{y}}(h)$$
where $\mathcal{T}_{y}$ is the tangent space of $\partial\Omega$ at $y\in \partial \Omega \setminus\{x\}$, $z_{y}=\frac{x-y}{|x-y|}$ and, for any $z \in \R^{d}$,  $\mathds{P}_{z}$ denote the orthogonal projection on the hyperplane orthogonal to $z$, 
$$\mathds{P}_{z}h=h_{z}^{\perp}:=h-\langle h,\,\bar{z}\rangle\,\bar{z}, \qquad \bar{z}=\frac{z}{|z|} \in \mathbb{S}^{d-1}, h \in \R^{d}.$$
Now, one notices that
$$\omega \in \mathcal{G}_{x} \implies \omega \cdot n(y_{\omega})=0, \qquad \omega=\Psi(y_{\omega}).$$
In particular, one has $\omega \in \mathcal{T}_{y_{\omega}}$ and $\mathds{P}_{z_{y_{\omega}}}(\omega)=0$ since $z_{y_{\omega}}=\frac{x-y_{\omega}}{|x-y_{\omega}}|=\omega$. In particular, $\d\Psi(y_{\omega})$ is not injective ($\omega \neq 0$ belongs to its kernel). We proved that, if $\omega \in \mathcal{G}_{x}$ then it is a \emph{critical value} of $\Psi$ and Sard Theorem implies in particular that the measure of $\mathcal{G}_{x}$ is zero. 

Now, to prove that $\mathbf{G}_{x}$ is also of zero measure, one simply notices that $\mathbf{G}_{x}$ is  the image of $\mathcal{G}_{x}$ through the smooth function
$$\mathcal{P}\::\:\omega \in \S^{d-1}\setminus \{w_{d} =0\} \longmapsto \mathcal{P}(\omega)=(\omega_{1},\ldots,\omega_{d-1}) \in \mathbf{B}_{d-1}$$
In particular, from the first part of the Lemma, $\mathbf{G}_{x} $ is included in the set of critical value of the smooth function $\mathcal{P} \circ \Psi$ and we conclude again with Sard Theorem.
\end{proof}

 \begin{lemme}\phantomsection\label{lem:ChVa} Assume that $\partial\Omega$ satisfies Assumption \ref{hypO}. For any $x \in \partial\Omega$, we set
 $$\S_{+}(x)=\left\{\sigma \in \S^{d-1}\;;\;\sigma \cdot n(x) >0\right\}.$$
Then, for any nonnegative measurable mapping $g\::\:\S^{d-1} \mapsto \R$ one has
$$\int_{\S_{+}(x)}g(\sigma)\,|\sigma\cdot n(x)|\d\sigma=\int_{\Gamma_{+}(x)}g\left(\frac{x-y}{|x-y|}\right)\mathcal{J}(x,y)\pi(\d y)$$
where $\Gamma_{+}(x)=\left\{y \in \partial \Omega\,;\,(x-y)\cdot n(x) >0\right\}$ and 
\begin{equation}\label{eq:Jxy}
\mathcal{J}(x,y)=\ind_{\{(x-y) \cdot n(y) < 0\}}(y)\frac{|(x-y)\cdot n(x)|}{|x-y|^{d+1}}\,|(x-y)\cdot n(y)|, \qquad \forall y \in \Gamma_{+}(x).\end{equation}
In particular, for any nonnegative measurable $G\::\:\partial \Omega \to \R$ and any $\varphi\::\:\R^{+}\to\C$ measurable we have
\begin{equation}\label{eq:CofV}
\int_{\S_{+}(x)}\left|\sigma \cdot n(x)\right|\,\varphi(\tau_{-}(x,\sigma)) \,G(x-\tau_{-}(x,\sigma)\sigma)\d\sigma=\int_{\Gamma_{+}(x)}G(y)\varphi(|x-y|)\mathcal{J}(x,y)\pi(\d y).\end{equation}
\end{lemme}
\begin{proof} Let now $x \in \partial\Omega$ be given. We can assume without generality that the system of coordinates in $\R^{d}$ is such that $n(x)=(0,\ldots,1)$. For a given $f\::\:\S_{+}(x)\to \R^{+}$, it holds
$$\int_{\S_{+}(x)}f(\sigma)\d \sigma=\int_{B_{x}}f(z,\sqrt{1-|z|^{2}})\frac{\d z}{\sqrt{1-|z|^{2}}},$$
where $B_{x}=\{z \in \R^{d-1}\;;\;|z| < 1, \; z \cdot n(x) >0\}.$ Moreover, according to Lemma \ref{lem:sard}
$$\int_{\S_{+}(x)}f(\sigma)\d\sigma=\int_{\S_{+}(x)\setminus\mathcal{G}_{x}}f(\sigma)\d\sigma$$
while
$$\int_{B_{x}}f(z,\sqrt{1-|z|^{2}})\frac{\d z}{\sqrt{1-|z|^{2}}}=\int_{B_{x}\setminus \mathbf{G}_{x}}f(z,\sqrt{1-|z|^{2}})\frac{\d z}{\sqrt{1-|z|^{2}}}$$
Then, for the special choice of $f(\sigma)=g(\sigma)\,|\sigma \cdot n(x)|$
we get
\begin{equation}\label{eq:G+x}
\int_{\S_{+}(x)}|\sigma\cdot n(x)|\,g(\sigma)\d\sigma=\\
\int_{B_{x}} g(\sigma(z)))\ind_{B_{x}\setminus\mathbf{G}_{x}}(z)\d z\end{equation}
with $\sigma(z) =(z_{1},\ldots,z_{d-1},\sqrt{1-|z|^{2}})$ for $|z| <1.$ 

Notice that, with the notations of Lemma \ref{lem:jaco}, one has $B_{x}\setminus \mathbf{G}_{x}=\mathcal{O}_{x}$. Still using the notations of Lemma \ref{lem:jaco}, we introduce the regular parametrization of $\partial\Omega$ given by
$$\mathfrak{p}\::\:z \in \mathcal{O}_{x} \mapsto y=\mathfrak{p}(z)=x-\tau_{-}(x,\sigma(z))\sigma(z) \in \partial\Omega.$$
With this parametrization, notice that 
$$\tau_{-}(x,\sigma(z))=|x-y|$$
since $\sigma(z) \in \S^{d-1}$ and therefore
$$\sigma(z)=\frac{x-\mathfrak{p}(z)}{\tau_{-}(x,\sigma(z))}=\frac{x-y}{|x-y|}.$$
According to \cite[Lemma 5.2.11 \& Theorem 5.2.16, pp. 128-131]{stroock}, from this parametrization, the Lebesgue surface measure $\pi(\d y)$ on $\partial\Omega$ is given by
$$\pi(\d y)=\sqrt{\mathrm{det}\Bigg(\Bigg(\left\langle \frac{\partial \mathfrak{p}(z)}{\partial z_{i}},\frac{\partial \mathfrak{p}(z)}{\partial z_{j}}\right\rangle\Bigg)_{1\leq i,j\leq d-1}\Bigg)}\d z_{1}\ldots \d z_{d-1}=\sqrt{\mathrm{det}(\mathfrak{P}(z))}\d z$$
from which we deduce directly that
$$\int_{B_{x}} \ind_{B_{x}\setminus\mathbf{G}_{x}}(z)g(\sigma(z))\d z=\int_{\partial\Omega}g\left(\frac{x-y}{|x-y|}\right)\mathcal{J}(x,y)\pi(\d y)$$
where
$$\mathcal{J}(x,y)=\frac{1}{\sqrt{\mathrm{det}(\mathfrak{P}(z))}}\ind_{\mathcal{O}_{x}}(z)$$
has to be expressed in terms of $x$ and $y$. Using Lemma \ref{lem:jaco}, one has
$$\mathcal{J}(x,y)=\left|\frac{(\sigma(z)\cdot n(\mathfrak{p}(z)))\sigma_{d}(z)}{\tau_{-}(x,\sigma(z))^{d-1}}\right|\ind_{\mathcal{O}_{x}}(z)$$
with, as mentioned, $\tau_{-}(x,\sigma(z))=|x-y|,$ $\mathfrak{p}(z)=y$ and $\sigma(z)=\frac{x-y}{|x-y|}$. Notice that 
$$\sigma_{d}(z)=\sigma(z)\cdot n(z)=\frac{(x-y)}{|x-y|}\cdot n(x)$$ 
which gives the desired expression \eqref{eq:Jxy} of $\mathcal{J}(x,y)$. Now, if $g(\sigma)=\varphi(\tau_{-}(x,\sigma))G(x-\tau_{-}(x,\sigma)\sigma)$, we get \eqref{eq:CofV}.\end{proof}
\begin{nb}\label{eq:sym} Noticing that $\mathcal{J}(x,y)=\ind_{\Gamma_{+}(x)}(y)\mathcal{J}(x,y)$, one sees that
$$\mathcal{J}(x,y)=\mathcal{J}(y,x) \qquad \forall (x,y) \in \partial\Omega\times\partial\Omega$$
\end{nb}

We end this Section with a useful technical result
\begin{lemme}\label{lem:convede}
 Assume that $\partial\Omega$ satisfies Assumption \ref{hypO}. Then
$$\lim_{\delta\to0^{+}}\sup_{y \in \partial\Omega}\int_{|x-y| \leq \delta}\mathcal{J}(x,y)\pi(\d x)=0.$$
\end{lemme}
\begin{proof} First notices that the straightforward estimate 
\begin{equation}\label{eq:boundJ}
\mathcal{J}(x,y) \leq |x-y|^{1-d}\end{equation}
is not strong enough to derive the result (see the subsequent Lemma \ref{lem:1} for more details on this point). We need to proceed in a different way. Observe that, thanks to Remark \ref{eq:sym}, for any $y \in \partial\Omega$ it holds
$$\int_{|x-y|\leq\delta}\mathcal{J}(x,y)\pi(\d x)=\int_{|x-y|\leq\delta}\mathcal{J}(y,x)\pi(\d x)=\int_{\S_{+}(y)}\mathbf{1}_{(0,\delta]}(\tau_{-}(y,\sigma))\d\sigma$$
where we used Lemma \ref{lem:ChVa} with the functions $\varphi(r)=\mathbf{1}_{[0,\delta]}(r)$ and $G\equiv 1$. Clearly, for any \emph{fixed} $y \in \partial\Omega$
\begin{equation}\label{eq:lim}
\lim_{\delta\to0^{+}}\int_{\S_{+}(y)}\mathbf{1}_{(0,\delta]}(\tau_{-}(y,\sigma))\d\sigma=0\end{equation}
according to the dominated convergence theorem so one needs to check that the convergence \eqref{eq:lim} is \emph{uniform} with respect to $y\in \partial\Omega$. Assume it is not the case so that there exist $c >0$, a sequence $\{y_{n}\}_{n} \subset\partial\Omega$ and a sequence $(\delta_{n})_{n} \subset (0,\infty)$ converging to $0$ such that
$$\int_{\S_{+}(y_{n})}\mathbf{1}_{(0,\delta_{n}]}(\tau_{-}(y_{n},\sigma))\d\sigma \geq c \qquad \forall n \in \N.$$
First, one deduces from Fatou's lemma that
\begin{multline}\label{fatou}
0 < c \leq \limsup_{n}\int_{\S^{d-1}}\ind_{\S_{+}(y_{n})}(\sigma)\ind_{(0,\delta_{n}]}(\tau_{-}(y_{n},\sigma))\d\sigma\\
 \leq \int_{\S^{d-1}}\limsup_{n}\ind_{\S_{+}(y_{n})}(\sigma)\ind_{(0,\delta_{n}]}(\tau_{-}(y_{n},\sigma))\d\sigma.\end{multline}
Of course, there is no loss of generality in assuming that $\{y_{n}\}_{n}$ converges to some $y \in \partial\Omega.$ Now, being $\partial\Omega$ of class $\mathcal{C}^{1}$, it holds that $\lim_{n}n(y_{n})=n(y)$ and therefore there is $n_{0}\in \N$ such that
$$\S_{+}(y) \subset \S_{+}(y_{n}) \qquad \forall n \geq n_{0}.$$
Moreover, for $\sigma \in \S_{+}(y)$, $\tau_{+}(y,\sigma) >0$ and, since $\tau_{+}$ is lower-semicontinuous on $\partial\Omega\times V$ (see \cite[Lemma 1.5]{voigt}), it holds that
$$\liminf_{n\to\infty}\tau_{-}(y_{n},\sigma) \geq \tau_{-}(y,\sigma) >0 \qquad \forall \sigma \in \S_{+}(y).$$
As a consequence, one has
$$\limsup_{n\to\infty}\ind_{(0,\delta]}(\tau_{-}(y_{n},\sigma))=0 \qquad \forall \sigma \in \S_{+}(y)$$
Since $\{\sigma \in \S^{d-1}\;;\;\sigma \cdot n(x)=0\}$ is a subset of $\S^{d-1}$ of zero Lebesgue measure, we see that
$$\limsup_{n}\ind_{\S_{+}(y_{n})}(\sigma)\ind_{(0,\delta]}(\tau_{-}(y_{n},\sigma))=0$$
for almost every $\sigma \in \S^{d-1}$ which contradicts \eqref{fatou}. This proves the result.
\end{proof}

Whenever the boundary $\partial\Omega$ is of class $\mathcal{C}^{2}$ -- which is not the framework adopted in this paper -- one can strengthen the estimates \eqref{eq:boundJ}. Namely, one has the following result (see \cite[Lemma 2]{guo03} for a similar result)
\begin{lemme}\label{lem:1}
Under the above assumption on $\Omega$ and if $\partial\Omega$ is of class $\mathcal{C}^{2}$ then, there exists a positive constant $C_{\Omega} >0$ such that
$$\left|(x-y)\,\cdot \,n(x)\right| \leq C_{\Omega}\,|x-y|^{2}, \qquad \forall x,y \in \partial \Omega.$$
Consequently, with the notations of Lemma \ref{lem:ChVa}, there is a positive constant $C >0$ such that
$$\mathcal{J}(x,y) \leq \frac{C}{|x-y|^{d-3}}, \qquad \forall x,y \in \partial \Omega, x\neq y.$$
In particular, for $d=2,3$, $\mathcal{J}(\cdot,\cdot)$ is bounded.  \end{lemme}
\begin{proof} The intuition behind the estimate is that, from the smoothness of $\partial \Omega$,  for any $x\neq y \in \partial\Omega$, if $\bm{e}_{x}(y)=\frac{x-y}{|x-y|}$ denotes the unit vector with direction $x-y$, then
$$\lim_{y \to x}\bm{e}_{x}(y)\,\cdot \,n(x)=0$$
since $\bm{e}_{x}(y)$ tends to be tangent to $\partial \Omega.$ Then $( x-y)\,\cdot \,n(x)$ is of the order $|x-y|^{2}$ for $x \simeq y$. Let us make this rigorous. For a given $x \in \partial \Omega$, one can find a local parametrization of a neighbourood $\mathcal{O}_{x} \subset \partial\Omega$ containing $x$ as 
$$\mathcal{O}_{x}=\left\{(u,\Phi(u))\;;\;u \in U\right\}$$
where $U$ is a open subset of $\R^{d-1}$ and $\Phi\,:\,U \to \mathcal{O}_{x}$ is a $\mathcal{C}^{2}$-diffeomorphism. Denoting by $|\cdot|$ the euclidian norm of $\R^{d-1}$ and by $\nabla \Phi$ the gradient of $\Phi$ (in $\R^{d-1}$) we get, with $x=(u_{0},\Phi(u_{0})) \in \partial \Omega$, 
$$n(x)=\frac{1}{\sqrt{1+|\nabla \Phi(u_{0})|^{2}}}\left(\nabla \Phi(u_{0}),-1\right)$$
so that
$$(x-y)\,\cdot n(x)=\frac{1}{\sqrt{1+|\nabla \Phi(u_{0})|^{2}}} \bigg(\langle u-u_{0}\,;\,\nabla \Phi(u_{0})\rangle_{d-1}-\left(\Phi(u)-\Phi(u_{0})\right)\bigg)$$
where $\langle \cdot,\cdot\rangle_{d-1}$ is the inner product in $\R^{d-1}.$ Then, using Taylor expansion, one sees that, for $u \simeq u_{0}$, 
$$\langle u-u_{0}\,;\,\nabla \Phi(u_{0})\rangle_{d-1}-\left(\Phi(u)-\Phi(u_{0}\right)=-\frac{1}{2}\langle \mathrm{Hess}\Phi(u_{0})(u-u_{0})\,;\,u-u_{0}\rangle_{d-1}+\mathrm{o}(|u-u_{0}|^{2})$$
where $\mathrm{Hess}\Phi(u_{0})$ is the Hessian matrix of $\Phi$ at $u_{0} \in U \subset \R^{d-1}.$ Then, since $\Phi$ is of class $\mathcal{C}^{2},$ by setting $M:=\|\mathrm{Hess}\Phi\|_{L^{\infty}(\overline{U})}$ (notice that, if $\partial \Omega$ is compact, then $U$ is relatively compact in $\R^{d-1}$) we get
$$|(x-y)\,\cdot n(x)| \leq M|u-u_{0}|^{2} \leq M|x-y|^{2}\qquad \text{ for } u \simeq u_{0}.$$
 Since $\partial \Omega$ is compact and $\Omega$ bounded, this easily yields the conclusion. Now, from \eqref{eq:Jxy}, we get
$$\mathcal{J}(x,y) \leq C_{\Omega}^{2}|x-y|^{3-d}  \qquad \forall y \in \Gamma_{+}(x)$$
which proves the boundedness of $\mathcal{J}$ for $d=2,3$ since $\mathcal{J}(x,y) \leq C_{\Omega}^{2}D^{3-d}$ in such a case. \end{proof}

\begin{nb} Applying Lemma \ref{lem:ChVa} with $\varphi=1$, one sees that in dimension $d=2,3$, the boundedness of $\mathcal{J}(x,y)$ implies  the existence of a positive constant $C >0$ such that
$$\int_{\S_{+}(x)}G(x-\tau_{-}(x,\sigma)\sigma)\,|\sigma \cdot n(x)|\d\sigma \leq C \int_{\Gamma_{+}(x)}G(y)\pi(\d y) \qquad \forall G \geq 0.$$
This easily allows to recover \cite[Lemma 2.3, Eq. (2.6)]{EGKM}. 
\end{nb}
\begin{nb} For $\partial\Omega$ of class $\mathcal{C}^{2}$, one can quantifies the convergence to $0$ obtained in Lemma \ref{lem:convede}. Indeed, since
$$\int_{|x-y| \leq \delta}\mathcal{J}(x,y)\pi(\d x) \leq \delta\int_{|x-y|\leq \delta}\frac{\pi(\d x)}{|x-y|^{d-2}} \qquad \forall y \in \partial\Omega$$
and because
$$\lim_{\delta\to0^{+}}\sup_{y \in \partial\Omega}\int_{|x-y|\leq \delta}\frac{\pi(\d x)}{|x-y|^{d-2}}=0$$
by a classical argument, one obtains that
$$\sup_{y \in \partial\Omega}\int_{|x-y|\leq \delta}\mathcal{J}(x,y)\pi(\d x)=o(\delta) \quad \text{ as } \delta \to 0^{+}.$$
\end{nb}


\begin{thebibliography}{GVdMP}

\bibitem{aoki}
\textsc{K. Aoki, F. Golse,} On the speed of approach to equilibrium for a collisionless gas, \textit{Kinet. Relat. Models} {\bf 4} (2011), 87--107.

 

\bibitem{arlotti86} 
\textsc{L. Arlotti}, Boundary conditions for
streaming operator in a bounded convex body, \textit{Transp. Theory
Stat. Phys.} {\bf 15} 959--972, 1986.
 

\bibitem{AL05}
\textsc{L. Arlotti, B. Lods}, Substochastic semigroups for transport equations with conservative boundary conditions, \textit{J. Evol. Equations}, {\bf 5} (2005) 485--508.

\bibitem{mjm1}
\textsc{L. Arlotti, J. Banasiak, B. Lods}, A new approach to transport equations associated to a regular field: trace results and well--posedness, \textit{Mediterr. J. Math.}, {\bf 6} (2009), 367--402.


\bibitem{bernou1}
\textsc{A. Bernou}, A semigroup approach to the convergence rate of a collisionless gas, \texttt{https://arxiv.org/abs/1911.03228}, 2019.

\bibitem{bernou2}
\textsc{A. Bernou, N. Fournier}, A coupling approach for the convergence to equilibrium for a collisionless gas, \texttt{https://arxiv.org/abs/1910.02739}, 2019.




\bibitem{brezis}
\textsc{H. Br\'ezis}, \textbf{Functional analysis, Sobolev spaces and partial differential equations,} Universitext. Springer, New York, 2011.


 

\bibitem{briant}
\textsc{M. Briant, Y. Guo,} Asymptotic stability of the Boltzmann equation with Maxwell boundary conditions, \textit{J. Differential Equations} {\bf 261} (2016), 7000--7079.


\bibitem{ces1}
\textsc{M. Cessenat}, Th\'eor\`emes de traces $L_p$ pour les espaces
de fonctions de la neutronique. \textit{C. R. Acad. Sci. Paris.},
Ser. I {\bf 299} 831--834, 1984.
\bibitem{ces2}
\textsc{M. Cessenat}, Th\'eor\`emes de traces pour les espaces de
fonctions de la neutronique. \textit{C. R. Acad. Sci. Paris.}, Ser.
I {\bf 300} 89--92, 1985.

 \bibitem{chacon}
\textsc{R. V. Chacon, U. Krengel}, Linear modulus of linear operator, 
\textit{Proc. Amer. Math. Soc.}, {\bf 15} (1964), 553--559. 


\bibitem{chill}
\textsc{R. Chill, D. Seifert,} Quantified versions of Ingham's theorem, \textit{Bull. Lond. Math. Soc.} {\bf 48} (2016), 519--532.




\bibitem{EGKM}
\textsc{R. Esposito, Y. Guo, C. Kim, R. Marra,} Non-isothermal boundary in the Boltzmann theory and Fourier law,
\textit{Comm. Math. Phys.}, {\bf 323} (2013) 177--239.

\bibitem{guo03}
\textsc{Y. Guo,} Decay and continuity of the Boltzmann equation in bounded domains, \textit{Arch. Ration. Mech. Anal.}, {\bf 197} (2010) 713--809. 


\bibitem{kato}
{\sc T. Kato}, \textbf{Perturbation theory for linear operators}, Classics in Mathematics, Springer Verlag, 1980.

\bibitem{kuo}
\textsc{H. W. Kuo,} Equilibrating effect of Maxwell-type boundary condition in highly rarefied gas, \textit{J. Stat. Phys.} {\bf 161} (2015), 743--800.


\bibitem{liu}
\textsc{H. W. Kuo, T. P. Liu, L. C. Tsai,} Free molecular flow with boundary effect, \textit{Comm. Math. Phys.} {\bf 318}
(2013), 375--409.


\bibitem{m2as}
\textsc{B. Lods}, On linear kinetic equations involving unbounded cross-sections. \textit{Math. Methods Appl. Sci.}, {\bf 27} (2004) 1049--1075.

\bibitem{LM-iso}
\textsc{B. Lods, M. Mokhtar-Kharroubi,} Exponential convergence for collisionless kinetic equation with diffuse boundary conditions and non zero velocities, preprint, 2020.


\bibitem{LMR}
\textsc{B. Lods, M. Mokhtar-Kharroubi, R. Rudnicki}, Invariant density and time asymptotics for collisionless kinetic equations with partly diffuse boundary operators, \textit{Ann. I. H. Poincar\'e -- AN}, to appear, 2020.


\bibitem{marek}
\textsc{I. Marek,} Frobenius theory of positive operators: Comparison theorems and applications, \textit{SIAM J. Appl. Math.} {\bf 19} (1970), 607--628. 

\bibitem{mmkop}
\textsc{M. Mokhtar-Kharroubi,} Optimal spectral theory of the linear Boltzmann equation, \textit{J. Funct. Anal.} {\bf 226} (2005), 21--47. 

\bibitem{MKR}
\textsc{M. Mokhtar-Kharroubi, R. Rudnicki}, 
On asymptotic stability and sweeping of collisionless kinetic equations. \textit{Acta Appl. Math.} \textbf{147} (2017), 19--38. 

 

\bibitem{mmkseifert}
\textsc{M. Mokhtar-Kharroubi, D. Seifert}, Rates of convergence to equilibrium for collisionless kinetic equations in slab geometry, \textit{J. Funct. Anal.} {\bf 275} (2018),  2404--2452. 


 
\bibitem{stroock}
\textsc{D. W. Stroock}, \textit{\textbf{Essentials of integration theory for analysis}}, Springer, 2011.


\bibitem{aoki1}
\textsc{T. Tsuji, K. Aoki, F Golse,} Relaxation of a free-molecular gas to equilibrium caused by interaction with vessel wall,
\textit{J. Stat. Phys.} {\bf 140} (2010), 518--543.
 
\bibitem{voigt84}
\textsc{J. Voigt}, Positivity in time dependent linear transport theory,
\textit{Acta Applicandae Mathematicae} \textbf{2} (1984) 311--331.
\bibitem{voigt}
\textsc{J. Voigt}, 	{\textbf{Functional analytic treatment of the
initial boundary value problem for collisionless gases}},
Habilitationsschrift, M\"unchen, 1981.
\end{thebibliography}
\end{document}